\pdfoutput=1
\documentclass[11pt]{article}
\usepackage{authblk}
\usepackage[toc,page]{appendix}

\usepackage{hyperref}
\hypersetup{
    colorlinks=true,
    linkcolor=blue,
    citecolor=red,
    urlcolor=blue,
    pdfborder={0 0 0}
}

\usepackage{color}
\usepackage{helvet}         
\usepackage{courier}        
\usepackage{type1cm}        
\usepackage{framed}
\usepackage{tikz}
\usepackage{makeidx}         
\usepackage{graphicx}        
\usepackage{multicol}        
\usepackage[bottom]{footmisc}
\usepackage{amsmath}
\usepackage{tabularx}
\usepackage{amssymb}
\usepackage{bbold}
\usepackage{amsthm}
\usepackage{subcaption}
\usepackage{sidecap}
\usepackage{floatrow}
\usepackage{pdflscape}
\usepackage{bm}
\usepackage{comment}
\usepackage[font=small]{caption}
\usepackage[top=2cm, bottom=2cm, left=2cm, right=2cm]{geometry}
\newtheorem{theorem}{Theorem}
\newtheorem{corollary}[theorem]{Corollary}

\newtheorem{lemma}[theorem]{Lemma}

\newtheorem{proposition}[theorem]{Proposition}

\newtheorem{remark}[theorem]{Remark}

\numberwithin{theorem}{section}
\numberwithin{figure}{section}
\numberwithin{equation}{section}

\DeclareMathOperator{\dist}{dist}
\DeclareMathOperator{\SLE}{SLE}
\DeclareMathOperator{\hSLE}{hSLE}

\DeclareMathOperator{\GFF}{GFF}

\DeclareMathOperator{\diam}{diam}

\DeclareMathOperator{\ust}{UST}

\def\cZ{\mathcal{Z}}

\def\cU{\mathcal{U}}
\def\cT{\mathcal{T}}

\def\cR{\mathcal{R}}
\def\cQ{\mathcal{Q}}

\def\cN{\mathcal{N}}

\def\cH{\mathcal{H}}
\def\cG{\mathcal{G}}

\def\cD{\mathcal{D}}

\def\cA{\mathcal{A}}

\def\ph{\varphi}

\begin{document}

\title{Piecewise Temperleyan dimers and a multiple SLE$_8$}

 \author{Nathana\"el Berestycki\thanks{University of Vienna, \texttt{nathanael.berestycki@univie.ac.at}} \ and Mingchang Liu\thanks{Tsinghua University, \texttt{liumc18@mails.tsinghua.edu.cn} }}

\date{\today}

\maketitle

\abstract{
We consider the dimer model on piecewise Temperleyan, simply connected domains, on families of graphs which include the square lattice as well as superposition graphs. 
We focus on the spanning tree $\cT_\delta$ associated to this model via Temperley's bijection, which turns out to be a Uniform Spanning Tree with singular alternating boundary conditions. Generalising the work of the second author with Peltola and Wu \cite{LiuPeltolaWuUST} we obtain a scaling limit result for $\cT_\delta$. For instance, in the simplest nontrivial case, the limit of $\cT_\delta$ is described by a pair of trees whose Peano curves are shown to converge jointly to a multiple SLE$_8$ pair. The interface between the trees is shown to be given by  an SLE$_2(-1, \ldots, -1)$ curve.   
More generally we provide an equivalent description of the scaling limit in terms of imaginary geometry. This allows us to make use of the results developed by the first author and Laslier and Ray \cite{BLRdimers}. We deduce that, universally across these classes of graphs, the corresponding height function converges to a multiple of the Gaussian free field with boundary conditions that jump at each non-Temperleyan corner. After centering, this generalises a result of Russkikh \cite{RusskikhDimers} who proved it in the case of the square lattice. Along the way, we obtain results of independent interest on chordal hypergeometric SLE$_8$; for instance we show its law is equal to that of an SLE$_8 (\bar \rho)$ for a certain vector of force points, conditional on its hitting distribution on a specified boundary arc.
}


%

\tableofcontents

\newcommand{\eps}{\epsilon}
\newcommand{\ov}{\overline}
\newcommand{\U}{\mathbb{U}}
\newcommand{\T}{\mathbb{T}}
\newcommand{\HH}{\mathbb{H}}
\newcommand{\LA}{\mathcal{A}}
\newcommand{\LB}{\mathcal{B}}
\newcommand{\LC}{\mathcal{C}}
\newcommand{\LD}{\mathcal{D}}
\newcommand{\LF}{\mathcal{F}}
\newcommand{\LK}{\mathcal{K}}
\newcommand{\LE}{\mathcal{E}}
\newcommand{\LG}{\mathcal{G}}
\newcommand{\LL}{\mathcal{L}}
\newcommand{\LM}{\mathcal{M}}
\newcommand{\LQ}{\mathcal{Q}}
\newcommand{\LP}{\mathcal{P}}
\newcommand{\LR}{\mathcal{R}}
\newcommand{\LT}{\mathcal{T}}
\newcommand{\LS}{\mathcal{S}}
\newcommand{\LU}{\mathcal{U}}
\newcommand{\LV}{\mathcal{V}}
\newcommand{\LX}{\mathcal{X}}
\newcommand{\PartF}{\mathcal{Z}}
\newcommand{\LH}{\mathcal{H}}
\newcommand{\R}{\mathbb{R}}

\newcommand{\C}{\mathbb{C}}
\newcommand{\N}{\mathbb{N}}
\newcommand{\Z}{\mathbb{Z}}
\newcommand{\E}{\mathbb{E}}
\newcommand{\PP}{\mathbb{P}}
\newcommand{\QQ}{\mathbb{Q}}
\newcommand{\A}{\mathbb{A}}
\newcommand{\one}{\mathbb{1}}
\newcommand{\bn}{\mathbf{n}}
\newcommand{\MR}{MR}
\newcommand{\cond}{\,|\,}
\newcommand{\la}{\langle}
\newcommand{\ra}{\rangle}
\newcommand{\tree}{\Upsilon}
\newcommand{\prob}{\mathbb{P}}
\renewcommand{\Im}{\mathrm{Im}}
\renewcommand{\Re}{\mathrm{Re}}
\newcommand{\ii}{\mathfrak{i}}

\def\d{\delta}
\def\Gd{G^\delta}

\section{Introduction}
\label{sec::intro}

The dimer model is one of the simplest but also most intriguing models of statistical mechanics. Introduced in the pioneering work of Temperley and Fisher \cite{TemperleyFIsher}
and, independently and nearly simultaneously, Kasteleyn \cite{Kasteleyn}, in the 1960s, the model can be defined on any finite weighted graph $G = (V, E)$ admitting a perfect matching, and is the probability measure on perfect matchings $\mathbf{m}$ of $G$ given by 
\begin{equation}\label{E:gibbs}
\mu ( \textbf{m}) = \frac1Z \prod_{e\in \mathbf{m}} w_e
\end{equation}
where $(w_e)_{e\in E}$ denote the edge weights of $G$. The dimer model is particularly well studied in the case where $G$ is planar and bipartite, in which case it is in some sense exactly solvable. Nevertheless, despite remarkable progress over the last sixty years, the dimer model on large but finite graphs (say, on subdomains of the plane) remains notoriously difficult to study due to its extreme sensitivity to the microscopic details of the boundary conditions.

The dimer model is typically studied through its \textbf{height function}, $h: F \to \R$ introduced by Thurston and defined on the set of faces $F$ of the graph. The height function's definition depends on the choice of a reference dimer configuration $\mathbf{m}_0$ or more generally reference flow, and even then is defined only up to constants. However, the centered height function $h - \E(h)$ is canonically associated with the dimer model. See \cite{KeyonDimerLecture} for references and relevant definitions. This height function turns the dimer model into a model of random surfaces and the main question concern its large scale behaviour. A remarkable conjecture of Kenyon and Okunkov predicts that the large scale behaviour is in great generality described by the Gaussian free field, albeit after a possibly nontrivial (and not conformal) change of coordinates. This conjecture was proved in the case of Temperleyan boundary conditions in the landmark paper of Kenyon \cite{Kenyon_ci}, \cite{Kenyon_GFF} and in a handful of other cases \cite{RusskikhDimers, BLRdimers, BLR_Riemann1, Petrov}; see \cite{Gorin_book} and references therein for an introduction in the case of lozenge tilings. 

The Temperleyan boundary conditions introduced and studied by Kenyon in \cite{Kenyon_ci} are perhaps the most canonical from the combinatorial point of view; they appear naturally both in the methods involving the discrete holomorphicity of the inverse Kasteleyn matrix, and for methods involving the Temperley bijection as studied say in \cite{BLRdimers, BLR_Riemann1, BLR_Riemann2} and which will be relevant to this article. It is worth noting that they are in some sense \emph{trivial} from the point of view of the phenomenology of the dimer model, since they exclude the so-called frozen and gaseous phases, and lead to GFF fluctuations without change of coordinates.

In this paper, we will consider the dimer model in $2n$-black-piecewise Temperleyan domains on the square lattice (or simply $2n$-piecewise Temperleyan in the following), which we will introduce precisely in Section~\ref{subsec::Temperleyan}. These boundary conditions, first considered in the work of Russkikh \cite{RusskikhDimers}, are perhaps the simplest way to relax the Temperleyan condition on the boundary. 
By definition, in such a domain there is a total of $2n$ white corners for some $n\ge 2$, of which necessarily $n-1$ are concave, and $n+1$ convex (see Figure \ref{fig::temperleyan} for an illustration). 
We will make an additional assumption on these domains: between two consecutive concave white corners of $\Omega$ (i.e., not separated by intermediate white corners), the boundary black squares belong to a fixed type of black vertices, which we take to be
 $\blacksquare_0$. 
 
 This can always be assumed without loss of generality if $n =2, 3$ but not if $n\ge 4$ (in that case there exist examples of $2n$-piecewise Temperleyan domains which do not satisfy this additional condition). 
 In fact our results will be valid much more generally for graphs obtained by planar superposition satisfying very mild assumptions, see Section \ref{SS:gen} for a precise description. For such graphs we will formulate a natural analogue of the boundary conditions described above (piecewise Temperleyan boundary together with the additional assumption).     
Our main result, valid both for the square lattice and more generally for superposition graphs defined in Section \ref{SS:gen}, is as follows.

\begin{theorem}\label{thm::heightconv}
Suppose $\Omega$ is a simply connected domain with $C^1$ simple boundary and suppose that $\{\Omega_\delta\}_{\delta>0}$ is a sequence of $2n$-black-piecewise Temperleyan domains satisfying the above assumption, which converges to $\Omega$. We denote by $h_\delta$ the height function of the dimer model on $\Omega_\delta$ and denote by $h$ the Gaussian free field on $\Omega$ with Dirichlet boundary conditions. Define $\chi:=\frac{1}{\sqrt 2}$. Then, we have for every $f\in C_c^\infty(\Omega)$, and $k\ge 1$, we have
\[\E\left[\left|(h_\delta-\E[h_\delta],f)-\frac{1}{\chi}(h,f)\right|^k\right]\to 0.\]
\end{theorem}

See \eqref{eqn::converge1} for the precise notion of convergence of $\Omega_\delta$ to $\Omega$. In the case of the square lattice, this result had already been proved by Russkikh \cite{RusskikhDimers}. Russkikh's proof follows the method developed in the groundbreaking work of Kenyon \cite{Kenyon_ci, Kenyon_GFF} which relies on the analysis of the inverse Kaseteleyn matrix. In the Temperleyan setting, Kenyon had shown that the inverse Kasteleyn matrix converges in a suitable to the holomorphic function associated with the derivative of the Green function with Dirichlet boundary conditions. These eventually lead Dirichlet boundary conditions for the field. By contrast, in the piecewise Temperleyan case analysed by Russkikh, the inverse Kasteleyn matrix corresponds to a Green function with \textbf{mixed Neumann--Dirichlet} boundary conditions. Yet surprisingly, as a result of fairly striking cancellations, she demonstrated that the limiting centered height field has purely Dirichlet boundary conditions, as is the case in Theorem \ref{thm::heightconv}. We shed some light on this phenomenon, showing in a precise geometric way how the presence of non-Temperleyan corners influences the limiting geometry of the model. 

This geometric phenomenon can in fact also be formulated purely in terms of height functions. This relies on the connection to \textbf{imaginary geometry} initiated in \cite{BLRdimers}, and which we now begin to discuss.
As mentioned above, $h_\delta - \E[h_\delta]$ does not depend on the choice of a reference flow. There is however \emph{one} particular choice of reference flow which turns out to be quite canonical, and which corresponds to the normalised winding of the edges in the graph (see e.g. \cite{dimer_tree, BLRtgraph} or \cite{BLR_Riemann1} for definitions). 
For this choice, it will turn out that $\E[h_\delta]$ converges to a harmonic function $u$ (which will be defined in Section~\ref{subsec::GFF}) locally uniformly. This harmonic function will give the boundary values of the Gaussian free field. In contrast to the purely Temperleyan case, these boundary values are now nonconstant, and have jumps at each non-Temperleyan corner. Thus the effect of the non-Temperleyan boundary conditions is to create jumps at the corner (these jumps can already be seen to occur at the discrete level). Of course, when we subtract the expectation, the corresponding field becomes purely Dirichlet, and this cancels the effect of the non-Temperleyan corners. This explains why the limiting field in Russkikh's work \cite{RusskikhDimers} is the Dirichlet Gaussian free field, just as in the purely Temperleyan case. 

\paragraph{Temperley's bijection.} To formulate this geometrically, we make use of {Temperley's bijection} which associates to any dimer configuration a spanning tree on a related graph. Let $\cT_\delta$ denote the random tree associated to the dimer configuration $\mathbf{m}$ of \eqref{E:gibbs} in a piecewise-Temperleyan domain (to be precise, $\cT_\delta$ is the $\blacksquare_0$-tree; Temperley's bijection also associates to $\mathbf{m}$ a tree living on the $\blacksquare_1$-lattice, but this is simply the planar dual of $\cT_\delta$). Let $u$ denote the harmonic function from $\Omega$ which jumps alternatively by $\pm \lambda$ at each non-Temperleyan corner of $\partial \Omega$ (see Figure \ref{fig::general domain} for an illustration): that is, 
let us write the $2n$ white corners as $x_1, \ldots, x_{2n}$ as we go along the boundary in counterclockwise order, with both $x_{1}$ and $x_{2n}$ assumed to be convex (recall there are two more convex corners than concave). Assume first that $\Omega = \HH$.
 We define $u = u_\HH$ on the real line to be $+\lambda$ on $(-\infty, x_1)$ and then jump by $-\lambda$ on each convex corner, and $+\lambda$ at each concave corner. Note that $u$ will be equal to $-\lambda$ on $(x_{2n}, \infty)$. Then $u_\HH$ is simply the harmonic extension of this boundary data. When $\Omega$ is not the upper-half plane, $u$ is obtained by applying the conformal change of coordinates of imaginary geometry associated with $\kappa =2$, that is, 
$$
u_\Omega = u_\HH \circ \ph^{-1} -\chi\arg (\ph^{-1})'; 
$$
 where $\ph: \HH \to \Omega$ is a conformal mapping and $\chi = 2/ \sqrt{2} - \sqrt{2}/2 = 1/\sqrt{2}$.  
\begin{theorem}\label{T:IG}
As the mesh size $\delta \to 0$, the Temperleyan tree $\cT_\delta$ converges to the tree consisting of the flow lines for $\kappa =2$ of the field $h_\Omega+u_\Omega$, where $h_\Omega$ is a $\GFF$ with Dirichlet boundary conditions on $\Omega$.  
\end{theorem}


As announced above, note that the effect of the white corners is thus to change the boundary data of the limiting field (but, as already mentioned, this disappears when we take away the expectation). We also obtain a more direct, geometric description of the scaling limit of $\cT_\delta$ in the particular case when $n =2$, which is the simplest non trivial case. In that case there are four white corners which we denote respectively by $a, b, c, d$. Consistent with our earlier notation, the single white concave corner is denoted by $b$. It can be shown that, deterministically, the arcs $(ab), \ldots, (da)$ of the topological quadrilateral $(\Omega, a, b,c,d)$ alternate between Dirichlet and Neumann portions for $\cT_\delta$ in a precise sense. Furthermore, $\cT_\delta$ contains a unique curve, which we call $\gamma_\delta$, starting from $b$ and ending in a random location $z$ on the arc $(cd)$. This curves splits $\cT_\delta$ into two trees, which we call $\cT_\delta^L, \cT_\delta^R$ respectively (left and right correspond to the planar orientation inherited from $\gamma_\delta$ viewed as a curve oriented from $b$ to $(cd)$). To this pair of trees we can associate a pair of Peano curves, which we call $\eta_\delta^L$ and $\eta_\delta^R$ respectively. These curves emanate from $d$ and $c$ respectively, and terminate in $a$ and $b$ respectively. See Figure \ref{F:mult} for an illustration. We now state a scaling limit result describing the limiting law of these random curves. 
\begin{figure}
\begin{center}
\includegraphics[scale=.5]{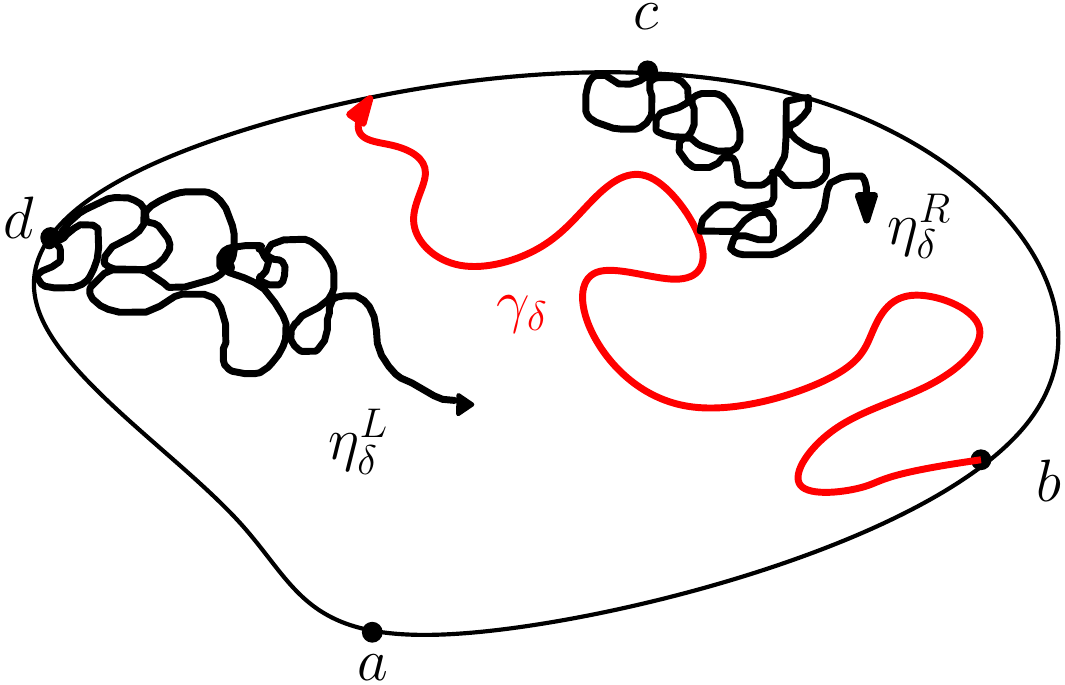}
\caption{Setup of Theorem \ref{T:multi}. In the limit, we get a multiple SLE$_8$ separated by a chordal SLE$_2 (-1, \ldots, -1)$ interface.}
\label{F:mult}
\end{center}
\end{figure}

\begin{theorem}\label{T:multi}
As $\delta \to 0$, the random curves $\gamma_\delta, \eta^L_\delta, \eta^R_\delta$ converge jointly to a triple $(\gamma, \eta^L, \eta^R)$. The interface $\gamma$ between the two trees has the law of a chordal SLE$_2 (-1, -1, -1, -1)$ from $b$, stopped when hitting the arc $(cd)$ (call $z$ this hitting point). 
Given $\gamma$, $\eta^R$ and $\eta^L$ are independent standard chordal SLE$_8$ in their respective domains.
Hence the joint law of $(\eta^L, \eta^R)$ is a \textbf{multiple SLE$_8$} in the sense that given $\eta^L$, the curve $\eta^R$ is simply a standard SLE$_8$ in $\Omega \setminus \eta^L$, from $c$ to $b$ (and the same holds for $\eta^L$ given $\eta^R$). 
\end{theorem}

\paragraph{Marginal laws.} The marginal laws of $\eta^L$ and $\eta^R$ can also be described: $\eta^L$ has the law of a chordal SLE$_8(2, - 2)$ in $\Omega$ from $d$ to $a$ with force points at $c$ (corresponding to weight $\rho = 2$) and $b$ (corresponding to weight $\rho = -2$). Similarly, $\eta^R$ has the law of a chordal SLE$_8(2, 2)$ in $\Omega$ from $c$ to $b$, with force points at $d$ and $a$. Note that in this theorem, the target point of $\gamma$ needs not be specified, as given the value of the force points (four force points of weight $\rho_i = -1$, summing up to a total weight of $\sum_i \rho_i = -4 = \kappa - 6$), the curve has the locality property and is hence target independent.

\paragraph{Multiple SLE.} A joint law on curves $(\eta_1, \eta_2)$ is called multiple (say chordal) SLE$_\kappa$ in a domain $D$, if the conditional law of $\eta_i$ given $\eta_{2-i}$ is a chordal SLE$_\kappa$ in the complementary domain. When $\kappa\le 4$ such laws are known to be unique by a result of Beffara, Peltola and Wu \cite{BeffaraPeltolaWuUniqueness}. This is however not known when $\kappa >4$. In fact, for $\kappa \ge 8$ there can be no uniqueness as this is the space-filling regime: any choice of a law for the interface $\gamma$ separating the two curves gives rise to a pair $(\eta_1, \eta_2)$ of multiple SLEs in the sense of this definition.  We do not know how to canonically characterise the law of the multiple SLE$_8$ $(\eta^L, \eta^R)$ appearing in the theorem, except by specifying the marginals as we have done.


\paragraph{Connection to hypergeometric SLE$_8$.} It is useful at this point to make a comparison with the paper \cite{HanLiuWuUST}, in which Uniform Spanning Trees with related alternating boundary conditions were also studied. To explain the model they considered, let $(\Omega, a,b,c,d)$ denote a topological rectangle. The boundary conditions divide the boundary into four arcs: $(ab), (bc), (cd), (da)$. They impose Neumann (also known as reflecting or free) boundary conditions on the arcs $(bc)$ and $(da)$, while the arcs $(ab)$ and $(cd)$ are each wired -- but \emph{not} wired to each other. They then consider a uniform spanning tree $\tilde \cT_\delta$  with these boundary conditions, if the mesh size is $\delta>0$. Note that since $(a b)$ and $(c d)$ are not wired with one another, this forces the presence of a path  (let us call it $\tilde \gamma_\delta$) connecting $(a b)$ to $(c d)$. The Temperleyan tree $\cT_\delta$ associated to a piecewise Temperleyan dimer model in the case $n=2$ can therefore be viewed as an asymptotically degenerate version of $\tilde \cT_\delta$ in which the interface $\tilde \gamma_\delta$ is conditioned to start in $b_\delta$. One of the main results of \cite{HanLiuWuUST} is that the pair of discrete Peano curves $(\tilde \eta^L_\delta, \tilde \eta^R_\delta)$ on either side of $\tilde \gamma_\delta$ in $\tilde \cT_\delta$ converges to a scaling limit given by a multiple SLE$_8$ pair $(\tilde \eta^L, \tilde \eta^R)$, but whose marginal law is the so-called \emph{hypergeometric} SLE$_8$ with parameter $\nu = 0$ (whose driving function will be described later in \eqref{E:hSLE})\footnote{We warn the reader that two related notions, both bearing the name of hypergeometric SLE and the notation hSLE, have appeared almost simultaneously and independently in two different papers by two different authors, namely Wei Qian in \cite{WeiTrichordal} and Hao Wu in \cite{WuHyperSLE}. 
In both cases the driving function is given by the same formula involving a hypergeometric function. 
To describe the curve unambiguously, one must additionally specify the value of the three parameters $a,b $ and $c$ entering the hypergeometric function. In \cite{WeiTrichordal} this choice is only explicit for the value $\kappa = 8/3$, which is the main focus of that paper. \\
In addition, the hypergeometric SLE$_\kappa$ with parameter $\nu$ (in the notations of \cite{WuHyperSLE}) coincides with the so-called \emph{intermediate} SLE$_\kappa$ introduced earlier by Dapeng Zhan in \cite[Section 3]{Zhan}, with the minor difference that there the range of values of $\kappa$ was restricted to $\kappa \in (0,4)$ and $\nu \ge \kappa/2-2$ (in fact, when $\nu \le \kappa/2-6$, the choice of parameters in the definition of hypergeometric SLE$_\kappa(\nu)$ considered in \cite{WuHyperSLE} is different from the case where $\nu \ge \kappa/2 -6$; thus the corresponding hypergeometric SLE$_\kappa(\nu)$ is in fact different from the intermediate SLEs considered in \cite{Zhan} in that case). 
%
For the avoidance of doubt, in this paper we are concerned with the case $\kappa =8$, and we will use the word ``hypergeometric SLE$_8$'' and the notation hSLE$_8$ to refer to the hypergeometric SLE$_8$ with parameter $\nu = 0$ in the notations of \cite{WuHyperSLE}. Once again, this corresponds to Zhan's intermediate SLE generalised to $\kappa = 8$, and with the parameter $\rho$ from \cite{Zhan} taken to be $ \rho = 2$. We thank Wei Qian and Hao Wu for discussions regarding this.\label{footnote}}. 
Combining our results with those of \cite{HanLiuWuUST}, we are therefore able to obtain some simple descriptions of this hypergeometric SLE$_8$ once we condition on the hitting point of the opposite arc. For instance, consider the curve $\tilde\eta^R$ from $c$ to $b$ and condition on the hitting position of the arc $(ab)$. Despite the apparent complexity of the driving function, once we condition on the hitting point of the arc $(ab)$ by the curve, then hSLE$_8$ reduces to a more standard SLE$_8(\bar \rho)$ where the weight vector $\bar \rho$ can be explicitly described.


\begin{theorem}\label{T:condSLE8}
Suppose $\tilde\eta^R$ has the law of the hypergeometric $\SLE_8$ in $(\Omega;a,b,c,d)$ from $c$ to $b$ with force points $d$, $a$. We denote by $T$ the hitting time by $\tilde\eta^R$ of $(ab)$. Then, given $\eta^R(T) = z$, the conditional distribution of $\tilde\eta^R([0,T])$ equals the law a chordal $\SLE_8(2,2,-4)$ up to time $T$ in $\Omega$ from $c$ to $b$, with marked points $d$, $a$ and $z$ respectively. Furthermore, conditionally on $\tilde\eta^R(T) = z$ and on $\tilde\eta^R([0, T])$,  $\eta^R$ evolves after time $T$ as a chordal SLE$_8$ from $z$ to $b$ in the domain $\Omega \setminus \tilde\eta^R([0, T])$. 
\end{theorem}  

Given the above result, it is natural to ask if other hypergeometric SLEs (still in the sense of \cite{WuHyperSLE}) can be described in terms of SLE$_\kappa (\rho)$ after conditioning. This will be addressed in future work, see \cite{Liu_reverse}; see also the fourth bullet point of Remark~\ref{rem::hyper} for additional discussion. We conclude by noting that Qian \cite{WeiTrichordal} was able to relate the law of the hypergeometric SLE  to tilted SLE$_\kappa(\rho)$ processes via a Girsanov change of measure, see Section 4.3.2 of \cite{WeiTrichordal}. Nevertheless Theorem \ref{T:condSLE8} is new to the best of our knowledge; for instance the numbers of force points here and in Section 4.3.2 of \cite{WeiTrichordal} are not identical. 

\paragraph{Acknowledgements.} Nathana\"el Berestycki’s research is supported by FWF grant P33083, “Scaling limits in random conformal geometry”. This work took place while Mingchang Liu was visiting the University of Vienna, whose hospitality is gratefully acknowledged. This visit was made possible thanks to the support of the State Scholarship Fund No.202106210235 from the Chinese government. 
N.B. also thanks Beno\^it Laslier and Marianna Russkikh for many useful discussions which took place prior to this work.



\section{Preliminaries}
\label{sec::preliminaries}

\subsection{Piecewise Temperleyan domain and Temperley's bijection}
\label{subsec::Temperleyan}

\label{SS:bij}
In this subsection, we will recall the definition of black-piecewise Temperleyan domain and establish the corresponding Temperleyan's bijection between dimer configurations and spanning trees.

First of all, let us recall the definition of black-piecewise Temperleyan domain, which was introduced in~\cite{RusskikhDimers}. Suppose $\Omega\subset\C$ is a simply connected domain, whose boundary $\partial\Omega$ is a simple curve on $\Z^2$. $\Omega$ can be divided into squares of size one which inherit from $\Z^2$ the checkboard black-white colouring. It is further useful to subdivide the black squares into two distinct sublattices: more precisely, denote by $\blacksquare_0$ the black squares on even rows and denote by $\blacksquare_1$ the black color on odd rows. We call $\Omega$ is a $2n$-black-piecewise Temperleyan domain, if it has $2n$ white corners. It can then be shown that of those, necessarily  $n+1$ are convex and $n-1$ are concave corners. If $n \ge 3$, we assume that between two consecutive concave white corners of $\Omega$ (i.e., not separated by intermediate white corners), the boundary black squares belong to 
 $\blacksquare_0$.  See the left hand side of Figure~\ref{fig::temperleyan} for an example of black-piecewise Temperleyan domain in the case $n=2$. When $n \ge 3$, it is possible for two consecutive white corners to be concave, but the above assumption imposes restrictions on the combinatorics, as we will see below.

\begin{figure}[ht!]
\begin{center}
\includegraphics[width=1\textwidth]{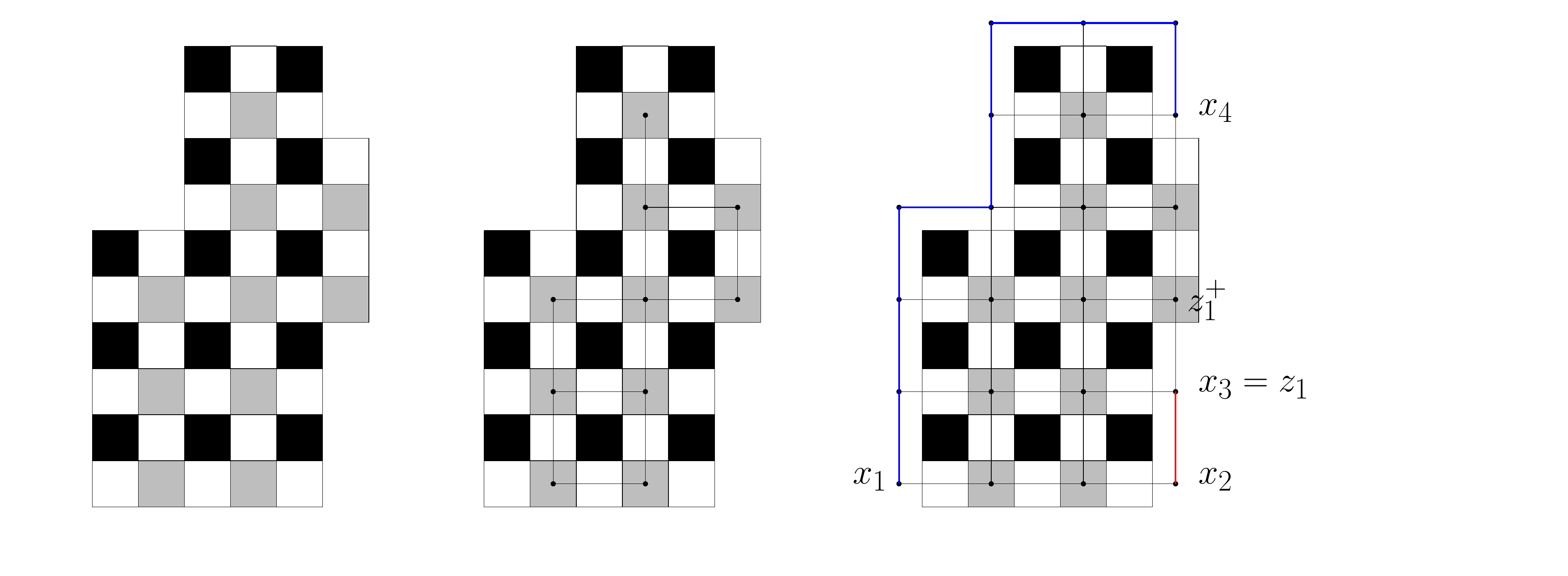}
\end{center}
\caption{\label{fig::temperleyan}On the left is a $4$-black-piecewise Temperleyan domain. The black squares belong to type $\blacksquare = \blacksquare_1$ and the grey squares belong to type ${\textcolor{gray}{\blacksquare}} = \blacksquare_0$. The middle picture shows the graph $\Omega_0^*$. The right picture shows $\Omega_0^*$ and its outer boundary, which result in the graph $\overline\Omega_0^*$. The boundary $\overline\Omega_0^*\setminus\Omega_0^*$ is shown in colour and consists of two components, coloured respectively red and blue.}
\end{figure}

Now, we introduce the Temperleyan's bijection when $\Omega$ is a $2n$-black-piecewise Temperleyan domain.
Define $\Omega^*$ to be the dual graph of $\Omega$. The vertices of $\Omega^*$ are the centers of the squares of $\Omega$ and the edges are straight lines connecting nearest pairs of vertices. We then say that a vertex of $\Omega^*$ is $\blacksquare_0$, 
$\blacksquare_1$ or white, depending on the colour of the corresponding square in $\Omega$. Define $\Omega_0^*$ to be the graph whose vertices are the $\blacksquare_0$ vertices of $\Omega^*$ and whose edges are straight lines connecting nearest pairs of vertices. Define $\overline\Omega^*_0$ to be the graph whose vertices are the union of $\Omega_0^*$ and $\blacksquare_0$ vertices adjacent to $\Omega_0^*$ and whose edges are straight lines connecting nearest pairs of vertices. The boundary $\overline \Omega_0^*\setminus\Omega_0^*$ consists of $n$ connected components. We call the endpoints of the components $x_1, \ldots, x_{2n}$, labelled in counterclockwise order in such a way that the connected components are respectively formed by $(x_{2}x_3),\ldots,(x_{2n}x_{1})$. We can assume without loss of generality that $x_1 $ and $x_{2n}$ are adjacent to two convex corners. (In general, $x_i$ will be adjacent to a white corner when that corner is convex, and will be within distance 3 of that corner if that corner is concave). See Figure~\ref{fig::temperleyan} for an illustration. In fact, the arc $(x_{2n}x_1)$ is distinguished by the fact that it is the only connected portion of the boundary $\overline \Omega_0^*\setminus\Omega_0^*$ connecting two vertices adjacent to a convex corner. In other for each $1\le i \le n-1$, among $x_{2i}$ and $x_{2i +1}$ there is exactly one of these two vertices which is adjacent to a convex corner (the other will be within distance three of a concave white corner). 
This comes from our assumption that between two consecutive concave corners, the black vertices are all of type $\blacksquare_0$. As we see, this does not prevent concave white corners to come consecutively, but in particular no three consecutive white corners can exist.

 
 We denote by $z_1^+,\ldots,z_{n-1}^+$ the vertices of $\Omega^*_0$ adjacent to the concave white corners, also labelled in  counterclockwise order starting immediately after $x_1$ (so $z_i^+$ can be identified with a black corner, and is on the outer boundary of $\Omega_0^*$). As already mentioned, by our assumption, for $1\le i\le n-1$, there is exactly one vertex of $\{x_{2i},x_{2i+1}\}$ which is adjacent to $z_i^+$. Let us denote it by $z_i \in \{x_{2i}, x_{2i+1}\}$. See Figure~\ref{fig::temperleyan}.

A dimer cover of $\Omega^*$, denoted by $\omega$, is a subset of the edges of $\Omega^*$ such that each vertex belongs to exactly one edge in $\omega$. Denote by $\LD(\Omega^*)$ the set of dimer covers of $\Omega^*$. 
Given a dimer configuration $\omega \in \LD(\Omega^*)$, we now associate to $\omega$ a tree on $\overline\Omega_0^*$ in the following manner.  
Let us consider each edge $e$ in $\omega$ as directed from black to white vertex, and let us extend it to the next black vertex along this direction (i.e., ``double'' the oriented edge $e$). We only consider the edges starting from $\blacksquare_0$ vertices. In this manner, we obtain a subset of edges of $\overline\Omega^*_0$. We add the edges $\cup_{i=1}^{n}(x_{2i}x_{2i+1})\cup\cup_{i=1}^{n-1}\{z_iz_i^+\}$ to this subset and also direct the edges $\cup_{i=1}^{n-1}(x_{2i}x_{2i+1})\cup\cup_{i=1}^{n-1}\{z_iz_i^+\}$ from the black vertices to white vertices. 
Note however that we do \emph{not} direct the edges along $(x_{2n}x_{1})$ (indeed no consistent orientation on the last arc).
Denote by $\omega^*$ the subset of edges we obtain, and note that except on $(x_{2n} x_1)$ out of every vertex in $\Omega_0^*$ there is a single edge going out of this vertex in $\omega^*$.
See Figure~\ref{fig::Dimer}.
\begin{figure}[ht!]
\begin{center}
\includegraphics[width=1\textwidth]{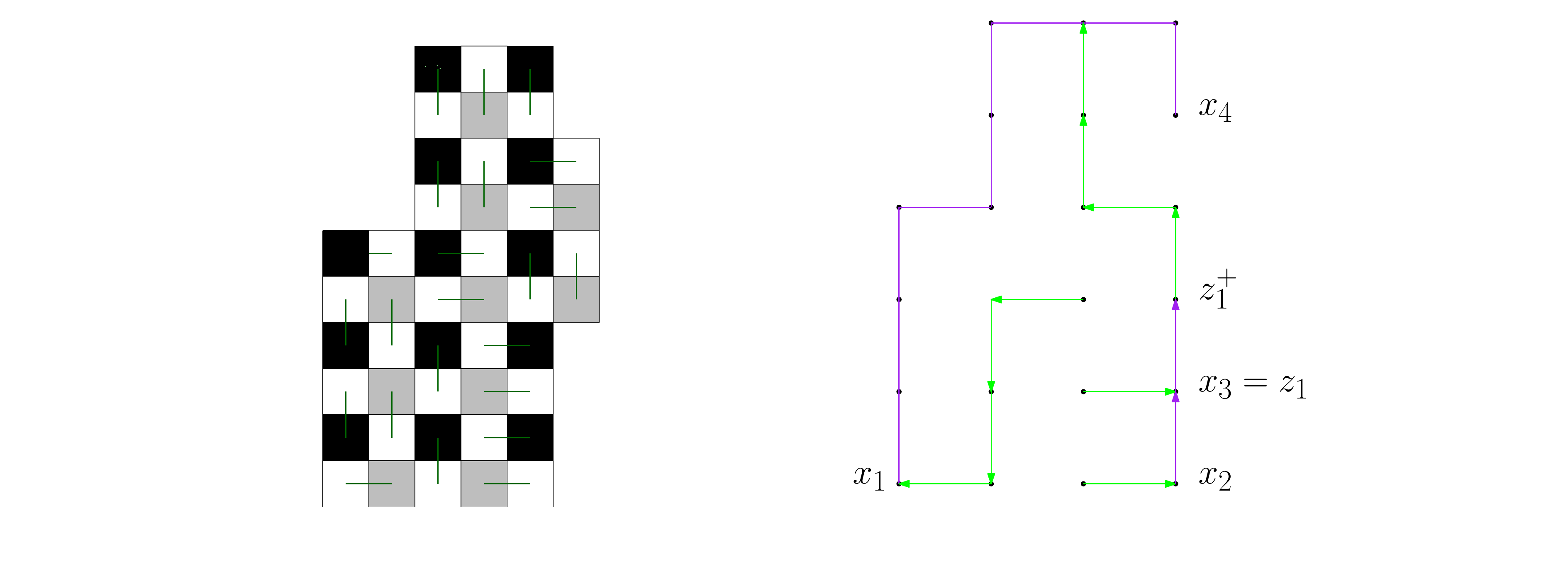}
\end{center}
\caption{\label{fig::Dimer}In the left picture, the set of dark green edges is a dimer configuration. In the right picture, the green edges are from the dimer configuration and the purple edges are added artificially.}
\end{figure}
Denote by $\LU(\overline\Omega^*_0)$ the subset of spanning trees on $\overline\Omega^*_0$
satisfying the following conditions: the spanning tree contains $\cup_{i=1}^n(x_{2i}x_{2i+1})$ and for every $1\le i\le n-1$, the first step of the branch connecting $(x_{2i}x_{2i+1})$ to $(x_{2n}x_{1})$ is through the edge $z_iz_i^+$. Equivalently, we could also wire the boundary components $c_1= (x_2 x_3), \ldots c_n = (x_{2n} x_1)$ and view the spanning tree as an oriented (or rooted) spanning tree, rooted at $c_n = (x_{2n} x_1)$, and such that the path connecting $c_i$ to $c_n$ starts with the oriented edge $z_iz_i^+$.

\begin{lemma}[Temperley's bijection]\label{lem::bij}
The map $\omega\mapsto\omega^*$ is a bijection between $\LD(\Omega^*)$ and $\LU(\overline\Omega_0^*)$.
\end{lemma}
\begin{proof}
First, we will show that $\omega^*$ is indeed an element in $\LU(\overline\Omega_0^*)$. Since $\omega$ is a dimer cover, it is clear that $\omega^*$ contains all the vertices of $\overline\Omega_0^*$. We need to show that $\omega^*$ does not contain loops. Suppose that $L$ is a loop on $\overline\Omega_0^*$ which is contained in  $\omega^*$. We denote by $\tilde L$ the union of $L$ and the squares surrounded by $L$. By induction on the number $\blacksquare_1$ squares in $\tilde L$, we have that the number of squares in $\tilde L$ is odd. Note that if $L$ belongs to $\omega^*$, the subgraph $\tilde L$ is covered by edges in $\omega$. This implies that the number of squares in $\tilde L$ should be even, which is a contradiction. Thus, the loop $L$ can not belong to $\omega^*$. This implies that $\omega^*$ is a spanning forest on $\overline\Omega_0^*$.

To show that $\omega^*$ is a tree, recall that starting from every vertex in $\overline \Omega_0^* $, except on $c_n = (x_{2n} x_1)$, there is a single oriented edge in $\omega^*$ going out of that vertex. Following the edges in the forward direction, we obtain a path which necessarily ends in $c_n$. Thus every vertex is connected to $c_n$ and $\omega^*$ is in $\LU(\overline\Omega_0^*)$ (i.e., is a spanning tree), as by definition of $\omega^*$, this path connecting $c_i$ to $c_n$ starts with the oriented edge $z_iz_i^+$. 


Second, we will show that this is a bijection. Define the dual graph of $\overline\Omega_1^*$ similarly as defining $\overline\Omega_0^*$ by considering the $\blacksquare_1$ squares. Note that every spanning tree on $\overline\Omega_0^*$ determines a dual forest on the dual graph of $\overline\Omega_1^*$. Since every dimer cover $\omega$ is determined by $\omega^*$ and the corresponding dual forest, we deduce that the map $\omega\mapsto\omega^*$ is injective. For every $\omega^*\in\LU(\overline\Omega_0^*)$, it is clear that dividing the edges of $\omega^*$ and the edges of the corresponding dual forest, we will get a dimer cover on $\Omega$. Thus, the map $\omega\mapsto\omega^*$ is surjective. This completes the proof.
\end{proof}

\subsection{Generalisation}
\label{SS:gen}

We now explain briefly how the setup considered in this paper can be generalised to a much bigger class of planar graphs. Roughly these graphs are obtained by superposing a planar graph and their dual (which is a setup in which Temperley's bijection is known to apply, see \cite{KPW}) but where we pay attention to the choice of boundary conditions. Let $\{G_\delta\}_{\delta >0}$ denote a sequence of graphs embedded in a domain $\Omega$, with an edge outer-boundary forming a simple curve $\lambda$.   
%
%
Let us also fix $n\ge 2$ and boundary arcs $(x_2 x_3), \ldots, (x_{2n} x_1)$ on $\partial \Omega$. We let $ x_1^\delta, x_{2n}^\delta$ denote approximations of $x_1, \ldots, x_{2n}$ along $\lambda$ (and hence on $V(G_\delta)$). Let $c_i$ denote the (wired) arc $(x_{2i} x_{2i+1})$ (understood cyclically), and for each $1\le i \le n$, fix $z_i \in \{x_{2i}, x_{2i+1}\}$. As in Theorem \ref{thm::heightconv} we assume that $\Omega$ is a simply connected domain with simple $C^1$ boundary. In fact it would suffice to assume that the reflecting arcs correspond to disjoint simple $C^1$ curves. 

Let $\cZ = \cup_{i=1}^n \{ z_i\}$.

 

Let $G^*_\delta$ denote the interior planar dual of $G_\delta$, and let $G^s$ denote the superposition of $G_\delta$ and $G^*_\delta$. The graph $G^s$ is bipartite, with $V(G^s) = B^s\cup W^s$ with $B^s = V(G_\delta) \cup V(G^*_\delta)$, and $W = E(G_\delta)$, and there is an edge between $e $ between $b \in B^s$ and $w \in W^s$ if and only if $(bw)$ is part of a primal or dual edge of $G_\delta$. 
 Finally, for each $1\le i \le n$, we remove from the superposition graph $G^s$ the arcs $c_i = (x_{2i} x_{2i+1})$, understood cyclically, and if $1\le i \le n-1$ we also remove the edge $\{z_i w_i\}$, where $w_i$ is the unique white vertex adjacent to $z_i$ on $\lambda$. We let $\cG^D_\delta$ be the resulting graph and call it the dimer graph.  

To make the parallel with the previous section clearer, it will be useful to call $z_i^+$ the black vertex adjacent to $w_i$ on $\lambda$ which is not $z_i$, so $w_i \in \{z_i z_{i}^+\}$. We call a white vertex on the boundary of $\cG^D_\delta$ a concave corner if it is adjacent to the arc $z_i^+$ for some $1\le i \le n-1$. We call it a convex corner if it is adjacent to $c_i$, for $1\le i \le n$. Thus there are $n+1$ convex white corners, and $n-1$ white concave corners.


We make the following assumptions about $G^\d $.
We may allow $G^\d$ to have
oriented edges with weights (but all the edge weights on $G^*_\delta$ will be set to one). A continuous time simple random walk
$\{X_t\}_{t \ge 0}$ on such a graph
$G^\d$ is defined in the usual way: the walker jumps from $u$ to $v$ at
rate $w(u,v)$ where $w(u,v)$ denotes the weight of the oriented edge
$(u,v)$. Furthermore, the walk is killed as soon as it reaches a vertex from $G_\delta \setminus G^-_\delta$. Given a vertex $u $ in $G^\d$, let $\PP_u$ denote the law of
continuous time simple random walk on $G^\d$ started from $u$.

For $A \subset \C$,
we denote by $A^\d$ the set of vertices of $\Gd$ in $A$.

\begin{enumerate} 

  \item \label{Bounded}\textbf{(Bounded density)} There exists $C$ such that for any $x \in \C$, the number of vertices of $G^\d$ in the square $x + [0,\delta]^2$ is smaller than $C$.

  \item \label{embedding} \textbf{(Good embedding)} The edges of the graph are embedded
  in such a way that they are piecewise smooth, do not cross each other and have uniformly bounded winding. Also, $0$ is a vertex.


\item \label{irreducibility}  \textbf{(Irreducible)} The continuous-time random walk on $G^\d$, is irreducible in the
  sense that for any two vertices $u$ and $v$ in $G^\d$, $\PP_u(X_1 =
  v)>0$.
\item \label{InvP} \textbf{(Invariance principle)}
The continuous time random walk $\{X_t\}_{t \ge 0}$ on $G^\d$
  started from $0$ satisfies:
\[
 {( X_{t/\delta^2})_{t \ge 0}} \xrightarrow[\delta \to 0 ]{(d) }
(B_{\phi(t)})_{t \ge 0}
\]
where $(B_t, t \ge 0)$ is a Brownian motion in $\Omega$ with \textbf{normal reflection} along $(x_1 x_2) \cup \ldots \cup (x_{2n-1} x_{2n})$, started from
$0$, and $\phi$ is a nondecreasing, continuous, possibly random function
satisfying $\phi(0) = 0 $ and $\phi(\infty) = \infty$. The above convergence
holds in law in Skorokhod topology.

\item \label{crossingestimate} \textbf{(Uniform crossing estimate).}  
Let
    $\cR$
   be the horizontal rectangle $[0,3]\times [0,1]$ and $\cR'$ be the vertical
rectangle with same dimensions, and let $B_1 := B((1/2,1/2),1/4)$ be the
\emph{starting ball} and $B_2:= B((5/2,1/2),1/4) $ be the \emph{target ball}. There
exist constants $\delta_0 >0$ and $\alpha_0>0$ such that for all $z
\in \C, \delta>0$, $\ell \ge 1/\delta_0$, $v \in \ell \delta B_1$ such that $v+z \in \Gd$,
\begin{equation}
  \PP_{v+z}(X \text{ hits }(\ell \delta  B_2+z) \text{ before exiting }  (\ell \delta \cR+z))
>\alpha_0.\label{eq:cross_left_right}
\end{equation}

The same statement as above holds for crossing from right to left, i.e., for
any $v \in \ell \delta B_2$, \eqref{eq:cross_left_right} holds if we replace $B_2$ by
$B_1$.
Also, the corresponding statements hold for the vertical rectangle $\cR'$.
\end{enumerate}

These assumptions essentially mirror those of \cite{BLRdimers}. In addition, in order to handle potential difficulties near the interface between reflecting and wired boundary pieces, we make the following additional assumption: 

 \begin{enumerate}
 
 \item[6.] For every topological quadrilateral $(\Lambda, a,b,c,d) \subset \Omega$, the discrete and continuous extremal lengths of $\Lambda$ are uniformly comparable. 
 
 \end{enumerate}

For precise definitions of the notions of discrete and continuous extremal lengths above we refer the reader for instance to \cite{ChelkakRobustComplexAnalysis}. In particular, by Theorem 7.1 in \cite{ChelkakRobustComplexAnalysis}, this additional assumption holds as soon as the edges of the graph $G_\delta$ are straight, edge lengths are locally comparable, the edge weights are uniformly elliptic, and the angles between edges are uniformly bounded away from 0 and $\pi$; see again \cite{ChelkakRobustComplexAnalysis} for details.

\medskip We will consider the set $\mathcal{U}( G_\delta)$ of spanning trees on $G_\delta$ with the boundary conditions that the arcs $c_1 = (x_2 x_3), \ldots c_n = (x_{2n} x_1)$ are each wired (but not to one another). Since $G_\delta$ is directed, by definition, a spanning tree $\mathbf{t} \in \cU(G_\delta)$ is such that there is exactly one forward edge out of every non-wired vertex.
 There is a natural measure on $\cU(G_\delta)$ which is simply
$$
\PP( \mathbf{t}) \propto \prod_{e \in \mathbf{t}} w_e; \quad \mathbf{t} \in \cU(G_\delta).
$$
We also consider the set $\cD (\cG^D_\delta)$ of dimer configurations on $\cG^D_\delta$, which we equip with a probability measure
$$
\PP (\mathbf{m})\propto  \prod_{ e \in \mathbf{m}} w_e; \quad \mathbf{m} \in \cD(\cG^D_\delta),
$$
where an edge $e$ of $\cG^D_\delta$ inherits the weight $w_e =1$ if it was part of a dual edge of $G_\delta$, and otherwise the weight of the corresponding primal edge oriented from black to white.

Temperley's bijection has an obvious extension to this setting: namely, given $\mathbf{m} \in \cD(\cG^D_\delta)$ we associate to it the spanning tree $\mathbf{t} \in \cU(G_\delta)$ by considering all the edges of $\mathbf{m}$ emanating out of a primal black vertex and doubling them up: i.e., for each $(bw) \in \mathbf{m}$ with $b \in V(G_\delta)$, we add to $\mathbf{t}$ the unique (primal) edge $bb'$ such that $w\in (bb')$. 

It is not hard to see that Temperley's bijection can be generalised to this setup:
\begin{lemma}[Generalised Temperley's bijection]\label{lem::bij2}
The map $\mathbf{m} \mapsto\mathbf{t}$ is a measure-preserving bijection between $\cD(\cG^D_\delta)$ and $\cU(G_\delta)$.
\end{lemma}

As a result of this bijection, all the results we prove for the dimer model on black-piecewise Temperleyan domains generalise in a straightforward manner to the dimer model on $\cG^D_\delta$ with the above law. We will not comment further on this distinction except where necessary, and we will write our proofs for concreteness in the setup of Section \ref{SS:bij}.

\subsection{Coupling of flow lines, counterflow lines and GFF}
\label{subsec::GFF}

In this section, we will recall some facts about imaginary geometry from~\cite{MillerSheffieldIG1} and~\cite{MillerSheffieldIG4}. We will use them to establish the coupling of the Gaussian free field (which will be the limit of the height function) and the continuous tree (which will be the limit of the Temperleyan tree of the last section), which will be described later. 

We first fix some constants:
\[\kappa=2;\quad\kappa'=8;\quad\lambda=\frac{\pi}{\sqrt 2};\quad\lambda'=\frac{\pi}{\sqrt 8};\quad\chi=\frac{1}{\sqrt 2}.\]

Let us start with the so-called \textbf{counterflow line coupling}. For any simply connected domain $(\Omega;a,b)$, where $a,b$ are two points on the boundary of $\Omega$ (understood as prime ends), we define the harmonic function $u_{(\Omega;a,b)}$ as following. Fix a  conformal map $\varphi$ from $\Omega$ onto the upper half plane $\HH$, which maps $a$ to $0$ and $b$ to $\infty$. 
Define $u_{(\HH;0,\infty)}$ to be the unique bounded harmonic function, which equals $\lambda'$ on $(-\infty,0)$ and equals $-\lambda'$ on $(0,\infty)$. Define
\[u_{(\Omega;a,b)}:=u_{(\HH;0,\infty)}\circ\varphi-\chi\arg\varphi'.\]
It can be checked that this does not depend on the choice of the map $\varphi$. 
Denote by $h_{\Omega}$ the Gaussian free field (GFF) with Dirichlet boundary condition on $\Omega$, with two-point or covariance function given by the Green function $G(x, y)$, normalised so that $G(x,y) = - \log |x-y| + O(1)$ as $y \to x$. Denote by $\gamma$ a chordal $\SLE_{\kappa'}$ in $\Omega$ from $a$ to $b$. 

\begin{theorem}~\cite[Theorem 1.1, Theorem 1.2]{MillerSheffieldIG1}\label{thm::countercoupling}
Fix a simply connected domain $(\Omega;a,b)$. There is a unique coupling of $h = h_{\Omega}+u_{(\Omega;a,b)}$ and $\gamma$, such that  for any stopping time $\tau$ of $\gamma$, the conditional law of $h_{\Omega}+u_{(\Omega;a,b)}$ given $\gamma[0,\tau]$ restricted in $(\Omega\setminus\gamma[0,\tau];\gamma(\tau),b)$ equals the law of 
\[h_{\Omega\setminus\gamma[0,\tau]}+u_{(\Omega\setminus\gamma[0,\tau];\gamma(\tau),b)},\]
where $h_{\Omega\setminus\gamma[0,\tau]}$ is a on $(\Omega\setminus\gamma[0,\tau];\gamma(\tau),b)$ which is independent of $\gamma((0, \tau])$. Moreover, in this coupling, $\gamma$ is measurably determined by $h_{\Omega}$.
\end{theorem}

Second, we introduce the \textbf{flow line coupling}. Fix a sequence $\{x_1,\ldots,x_{2n}\}$ such that $x_1<\ldots<x_{2n}$ and fix a sequence $\{z_1,\ldots,z_{n-1}\}$ such that $x_{2i}\le z_i\le x_{2i+1}$ for $1\le i\le n-1$. We define $u_{\HH}(\cdot;x_1,\ldots,x_{2n};z_1,\ldots,z_{n-1})$ to be the unique bounded harmonic function with the following boundary conditions: 
\begin{align*} \label{eq: coef_expansion}
u_{\HH}(\cdot;x_1,\ldots,x_{2n};z_1,\ldots,z_{n-1}) = 
\begin{cases}
0 , & \text{on } \cup_{i=1}^n(x_{2i-1},x_{2i}), \\
-\lambda ,  &\text{on } \cup_{i=1}^{n-1}(x_{2i},z_i)\cup(x_{2n},\infty), \\
\lambda, &\text{on } \cup_{i=1}^{n-1}(z_i,x_{2i+1})\cup(-\infty,x_{1}).
\end{cases}
\end{align*}
We fix a simply connected domain $(\Omega;x_1,\ldots,x_{2n};z_1,\ldots,z_{n-1})$, where $x_1,\ldots,x_{2n}$ are boundary  points on $\partial \Omega$ in counterclockwise orientation and $z_i$ is on the arc $[x_{2i}x_{2i+1}]$ for $1\le i\le n-1$. We fix a conformal map $\varphi$ from $\Omega$ onto $\HH$ such that $\varphi(x_1)<\ldots<\varphi(x_{2n})$.
Define
\[u_{\Omega}(\cdot;x_1,\ldots,x_{2n};z_1,\ldots,z_{n-1}):=u_{\HH}(\cdot;\varphi(x_1),\ldots,\varphi(x_{2n});\varphi(z_1),\ldots,\varphi(z_{n-1}))\circ\varphi-\chi\arg\varphi'.\]
Again, it can be checked this does not depend on the choice of the conformal map $\varphi$. 
In~\cite{MillerSheffieldIG1}, the authors established the coupling of $h_\Omega$ and the flow lines starting from the boundary points in the following sense (this can be viewed as a recursive definition of the set of flow lines emanating out of $z_1, \ldots, z_{n-1})$. 

\begin{theorem}~\cite[Theorem 1.1, Theorem 1.2]{MillerSheffieldIG1}\label{thm::flowcoupling}
Fix a simply connected domain $\Omega$ together with marked boundary points $(x_1,\ldots,x_{2n};z_1,\ldots,z_{n-1})$ as above. There is a unique coupling of curves $\{\gamma_1^I,\ldots,\gamma_{n-1}^I\}$ and $h_{\Omega}+u_{\Omega}(\cdot;x_1,\ldots,x_{2n};z_1,\ldots,z_{n-1})$ such that the following properties hold almost surely.
\begin{itemize}
\item
The law of the flow line $\gamma^I_1$ can be described as follows: we denote by $W$ the driving function of $\varphi(\gamma^I_1)$ and denote by $\{g_t:t\ge 0\}$ the corresponding Loewner flow in the upper half plane $\HH$, then we have
\begin{equation}\label{eqn::drving}
d W_t=\sqrt{2}dB_t+\sum_{i=1}^{2n}\frac{-1}{W_t-g_t(x_i)}+\sum_{i=2}^{n-1}\frac{2}{W_t-g_t(z_i)},\quad W_0=z_1,
\end{equation}
where $(B_t:t\ge 0)$ is the standard one-dimentional Brownian motion.
\item
The curve $\gamma_1^I$ will end in $\cup_{j=2}^{n}(x_{2j}z_j)$, where we identify $z_n = x_1$. Given $\gamma^I_1$, and given the event that $\gamma^I_1$ ends in $\cup_{j=2}^{n}(x_{2j}z_{j})$, we denote by $\Omega_L$ the connected component of $\HH\setminus\gamma^I_1$ in the left side of $\gamma^I_1$ and denote by $\Omega_R$ the connected component of $\HH\setminus\gamma^I_1$ in the right side of $\gamma^I_1$. The conditional law of $h_\Omega+u_{\Omega}(\cdot;x_1,\ldots,x_{2n};z_1,\ldots,z_{n-1})$ restricted in $\Omega_L$ equals
\[h_{\Omega_L}+u_{\Omega_L}(\cdot;x_1,x_2, x_{2j+1},\ldots,x_{2n};z_{j},\ldots,z_{n-1})\]
and the conditional law of $h_\Omega+u_{\Omega}(\cdot;x_1,\ldots,x_{2n};z_1,\ldots,z_{n-1})$ restricted in $\Omega_R$ equals
\[h_{\Omega_R}+u_{\Omega_R}(\cdot;x_{3},\ldots,x_{2j-1},x_{2j};z_2,\ldots,z_{j-1}).\]
The conditional law of $\{\gamma^I_{j},\ldots,\gamma^I_{n-1}\}$ equals the law of the flow lines of 
\[h_{\Omega_L}+u_{\Omega_L}(\cdot;x_1,x_2, x_{2j+1},\ldots,x_{2n};z_{j},\ldots,z_{n-1}).\]
 The conditional law of $\{\gamma^I_{2},\ldots,\gamma^I_{j-1}\}$ equals the law of the flow lines of 
\[h_{\Omega_R}+u_{\Omega_R}(\cdot;x_{3},\ldots,x_{2j-1},x_{2j};z_2,\ldots,z_{j-1}).\]
\item
Given $\cup_{i=1}^{n-1}\gamma_i$, there will be $n$ connected components of $\Omega\setminus\cup_{i=1}^{n-1}\gamma_i$ and each componnent contains one boundary arc $(x_{2i-1}x_{2i})$ for some $1\le i\le n$. We denote by $\Omega_i$ the connected component which contains $(x_{2i-1}x_{2i})$ on its boundary.
The conditional law of $h_\Omega+u_{\Omega}(\cdot;x_1,\ldots,x_{2n};z_1,\ldots,z_{n-1})$ restricted in $\Omega_i$ equals the law of 
\[h_{\Omega_i}+u_{(\Omega_i;x_{2i-1},x_{2i})}-\frac{\pi\chi}{2},\]
\end{itemize}
\end{theorem}
See Figure~\ref{fig::general domain} for an illustration.
\begin{figure}[ht!]
\begin{center}
\includegraphics[width=0.5\textwidth]{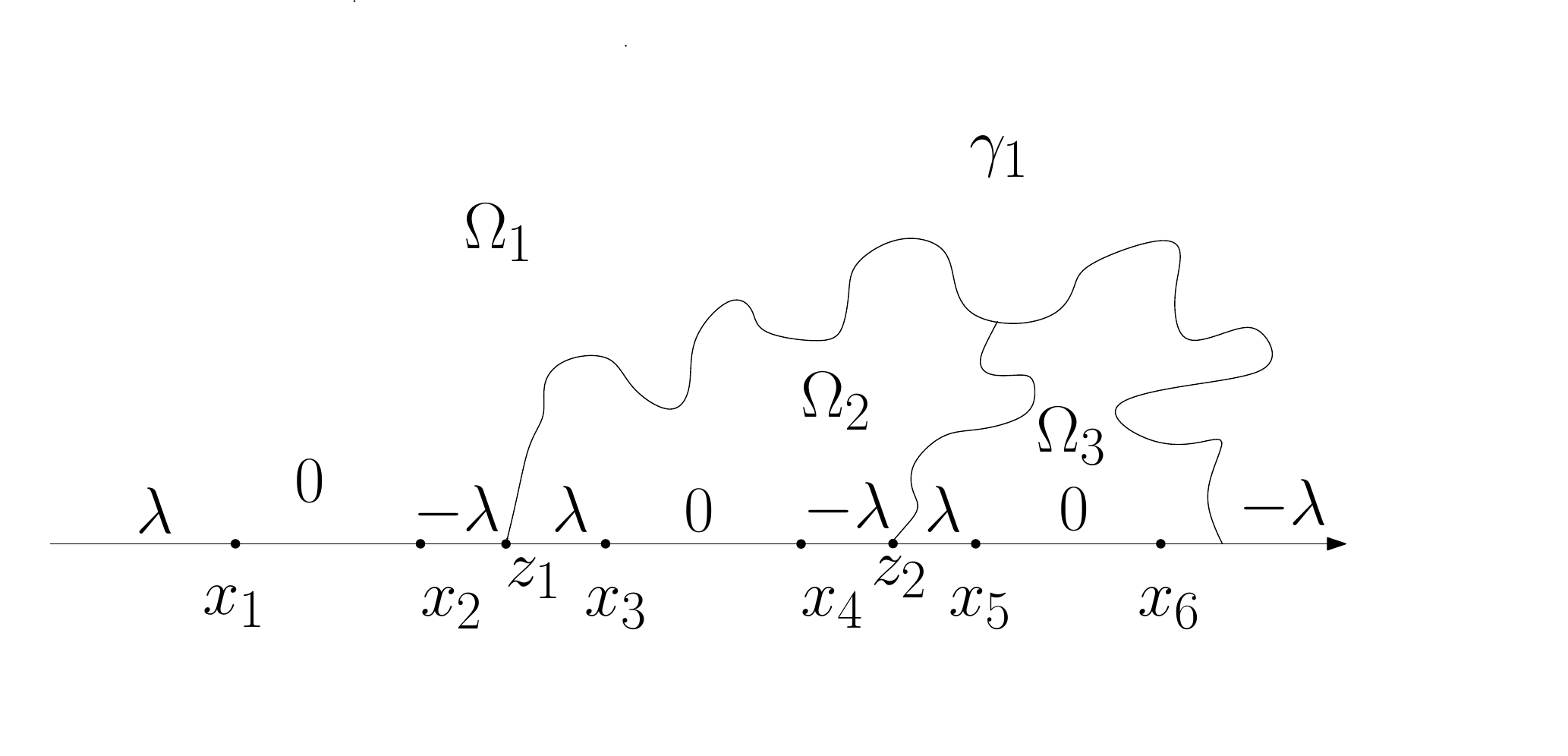}
\end{center}
\caption{\label{fig::general domain}This is an illustration of the boundary values of the $\GFF$ and flow lines when $n=3$.}
\end{figure}

The counterflow line coupling of Theorem  \ref{thm::countercoupling} can be augmented in a useful way by considering the left and right boundaries of the counterflow line $\gamma$ when it reaches a point $z \in \Omega$. 
This leads to the following coupling of counterflow and flow lines starting from interior points in imaginary geometry, stated in \cite[Theorem 1.1]{MillerSheffieldIG4}. 
Fix a dense countable set $\cQ = \{z_i\}_{i\ge 1}$ in $\Omega$. Let $h$ and $\gamma$ be coupled as in the counterflow line coupling of Theorem \ref{thm::countercoupling}. For every $z\in \Omega$, we denote by $\gamma_z$ the flow line of $h$ from $z$ with angle $-\frac{\pi}{2}$ and denote by $\gamma^*_z$ the flow line of $h$ from $z$ with angle $\frac{\pi}{2}$. They can be viewed as the boundaries of the counterflow line $\gamma$ when it reaches the point $z$. We will need the following properties of this coupling: 

\begin{theorem}~\cite[Theorem 1.1, Theorem 1.2, Theorem 1.7, Theorem 1.13]{MillerSheffieldIG4}\label{thm::intcoupling}
Fix a simply connected domain $(\Omega;a,b)$. There exists a unique coupling of $\{\gamma_{z_i}\}_{i\ge 1}$, $\{\gamma^*_{z_i}\}_{i\ge 1}$, $\gamma$ and $h_{\Omega}$ such that the following properties hold almost surely. 
\begin{itemize}
\item
The curve $\gamma$ is the counterflow line of $h_{\Omega}$ in the sense of Theorem~\ref{thm::countercoupling}. Let $n\ge 1$ and fix $z_1,\ldots,z_n \in \cQ$. The conditional law of $h_{\Omega}+u_{(\Omega;a,b)}$ given $\{\gamma_{z_i}\}_{1\le i\le n}$ equals to the law of 
\[h_{\Omega\setminus\cup_{i=1}^{n}\gamma_{z_i}}+u_{(\Omega\setminus\cup_{i=1}^{n}\gamma_{z_i};a,b)},\]
where $h_{\Omega\setminus\cup_{i=1}^{n}\gamma_{z_i}}$ is an independent GFF. The conditional law of $h_{\Omega}+u_{(\Omega;a,b)}$ given $\{\gamma_{z_i}^*\}_{1\le i\le n}$ equals to the law of 
\[h_{\Omega\setminus\cup_{i=1}^{n}\gamma_z^*}+u_{(\Omega\setminus\cup_{i=1}^{n}\gamma_z^*;a,b)},\]
where $h_{\Omega\setminus\cup_{i=1}^{n}\gamma_z^*}$ is independent of $\{\gamma^*_{z_i}\}_{1\le i\le n}$. Moreover, $\{\gamma_{z_i}\}_{1\le i\le n}$ and $\{\gamma^*_{z_i}\}_{1\le i\le n}$ are determined by $h_{\Omega}$.
\item
For every $i\ge 1$, the flow lines $\gamma_{z_i}$ and $\gamma^*_{z_i}$ are simple curves. Moreover, $\gamma_{z_i}$ hits $\partial \Omega$ at $(ab)$ and $\gamma_{z_i}^*$ hits $\partial \Omega$ at $(ba)$ and $\gamma_{z_i}\cap\gamma_{z_i}^*=\{z_i\}$. For $i\neq j$, the curves $\gamma_{z_i}$ merges with $\gamma_{z_j}$ when they intersect each other and the curves $\gamma^*_{z_i}$ merges with $\gamma^*_{z_j}$ when they intersect each other.
\item
Define $\tau_z$ to be the first time that $\gamma$ hits $z$. Almost surely, the left boundary $\gamma^*[0,\tau_z]$ equals to $\gamma_z$ and the right boundary of $\gamma[0,\tau_z]$ equals to $\gamma_z$. 
\end{itemize}
\end{theorem}

The marginal law (and joint law) of $\gamma_z$ and $\gamma_z^*$ could further be specified but we will not need this in the following, in fact we will recover a description of these laws as part of our proof.

\section{Coupling of the continuous tree and GFF when $n=2$}
\label{sec::GFFtree}

In this section, we will prove the convergence of the discrete Temperleyan trees introduced in Lemma~\ref{lem::bij} and establish the coupling of the limiting tree with the GFF with boundary conditions specified in Theorem \ref{thm::flowcoupling}, in the case $n=2$.
In this section, we will assume that $z_1=x_3$. (The proof in the other case $z_1=x_2$ is almost the same). It is convenient to change our notations slightly in the cas $n =2$, and denote the point $z_1$ by $b$, while we denote the remaining points $x_2,x_4,x_1$ by $a,c,d$ respectively.

Suppose $(\Omega_\delta;a_\delta,b_\delta,c_\delta,d_\delta)$ is a discrete simply connected domain with four marked boundary points, where $\Omega_\delta$ is a subset of $\delta\Z^2$ and $a_\delta,b_\delta,c_\delta$ and $d_\delta$ are on $\partial \Omega_\delta$. Consider the uniform spanning tree ($\ust$) on $\Omega_\delta$ on the graph where the boundary $(a_\delta b_\delta)$ is wired to be one vertex and the boundary arc $(c_\delta d_\delta)$ is wired to be another vertex. (This is the UST with so-called \emph{alternating boundary conditions} introduced in \cite{HanLiuWuUST} and already discussed above Theorem \ref{T:condSLE8}). Recall that there is a (single) branch, $\gamma_\delta$, connecting the two arcs $(a_\delta b_\delta)$ and $(c_\delta d_\delta)$ in this tree.  We may thus define the event 
\[E_\delta:=\{\text{the branch $\gamma_\delta$ starts from }b_\delta b_\delta^+\}.\]
As already explained above Theorem \ref{T:condSLE8}, it is not hard to see that the tree $\cT_\delta$ obtained from Temperley's bijection in the black $4$-Temperleyan domain is nothing but a UST with alternating boundary conditions, conditionally given the event $E_\delta$.
 Recall also 
 $\eta_{L,\delta}$ is the discrete Peano curve on the left side of $\gamma_{\delta}$ starting from $d_\delta$ and ending in $a_\delta$, while  $\eta_{R,\delta}$ is the discrete Peano curve from $c_\delta$ to $b_\delta$ on the right side of $\gamma_{\delta}$. For every $z_\delta\in \Omega_\delta$, denote by $\gamma_{z_\delta}$ the path from $z_\delta$ to $(a_\delta b_\delta)\cup(c_\delta d_\delta)\cup\gamma_{\delta}$.

We now recall the Schramm topology (introduced in \cite{{SchrammFirstSLE}}) which we use for the convergence of trees.
\begin{itemize}
\item
Let $\mathcal{P}$ be the space of unparameterised paths in $\C$.  Define a metric on $\mathcal{P}$ as follows:
\begin{align}\label{eqn::curves_metric}
d(\gamma_1, \gamma_2):=\inf\sup_{t\in[0,1]}\left|\hat{\gamma}_1(t)-\hat{\gamma}_2(t)\right|,
\end{align}
where the infimum is taken over all the choices of  parameterisations  $\hat{\gamma}_1$ and $\hat{\gamma}_2$ of $\gamma_1$ and $\gamma_2$. Suppose $\{{\gamma_\delta}\}_{\delta>0}$ and $\gamma$ are continuous curves in $\overline\HH$  from $0$ to $\infty$ and their Loewner driving functions are continuous. If $\gamma_\delta$ converges to $\gamma$ under the metric~\eqref{eqn::curves_metric}, then after parametrising by half-plane capacity, $\gamma_\delta$ converges to $\gamma$ locally uniformly as continuous functions on $[0,+\infty)$.
\item
For any metric space $X$, denote by $\LH(X)$ the space of closed subsets of $X$ equipped with the Hausdorff metric. For every $z,w\in \overline \Omega$, denote by $\LP(z,w,\overline \Omega)$ the set of paths connecting $z$ and $w$ in $\overline \Omega$. We will view 
\[\{z:z\in\overline \Omega\}\times \{w:w\in\overline \Omega\}\times \cup_{z,w\in\overline \Omega}\LP(z,w,\overline \Omega)\]
as a subset of $\LH(\overline \Omega\times \overline \Omega\times \LH(\overline \Omega))$. We will denote by $d_{\LH}$ the corresponding metric and denote by $\overline \Omega_{\LH}$ the closure of $\{z:z\in\overline \Omega\}\times \{w:w\in\overline \Omega\}\times \cup_{z,w\in\overline \Omega}\LP(z,w,\overline \Omega)$ under $d_{\LH}$. 
\end{itemize}

We now give the setup in which we will prove convergence of the discrete Temperleyan trees. Fix a simply connected domain $(\Omega;a,b,c,d)$ such that $\partial \Omega$ is locally connected and $(a,b,c,d)$ are four points on $\partial \Omega$ in counterclockwise order, and suppose that there exists a simply connected domain $\tilde \Omega$ whose boundary is a  simple $C^1$ curve, such that $\Omega\subset\tilde \Omega$ and $\partial \Omega\cap\partial\tilde \Omega$ equals $[bc]\cup[da]$. Suppose also that a sequence of discrete simply connected domains $\{(\Omega_\delta;a_\delta,b_\delta,c_\delta,d_\delta)\}_{\delta>0}$ converges to $(\Omega;a,b,c,d)$ in the following sense: there exists $C>0$ such that
\begin{equation}\label{eqn::quad}
d((b_\delta c_\delta),(bc))\le C\delta,\quad d((d_\delta a_\delta),(da))\le C\delta\quad\text{and}\quad d((a_\delta b_\delta),(ab))\to 0,\quad d((c_\delta d_\delta),(cd))\to 0\text{ as }\delta\to 0.
\end{equation}
For convenience, we require $\Omega_\delta\subset \Omega$. We emphasise that the assumptions on $\partial \Omega$ and the convergence of discrete domains are mainly {needed} to handle the reflecting boundaries: the assumption on $\tilde \Omega$ amounts to saying that the reflecting arcs $(bc)$ and $(da)$ should be simple and smooth, while the wired arcs $(ab)$ and $(cd)$ can be fairly arbitrary. We will view the discrete tree $\LT_\delta$ as 
\[\{(x_\delta,y_\delta,\gamma_{x_\delta,y_\delta})|x_\delta,y_\delta\in \Omega_\delta\text{ and }\gamma_{x_\delta,y_\delta}\text{ is the unique path connecting }x_\delta\text{ and }y_\delta\text{ in }\LT_\delta\},\] which is an element in $\overline \Omega_{\LH}$.

\begin{theorem}\label{thm::treeconv}
As $\delta \to 0$, the discrete tree $\LT_\delta$ converges in distribution (with respect to the metric $d_{\LH}$). Let $\LT$ denote the limit, which we call a continuous tree. The continuous tree $\LT$ can be constructed as follows: 

Fix any dense set $\cQ = \{z_i\}_{i\ge 1}$ in $\Omega$. Let $\gamma$ be a chordal $\SLE_2(-1,-1;-1,-1)$ curve in $\Omega$ with force points $a,d$ and $b,c$ starting from $b$, and suppose it is coupled with $h_\Omega+u_\Omega(\cdot;a,b,c,d;b)$ as in the flow line coupling of Theorem~\ref{thm::flowcoupling}. Given $\gamma$, let $\Omega_L$ and $\Omega_R$ be the two domains on either side of $\gamma$, as in Theorem~\ref{thm::flowcoupling}. Denote by $\cQ_L$ (resp. $\cQ_R$) the points of $\cQ$ that fall in $\Omega_L$ (resp. $\Omega_R$). Then, we sample $\big\{\gamma_{z_i^L}\big\}_{i\ge 1}$ and sample $\big\{\gamma^*_{z_i^R}\big\}_{i\ge 1}$ as in  Theorem~\ref{thm::intcoupling}, applied in $\Omega_L$ and $\Omega_R$ respectively. Then $\cT$, viewed as a random variable in the space $\cH$, is simply the closure of the union of $\big\{\gamma_{z_i^L}\big\}_{i\ge 1}$, $\big\{\gamma^*_{z_i^R}\big\}_{i\ge 1}$, and $\gamma$.
\end{theorem}

It is clear that Theorem \ref{T:multi} follows from Theorem \ref{thm::treeconv} and the counterflow line coupling of Theorem \ref{thm::countercoupling}. The rest of Section \ref{subsec::coupling} is devoted to a proof of Theorem \ref{thm::treeconv}.
We will frequently use the following discrete Beurling estimate for random walk on $\Omega_\delta$ the reflecting boundaries $(b_\delta c_\delta)$ and $(d_\delta a_\delta)$.
\begin{lemma}\label{lem::Beur}
There exists constants $C>0, \delta_0>0, 0<\eps<1$ such that the following estimate hold. 
For every $\delta\le r\le\delta_0$, for every connected subgraph $A$ of $\Omega_\delta$ with $\diam(A)>r$ and for every $z_\delta\in \Omega_\delta$ with $d= d(z_\delta, A\cup(a_\delta b_\delta)\cup(c_\delta d_\delta))\le r$, we have
\[
\PP\left[\LR\text{ hits }\partial B(z_\delta,C r) \text{ before }A\cup(a_\delta b_\delta)\cup(c_\delta d_\delta)\right]\le C(d/r)^\epsilon,
\]
where $\LR$ is the random walk starting from $z_\delta$, reflecting at $(b_\delta c_\delta)\cup(d_\delta a_\delta)$, and stopped when it hits $(a_\delta b_\delta)\cup(c_\delta d_\delta)$. 
\end{lemma}
\begin{proof}
If $\Omega_\delta$ is a subgraph of $\delta\Z^2$, the proof is same as the proof of~\cite[Lemma 11.2]{SchrammFirstSLE}. To generalize the discrete Beurling estimate
to the general setting mentioned in Section~\ref{SS:gen}, we will prove the result based on the assumptions there. 
To prove it, it suffices to prove the result in the case $d = r$ and the right hand side an absolute constant strictly less than 1, say $\epsilon$. 
Hence we may assume that $d(z_\delta, A\cup(a_\delta b_\delta)\cup(c_\delta d_\delta))=r$.

We will proceed by contradiction. Suppose for a sequence of $\{\eps_n,\delta_n,r_{n}, C_n\}$ and $\{\Omega_{\delta_n},z_{\delta_n}, A_{\delta_n}\}_{n\ge 1}$ such that $\delta_n\to 0$, $\eps_n\to 0$, $C_n\to \infty$, and $\diam(A_{\delta_n})>r_n$; but
\begin{equation}\label{eqn::au1}
\PP\left[\LR\text{ hits }\partial B(z_{\delta_n},C_n r_n)\text{ before }A_{\delta_n}\cup(a_{\delta_n} b_{\delta_n})\cup(c_{\delta_n} d_{\delta_n})\right]>1-\eps_n.
\end{equation}
We may assume that $r_{n}$ tends to zero. (Otherwise, if $C_{n}>2\diam(\Omega)$, then~\eqref{eqn::au1} trivially fails.) We denote $A_{\delta_n}\cup(a_{\delta_n} b_{\delta_n})\cup(c_{\delta_n} d_{\delta_n})$ by $\overline A_{\delta_n}$.

Since $\partial\Omega$ is $C^1$ and simple, if we denote this curve by $l:[0,1]\to \C$, there exists a constant $M>0$, such that for $0\le s<t<1$, we have
\begin{equation}\label{eqn::au2}
M^{-1}\le\frac{|l(t)-l(s)|}{(t-s)\wedge(1-t+s)}\le M.
\end{equation}
We denote by $w_{\delta_n}\in \overline A_{\delta_n}$ a point such that 
\[d(z_{\delta_n},w_{\delta_n})=d (z_{\delta_n},\overline A_{\delta_n}).\]
 If $z_{\delta_n}$ and $\overline A_{\delta_n}$ is separated by $l$, then there exist $s<t$ such that $l(s)$ and $l(t)$ belong to the straight line connecting $z_{\delta_n}$ and $w_{\delta_n}$ and $l[s,t]$ hits $\partial B(z_{\delta_n},2M^2r_n)$. In particular, we have
\[d(l(s),l(t))\le r_{n}\quad{\text{ and }}\diam(l[s,t])\ge M^2r_{n},\]
which is a contradiction with~\eqref{eqn::au2}.

Suppose first there exists $M'>0$, such that $\frac{ r_n }{\delta_n}\le M'$. In $B(z_{\delta_n}, 4M^2r_{n})$, by our assumptions~\eqref{Bounded} and~\eqref{irreducibility}, the probability that the random walk starting from $z_{\delta_n}$ hits $\overline A_{\delta_n}$ before $\partial B(z_{\delta_n},C_nr_n)$ is uniformly bounded away from zero. For $n$ large enough, we have $C_n>4M^2$, and this leads to a contradiction in conjunction with~\eqref{eqn::au1}.

Now suppose instead that $\frac{\delta_n}{r_{n}}$ tends to zero. By translation, we may assume $z_{\delta_n}=0$ for $n\ge 1$. We will consider 
$$\left\{\frac{1}{r_n}\left(\Omega_{\delta_n}\cap B(z_{\delta_n}, 4M^2r_n)\right),\frac{1}{r_n}(\overline A_{\delta_n}\cap  B(z_{\delta_n}, 4M^2r_n))\right\}_{n\ge 1}.$$
 There are two cases:
\begin{itemize}
\item
In the first case, if the free boundary $(b_{\delta_n} c_{\delta_n})\cup (d_{\delta_n} a_{\delta_n})$ does not intersect $B(z_{\delta_n},4M^2r_n)$, then we can apply directly the usual Beurling estimate (see e.g., \cite[Lemma 4.3]{BLRdimers}), hence we  derive a contradiction directly.
\item
In the second case, we may assume that the free boundary intersect $B(z_{\delta_n},4M^2r_n)$ for all $n
\ge 1$. Then, as $n\to\infty$, the intersection of $B(z_{\delta_n},4M^2r_n)\cap\partial\Omega_{\delta_n}$ converges to a straight line (under the metric~\eqref{eqn::curves_metric}) since the arc is $C^1$. (Recall that under our assumption $\partial\Omega_\delta=\partial\Omega$ in the setup of Section~\ref{SS:gen}.) Thus, without loss of generality, we may assume that $\frac{1}{r_n}\left(\Omega_{\delta_n}\cap B(z_{\delta_n}, 4M^2r_n)\right)$ converges to one component of $B(0,4M^2)\setminus l$ (in the sense that the boundary converges under the metric~\eqref{eqn::curves_metric}), where $l$ is a straight segment. Moreover, we will assume $\frac{1}{r_n}(\overline A_{\delta_n}\cap  B(z_{\delta_n}, 4M^2r_n))$ converges to a compact set $\overline A$ under the Hausdorff metric. Note that $l$ does not separate $\overline A$ and $z=0$. In general the segment $l$ might contain both some wired and reflecting parts. 
Without loss of generality we only consider the case where $l$ consists exclusively of the limit of a free boundary arc (in fact the argument is only easier if $l$ also contains some wired portion). 

 For $v_{\delta_n}\in V(\Omega_{\delta_n})$, We define 
\[h_{\delta_n}(v_{\delta_n}):=\PP\left[\LR\text{ hits }\partial B(z_{\delta_n},C_nr_n)\text{ before }\overline A_{\delta_n}\right],\] 
where $\LR$ is the random walk starting from $v_{\delta_n}$. 
By our assumption~\eqref{InvP}, $h_{\delta_n}$ converges to a harmonic function $h$ uniformly on $\frac{1}{r_n}\left(\Omega_{\delta_n}\cap B(z_{\delta_n}, 4M^2r_n)\right)$. Here $h$ is the unique bounded harmonic function with the following boundary conditions: $h=1$ on $\overline A$ and $h=0$ on $\partial B(0,4M^2)$, and $\partial_n h=0$ on $l$. By the minimum principle, we know that $h(0)>0$ (note that here we allow $0\in l$) and hence $h_{\delta_n}(z_{\delta_n})$ is uniformly bounded away from zero. Since for $n$ large enough, we have $C_n>4M^2$, this leads to a contradiction in conjunction with ~\eqref{eqn::au1}.
\end{itemize} 
This completes the proof of the lemma.
\end{proof}

\subsection{Convergence of $\gamma_\delta$ and discrete Peano curves}
\label{subsec::middle}

In this subsection, we will prove the convergence of $\{\gamma_{\delta}\}_{\delta>0}$ under the metric~\eqref{eqn::curves_metric} and discuss the convergence of the discrete Peano curves $\{\eta_L^\delta\}_{\delta>0}$ and $\{\eta_R^\delta\}_{\delta>0}$. Similar results have been proved in~\cite{HanLiuWuUST}, but the setup is different. In that paper, the authors considered the $\ust$ without conditioning. The line of the proof is almost the same and we will point out how to modify the proof there to the setting in this paper. We denote by $K$ the \textbf{conformal modulus} of $(\Omega;a,b,c,d)$ and denote by $\varphi$ the conformal map from $(\Omega;a,b,c,d)$ onto the rectangle $(0,1)\times (0,K)$ which sends $(a,b,c,d)$ to $(0,1,1+\ii K,\ii K)$. 

We will deal with the convergence of $\{\gamma_{\delta}\}_{\delta>0}$ at first. First of all, we need the following lemmas to get rid of the conditioning. Suppose $A$ is a circle such that $b \notin A$ and $A\cap \partial \Omega\neq\emptyset$. We denote by $\Omega_A$ the connected component of $\Omega\setminus A$ whose boundary contains $b$. We denote by $A_\delta$ the corresponding discrete approximations on $\Omega_\delta$ and denote by $\Omega_{\delta,A_\delta}$ the connected component of $\Omega_\delta\setminus A_\delta$ which contains $b_\delta$. Define
\[h_{A_\delta}(v):=\PP[\LR\text{ starting from }v \text{ hits }A_\delta\text{ before exits }\Omega_\delta],\]
where $\LR$ has the same law as the random walk reflecting at $(b_\delta c_\delta)\cup(d_\delta a_\delta)$, and ending at $(a_\delta b_\delta)\cup(c_\delta d_\delta)$. 
Denote by $h_A$ the unique bounded harmonic function on $\Omega_A$ with the following boundary conditions: $h_A$ equals one on $A$ and equals zero on $((ab)\cup(cd))\cap\partial \Omega_A$ and $\partial_{n}(h_A\circ\varphi^{-1})$ equals zero on \[([0,\ii K]\cup [1,1+\ii K])\cap\partial\varphi(\Omega_A),\]
 where $n$ is the outer normal, and $\varphi$ is the unique conformal from $\Omega$ to a rectangle $(0,1) \times (0,K)$ (when we identify $\C$ with $\R^2$) sending $a,b,c,d$ to the four corners of the rectangle in counterclockwise order starting from 0.
\begin{lemma}\label{lem::conv_circ}
For every $r>0$, we define $\Omega_{A,r}$ to be $\Omega_A\setminus\{z:d(z,A\cap((ab)\cup(cd)))<r\}$
and define $\Omega_{A,r,\delta}$ to be the discrete approximation of $\Omega_{A,r}$ on $\Omega_{\delta}$. Then, 
$h_{A_\delta}|_{\Omega_{A,r,\delta}}$ converges to $h_A|_{\Omega_{A,r}}$ uniformly.
\end{lemma}
\begin{proof} 
For the setup considered in Section~\ref{SS:gen}, we simply observe that random walk $\cR$ on $\Omega_\delta$ will converge to Brownian motion orthogonally reflected on $(bc)$ and $(da)$ and killed on $(ab)$ and $(cd)$ (orthogonal reflection is defined e.g. by doing the reflection in the rectangle and applying the conformal map $\varphi^{-1}$).

We also give a proof when we are considering the lattice $\Z^2$. We first show the equicontinuity of $\{h_{A_\delta}|_{\Omega_{A,r,\delta}}\}_{\delta>0}$. We fix $0<\eps < r$ sufficiently small, and fix $w_\delta, z_\delta\in \Omega_{A,r,\delta}$ such that $|w_\delta-z_\delta|<\eps$. There are three cases:
\begin{itemize}
\item
If $d(z_\delta,A_\delta)\le \sqrt\epsilon$, then by Lemma~\ref{lem::Beur}, there exist constants $C>0$ and $c>0$ such that 
\[1-h_{A_\delta}(z_\delta)\le C\epsilon^c\quad\text{ and }\quad 1-h_{A_\delta}(w_\delta)\le C\epsilon^c.\]
In this case, we have 
\[|h_{A_\delta}(w_\delta)-h_{A_\delta}(z_\delta)|\le C\epsilon^c.\]
\item
If $d(z_\delta,(a_\delta b_\delta)\cup(c_\delta d_\delta))\le \sqrt\epsilon$, by the similar argument above, we also have
\[|h_{A_\delta}(w_\delta)-h_{A_\delta}(z_\delta)|\le C\epsilon^c.\]
\item
If $d(z_\delta,(a_\delta b_\delta)\cup(c_\delta d_\delta)\cup A_\delta)> \sqrt\epsilon$, from the discrete harmonicity of $h_{A_\delta}$, we can draw a discrete path $l_\delta$ from $z_\delta$ to $A_\delta$, such that for every $v_\delta\in l_\delta$, we have $h_{A_\delta}(v_\delta)\ge h_{A_\delta}(z_\delta)$. Thus, by Lemma~\ref{lem::Beur}, there exist constants $C>0$ and $c>0$, we have
\[h_{A_\delta}(w_\delta)\ge \left(1-C\epsilon^c\right) h_{A_\delta}(z_\delta).\]
This implies that
\[
h_{A_\delta}(z_\delta)-h_{A_\delta}(w_\delta)\ge-C\epsilon^c.
\]
By symmetry, we have
\[|h_{A_\delta}(z_\delta)-h_{A_\delta}(w_\delta)|\le C\epsilon^c.\]
\end{itemize}
The estimates in these three cases give the equicontinuity of $\{h_{A_\delta}|_{\Omega_{A,r,\delta}}\}_{\delta>0}$.

To identify uniquely the limit, 
we can use argument similar to~\cite[Lemma B.2]{HanLiuWuUST}.
\end{proof}
We also need a similar result when $A$ is a subarc of $(cd)$. Suppose $(\tilde c\tilde d)$ is a subarc of $(cd)$ and we denote by $(\tilde c_\delta \tilde d_\delta)$ a discrete approximation of $(\tilde c\tilde d)$ on $\Omega_\delta$. Define
\begin{equation}\label{E:hcd}
h_{(\tilde c_\delta \tilde d_\delta)}(v):=\PP[\LR\text{ starting from }v \text{ hits }(a_\delta b_\delta)\cup(c_\delta d_\delta)\text{ at }(\tilde c_\delta\tilde d_\delta)],
\end{equation}
where $\LR$ has the same law as the random walk reflecting at $(b_\delta c_\delta)\cup(d_\delta a_\delta)$, and ending at $(a_\delta b_\delta)\cup(c_\delta d_\delta)$.

Denote by $h_{(\tilde c\tilde d)}$ the unique bounded harmonic function on $\Omega$ with the following boundary conditions: $h_{(\tilde c\tilde d)}$ equals one on $(\tilde c\tilde d)$ and equals zero on $(ab)\cup(c\tilde c)\cup(\tilde d d)$ and $\partial_{n}(h_{(\tilde c\tilde d)}\circ\varphi^{-1})$ equals zero on $[0,\ii K]\cup [1,1+\ii K]$, where $n$ is the outer normal. It is straightforward to check that $h_{(\tilde c\tilde d)}$ is nowhere zero on the arc $(bc)$ or $(da)$, i.e., $ h_{(\tilde c\tilde d)} (z) >0$ for $ z \in (bc) \cup (da)$.

\begin{lemma}\label{lem::conv_subarc}
For every $r>0$, we define the domain 
\[\Omega_{\tilde c,\tilde d,r}:= \Omega\setminus\{z:d(z,\{\tilde c,\tilde d\})<r\}\]
and define $\Omega_{\tilde c,\tilde d,r,\delta}$ to be the discrete approximation of $\Omega_{\tilde c,\tilde d,r}$ on $\Omega_{\delta}$. Then, the discrete harmonic function $h_{(\tilde c_\delta \tilde d_\delta)}|_{\Omega_{\tilde c,\tilde d,r,\delta}}$ converges to $h_{(\tilde c\tilde d)}|_{\Omega_{\tilde c,\tilde d,r}}$ uniformly.
\end{lemma}
\begin{proof}
The proof is similar to the proof of Lemma~\ref{lem::conv_circ}. 
\end{proof}

\begin{lemma}\label{lem::ratio}
We assume $A$ is subarc of $(cd)$ or $A$ is a circle such that $A\cap\partial \Omega\neq\emptyset$. Then, we have
\[\lim_{\delta\to 0}\frac{h_{A_\delta}(b^+_\delta)}{h_{(c_\delta d_\delta)}(b^+_\delta)}=\frac{\partial_n( h_A\circ\varphi^{-1})(b)}{\partial_n (h_{(cd)}\circ\varphi^{-1})(b)},\]
where $n$ is the outer normal perpendicular to $(0,1)$.
\end{lemma}
\begin{proof}
We define the function $g$ on a (subset of the) rectangle as follows:
\begin{align} 
g(z) :=
\begin{cases}
\frac{h_A\circ\varphi^{-1}(z)}{h_{(cd)}\circ\varphi^{-1}(z)}&\text{ if }z\in \varphi(\Omega_A)\\
\frac{\partial_n h_A\circ\varphi^{-1}(z)}{\partial_n h_{(cd)}\circ\varphi^{-1}(z)},&\text{ if }z\in (0,1)\cap\partial\varphi(\Omega_A).
\end{cases}
\end{align}
Note that $g\circ\varphi$ is continuous near $b$. Now, we consider a circle $\partial B(b,r)$  and we denote by $l_{b,r}$ the arc of that circle separating $b$ and $(cd)$, and denote by $x:=l_{b,r}\cap(ab)$. We denote by $l_{b,r,\delta}$ and $x_\delta$ some discrete approximation on $\Omega_\delta$ of $l_{b,r}$ and $x$ respectively. For every $\eps>0$, we may choose $r$ small enough, such that for every $w\in B(b,r)\cap\Omega$, we have
\[|g\circ\varphi(w)-g\circ\varphi(b)|\le \eps.\]
By uniform convergence in Lemma~\ref{lem::conv_circ} and Lemma~\ref{lem::conv_subarc}, combining with~\cite[Corollary 3.8]{ChelkakWanMassiveLERW}, there exists $\eta>0$, such that if $v_\delta \in \Omega_\delta$ with $|v_\delta-x_\delta|<\eta$, then we have
\[\left|\frac{h_{A_\delta}(v_\delta)}{h_{(c_\delta d_\delta)}(v_\delta)}-g\circ\varphi(x)\right|\le\epsilon,\quad\text{ for }\delta\text{ small enough}.\]
By the uniform convergence in Lemma~\ref{lem::conv_circ} and Lemma~\ref{lem::conv_subarc} we have $\frac{h_{A_\delta}}{h_{(c_\delta d_\delta)}}\big|_{l_{b,r,\delta}\cap B^c(x_\delta,\eta)}$ converges to $g\circ\varphi|_{l_{b,r}\cap B^c(x,\eta)}$ uniformly (by positivity of $h_{(\tilde c\tilde d)}$ on $l_{b,r,}\cap B^c(x,\eta)$). Thus, we have
\begin{equation}\label{eqn::conv}
\left|\frac{h_{A_\delta}(v_\delta)}{h_{(c_\delta d_\delta)}(v_\delta)}-g\circ\varphi(b)\right|\le\epsilon,\quad\text{ for all }v_\delta\in l_{b,r,\delta}.
\end{equation}
Moreover, we can assume without loss of generality that $r$ is small enough that $l_{b,r,\delta}$ seperates $A_\delta$ and $b_\delta$ in $\Omega_\delta$. We denote by $\tau_\delta(l_{b,r,\delta})$ the hitting time of random walk $\LR_\delta$ at $l_{b,r,\delta}$ and denote by $\tau_\delta((a_\delta b_\delta))$ and $\tau_\delta(A_\delta)$ similarly. Note that
\begin{align*}
\frac{h_{A_\delta}(b^+_\delta)}{h_{(c_\delta d_\delta)}(b^+_\delta)}&=\frac{\PP_{b^+_\delta}[\one_{\tau_\delta(l_{b,r,\delta})<\tau_\delta(A_\delta)}h_{A_\delta}(\tau_\delta(l_{b,r,\delta}))]}{\PP_{b^+_\delta}[\one_{\tau_\delta(l_{b,r,\delta})<\tau_\delta((c_\delta d_\delta))}h_{(c_\delta d_\delta)}(\tau_\delta(l_{b,r,\delta}))]}\\
&=\frac{\PP_{b^+_\delta}\left[\one_{\tau_\delta(l_{b,r,\delta})<\tau_\delta(A_\delta)}\frac{h_{A_\delta}(\tau_\delta(l_{b,r,\delta}))}{h_{(c_\delta d_\delta)}(\tau_\delta(l_{b,r,\delta}))}h_{(c_\delta d_\delta)}(\tau_\delta(l_{b,r,\delta}))\right]}{\PP_{b^+_\delta}[\one_{\tau_\delta(l_{b,r,\delta})<\tau_\delta((c_\delta d_\delta))}h_{(c_\delta d_\delta)}(\tau_\delta(l_{b,r,\delta}))]}.
\end{align*}
Combining with~\eqref{eqn::conv}, we have
\[\lim_{\delta\to 0}\frac{h_{A_\delta}(b^+_\delta)}{h_{(c_\delta d_\delta)}(b^+_\delta)}=\frac{\partial_n h_A\circ\varphi^{-1}(b)}{\partial_n h_{(cd)}\circ\varphi^{-1}(b)}.\]
This completes the proof. 
\end{proof}
\begin{remark}
Though~\cite[Corollary 3.8]{ChelkakWanMassiveLERW} is stated for the square lattice, the proof is based on ~\cite{ChelkakRobustComplexAnalysis}. The assumptions on the graph is much weaker in~\cite{ChelkakRobustComplexAnalysis} and we have included them in our assumptions in Section~\ref{SS:gen}.
\end{remark}
We now explain how this can be used to prove the convergence of $\{\gamma_{\delta}\}_{\delta>0}$. With Lemma~\ref{lem::ratio}, it is relatively easy to modify the proof in~\cite{HanLiuWuUST} to this setting. We will now explain it in detail. 
The tightness of $\{\gamma_{\delta}\}_{\delta>0}$ will be a consequence of the following Lemma, which is similar to~\cite[Theorem~3.9]{LawlerSchrammWernerLERWUST}. See also~\cite[Lemma~C.5]{HanLiuWuUST} where a similar argument is used. Recall that $\LR_\delta$ denotes the random walk in $\Omega_\delta$ with reflecting boundary conditions on $(b_\delta c_\delta)\cup(d_\delta a_\delta)$ and killed on $(a_\delta b_\delta)\cup(c_\delta d_\delta)$. Denote by $\PP_{v_\delta}$ the law of the random walk starting from $v_\delta$, for every $v_\delta\in \Omega_\delta$. 

\begin{lemma}\label{lem::simple}
Let $z_0\in\Omega$. For any $\alpha>0$ and $0<\beta\le (1/4)\dist((ab),(cd))$, let $\LA_{\delta}(z_0,\beta,\alpha;\gamma_{\delta})$ denote the event that there are two points $o_1,o_2\in\gamma_{\delta}$ with
$o_1,o_2\in B(z_0,\beta/4)$ with $|o_1 - o_2| \le \alpha$, and such that the subarc of $\gamma_{\delta}$ between $o_1$ and $o_2$ is not contained in $B(z_0,\beta)$. For every $\epsilon>0$ and every $\beta>0$, we can choose $\alpha=\alpha(\beta,\epsilon)$, small enough  that for every $\delta>0$, we have
\[\PP[\LA_{\delta}(z_{0},\beta,\alpha;\gamma_{\delta})]\le\eps.\]
\end{lemma}

\begin{proof}
Denote by $B_\delta(z_0,\beta)$ the discrete approximation of $B(z_0,\beta)$. By Wilson's algorithm, we know that the law of $\gamma_{\delta}$ is same as the law of loop-erasure of the random walk $\LR_\delta$, starting from $b_\delta^+$, conditional on hitting $(c_\delta d_\delta)$ before $(a_\delta b_\delta)$. We will use the same notations in~\cite[Lemma~C.5]{HanLiuWuUST}. Let $t_0=0$. For $j\ge 1$, define inductively \[s_j:=\inf\left\{r\ge t_{j-1}:\LR_\delta\in B\left(z_0,\beta/4\right)\right\}\quad \text{and}\quad t_j:=\inf\{r\ge s_{j}:\LR_\delta\notin B(z_0,\beta)\}.\] 
Let $\tau_\delta(c_\delta d_\delta)$ be the hitting time of $\LR_\delta$ at $(c_\delta d_\delta)$. For every $r>0$, define $\text{LE}(\LR_\delta[0,r])$ to be the loop erasure of $\LR_\delta[0,r]$. Define the event $\LA_{\delta}(z_0,\beta,\alpha;\text{LE}(\LR_\delta[0,t_j]))$ similarly as $\LA_{\delta}(z_0,\beta,\alpha;\gamma_{\delta})$ and define the event $T_j:=\{t_j\le\tau_\delta(c_\delta d_\delta)\}$ that there are $j$ crossings. Denote $\LA_{\delta}(z_0,\beta,\alpha;\text{LE}(\LR_\delta[0,t_j]))$ by $Y_j$ for simplicity. Note that $Y_1=\emptyset$ and 
\[\LA_{\delta}(z_0,\beta,\alpha;\gamma_{\delta})\subset\{\cup_{j=2}^{\infty}(Y_j\cap T_j)\}\subset \{T_{m+1}\cup (\cup_{j=2}^{m}Y_j)\},\quad\text{ for every }m>1.\]
Thus, we only need to control 
\[\PP_{b_\delta^+}[\cup_{j=1}^{m}Y_j\cond\LR_\delta\text{ hits }(c_\delta d_\delta)\text{ before }(a_\delta b_\delta)]\text{ and }\PP_{b_\delta^+}[T_{m+1}\cond\LR_\delta\text{ hits }(c_\delta d_\delta)\text{ before }(a_\delta b_\delta)].\]
For the first estimate, we have
\begin{align*}
&\PP_{b_\delta^+}[\cup_{j=1}^{m}Y_j\cond\LR_\delta\text{ hits }(c_\delta d_\delta)\text{ before }(a_\delta b_\delta)]\\
=&\frac{\PP_{b_\delta^+}[\one_{\{\LR_\delta\text{ hits }\partial B_{\delta}(z_0,\beta/4)\}}\PP_{\LR_\delta(s_1)}[\{\cup_{j=1}^{m}Y_j\}\cap\{\LR_\delta\text{ hits }(c_\delta d_\delta)\text{ before }(a_\delta b_\delta)\}]]}{h_{(c_\delta d_\delta)}(b_\delta^+)}\\
\le &\frac{h_{\partial B_\delta(z_0,\beta/4)}(b_\delta^+)}{h_{(c_\delta d_\delta)}(b_\delta^+)}\left(\max_{v_\delta\in\partial B_\delta(z_0,\beta/4)}\PP_{v_\delta}[\cup_{j=1}^{m}Y_j]\right).
\end{align*}
By the Beurling estimate of Lemma \ref{lem::Beur}, there exist $C>0$ and $k>0$, such that
\[\max_{v_\delta\in\partial B_\delta(z_0,\beta/4)}\PP_{v_\delta}[\cup_{j=1}^{m}Y_j]\le Cm^2\left(\frac{\alpha}{\beta}\right)^k.\]
If $b_\delta^+\in B(z_0,\beta/4)$, then combining with Lemma~\ref{lem::ratio}, there exists $C>0$, such that 
\[\frac{h_{\partial B_\delta(z_0,\beta/4)}(b_\delta^+)}{h_{(c_\delta d_\delta)}(b_\delta^+)}\le C.\]
If $b_\delta^+\notin B(z_0,\beta/4)$, then if we denote by $d=d(b_\delta^+,\partial B_\delta(z_0,\beta/4))$, we have
\[\frac{h_{\partial B_\delta(z_0,\beta/4)}(b_\delta^+)}{h_{(c_\delta d_\delta)}(b_\delta^+)}\le \frac{h_{\partial B_\delta(b_\delta^+,d)}(b_\delta^+)}{h_{(c_\delta d_\delta)}(b_\delta^+)}\le C,\]
where the last inequality is from Lemma~\ref{lem::ratio}.
Thus, we have
\[\PP_{b_\delta^+}[\cup_{j=1}^{m}Y_j\cond\LR_\delta\text{ hits }(c_\delta d_\delta)\text{ before }(a_\delta b_\delta)]\le Cm^2 \left(\frac{\alpha}{\beta}\right)^k.\]
For the second estimate, we have 
\begin{align*}
&\PP_{b_\delta^+}[T_{m+1}\cond\LR_\delta\text{ hits }(c_\delta d_\delta)\text{ before }(a_\delta b_\delta)]\\
=&\frac{\PP_{b_\delta^+}[\one_{\{\LR_\delta\text{ hits }\partial B_{\delta}(z_0,\beta)\}}\PP_{\LR_\delta(t_1)}[\{t_j\le\tau\}\cap\{\LR_\delta\text{ hits }(c_\delta d_\delta)\text{ before }(a_\delta b_\delta)\}]]}{h_{(c_\delta d_\delta)}(b_\delta^+)}\\
\le &\frac{h_{\partial B_\delta(z_0,\beta)}(b_\delta^+)}{h_{(c_\delta d_\delta)}(b_\delta^+)}\left(\max_{v_\delta\in\partial B_\delta(z_0,\beta)}\PP_{v_\delta}[\LR_\delta\text{ hits }\partial B_\delta(z_0,\beta/4)\text{ before }(a_\delta b_\delta)\cup(c_\delta d_\delta)]\right)^j.
\end{align*}
By the same argument in~\cite[Equation~C.2]{HanLiuWuUST} (or by using a uniform crossing estimate), we have that there exists $0<q = q(\beta)<1$ such that for every $\delta>0$, 
\[
\min_{v_\delta\in\partial B_\delta(z_0,\beta)}\PP_{v_\delta}[\LR_\delta\text{ hits }(a_\delta b_\delta)\cup(c_\delta d_\delta)\text{ before }\partial B_\delta(z_0,\beta/4)]\ge q.
\]
By the same argument in the first estimate, we have that there exists $C>0$, such that
\[
\PP_{b_\delta^+}[T_{m+1}\cond\LR_\delta\text{ hits }(c_\delta d_\delta)\text{ before }(a_\delta b_\delta)]\le C(1-q)^j.
\]
(See also Lemma 4.13 in \cite{BLRdimers} for a similar estimates). Thus, for every $\eps>0$, by first choosing $m$ large enough and then choosing appropriate $\alpha=\alpha(\beta,\epsilon)$, we have
\[
\PP[\LA_{\delta}(z_0,\beta,\alpha;\gamma_{\delta})]\le\eps.
\]
This completes the proof.
\end{proof}

By Lemma~\ref{lem::simple}, we can deduce the tightness of $\{\gamma_{\delta}\}_{\delta>0}$.
\begin{lemma}\label{lem::tightness}
The sequence $\{\gamma_{\delta}\}_{\delta>0}$ is tight. Moreover, any subsequential limit is a simple curve in $\overline\Omega$ which intersects $[ab]\cup[cd]$ only at two ends. 
\end{lemma}
\begin{proof}
The tightness of $\gamma_{\delta}$ follows from Lemma \ref{lem::simple}
using the ingenious argument given in \cite[Lemma 3.12]{LawlerSchrammWernerLERWUST}: essentially, a family of curves without bubbles forms a compact set. See also the proof of~\cite[Lemma C.2]{HanLiuWuUST} where this argument is also discussed. 
Since the proof of tightness is identical, we do not include here. 

Thus let us suppose that $\gamma$ is any subsequential limit. Without loss of generality and with a slight abuse of notation, we will use $\{\gamma_{\delta}\}_{\delta>0}$ for the subsequence converging to $\gamma$.
 We begin with the proof that $\gamma$ intersects $(ab) \cup (cd)$ at exactly two ends. 
 For this we modify the argument in~\cite[Lemma C.2]{HanLiuWuUST}. It suffices to show that almost surely $\gamma$ intersects $[ab]$ at one end, since the proof for $[cd]$ is identical.

For every $s>0$, we define $\tau_{s,\delta}$ to be the first time that $\gamma_{\delta}$ hits $\partial B_\delta(b^+_\delta,s)$ and define 
\[
\tau_{s,\delta}(\eps):=\inf\{t\ge\tau_{s,\delta}:d(\gamma_{\delta}(t),(a_\delta b_\delta))<\eps\}.
\]
Note that
\begin{align*}
\PP[\tau_{s,\delta}(\eps)<\infty]\le&\frac{\PP_{b_\delta^+}[\LR_\delta\text{ hits }\partial B_\delta(b_\delta^+,s)\text{ before }(a_\delta b_\delta)]}{\PP_{b_\delta^+}[\LR_\delta\text{ hits }(c_\delta d_\delta)\text{ before }(a_\delta b_\delta)]}\\
&\times \max_{v_\delta\in\partial B_\delta(b_\delta^+,s)}\PP_{v_\delta}[\LR_\delta\text{ hits the }\eps\text{-neighbour of }(a_\delta b_\delta)\text{ and hits }(c_\delta d_\delta)\text{ before }(a_\delta b_\delta)]\\
=&\frac{h_{\partial B_\delta(b_\delta^+,s)}(b_\delta^+)}{h_{(c_\delta d_\delta)}(b_\delta^+)}\\
&\times \max_{v_\delta\in\partial B_\delta(b_\delta^+,s)}\PP_{v_\delta}[\LR_\delta\text{ hits the }\eps\text{-neighbour of }(a_\delta b_\delta)\text{ and hits }(c_\delta d_\delta)\text{ before }(a_\delta b_\delta)].
\end{align*}
Combining with the Beurling estimate of Lemma~\ref{lem::Beur} and Lemma~\ref{lem::ratio} on the limiting behaviour of the ratio of harmonic functions on the right hand side, we have that there exist $C = C(s)>0$ and $c>0$, such that
\[\PP[\tau_{s,\delta}\le\tau_{s,\delta}(\eps)<\infty]\le C\epsilon^c.\]
Define $\tau_s$ and $\tau_s(\eps)$ for $\gamma$ similarly as defining $\tau_{s,\delta}$ and $\tau_{s,\delta}(\eps)$ for $\gamma_{\delta}$. By letting $\delta\to 0$, we have
\[\PP[\tau_{s}\le\tau_{s}(\eps)<\infty]\le C\epsilon^c.\]
By letting $\eps\to 0$, we know that almost surely, $\gamma$ will never hit $(ab)$ after the time $\tau_s$. By letting $s\to 0$, we know that almost surely, $\gamma$ intersects $[ab]$ only at one end. This completes the proof.
\end{proof}
Now, we begin to identify the limiting curve. We fix a subsequential limit $\gamma$ and we suppose $\{\gamma_{\delta_n}\}_{n\ge 0}$ converge to $\gamma$ in law. We parametrise $\gamma_{\delta_n}$ and $\gamma$ by the time interval $[0,1]$. We will show that $\gamma$ is a chordal $\SLE_2(-1,-1;-1,-1)$ curve in $\Omega$ from $b$ to $(cd)$ with the marked points given by $(a,d;b,c)$ (recall this curve is target invariant, so we do not need to specify the target point of $\gamma$ in $\Omega$, and just stop it upon hitting the arc $(cd)$).

The proof is almost the same as the proof of~\cite[Theorem 1.6]{HanLiuWuUST}. The main difference is that because of the conditioning, we can not deduce that $\gamma\cap((bc)\cup(da))=\emptyset$ at first. We will explain how to modify the proof in our setting in detail. 
\begin{proposition}\label{prop::middle}
The discrete curves $\{\gamma_{\delta}\}_{\delta>0}$ converges to $\SLE_2(-1,-1;-1,-1)$ curve from $b$ to $(cd)$ with the marked points $(a,d;b,c)$ in law, when $\delta\to 0$.
\end{proposition}
\begin{proof}
For every $s>0$, we define $\tau_n(s):=\inf\{t>0:d(\gamma_{\delta_n}(t),b_{\delta_n}^+)\ge s\}$ and define $\tau(s):=\inf\{t>0:d(\gamma(t),b)\ge s\}$. We may couple $\{\gamma_{\delta_n}\}_{n\ge 0}$ and $\gamma$ together such that $\gamma_{\delta_n}$ converges to $\gamma$ almost surely. By considering the continuous modification, we may also  assume that $\tau_n(s)$ converges to $\tau(s)$ almost surely. See for instance~\cite{KarrilaMultipleSLELocalGlobal} for details. 

First of all, we will prove that $\gamma(\tau(s))$ does not belong to $(bc)$ almost surely. Denote by $l_s$ the connected component of $\partial B(b,s)\cap \Omega$ seperating $b$ and $(cd)$. We define $y:=l_s\cap(bc)$ and define $y_{\delta_n}$ to be a discrete approximation of $y$ on $(b_{\delta_n}c_{\delta_n})$. For every $\eps>0$, we have that
\[\PP[\gamma_{\delta_n}\text{ hits }\partial B_\delta(y_\delta,2\eps)]\le\frac{h_{\partial B_{\delta_n}(y_{\delta_n},2\eps)}(b_{\delta_n}^+)}{h_{(c_{\delta_n}d_{\delta_n})}(b_{\delta_n}^+)}.\]
By letting $n\to\infty$, by Lemma~\ref{lem::ratio}, we have that
\[\PP[y\in\gamma]\le\PP[\gamma\text{ hits }\partial B(y,\eps)]\le \frac{\partial_n (h_{\partial B(y,2\eps)}\circ\varphi^{-1})(b)}{\partial_n( h_{(cd)}\circ\varphi^{-1})(b)},\]
where $\varphi$ is the conformal map from $\Omega$ onto the rectangle.
By letting $\eps\to 0$, we have
\begin{equation}\label{eqn::aux}
\PP[y\in\gamma]=0.
\end{equation}
This implies that $\gamma(\tau(s))$ does not belong to $(bc)$ almost surely.

Second, we define $T(\eps):=\inf\{t>\tau(s):d(\gamma(\tau(s)),\partial \Omega)\le\eps\}$. We denote by $\Omega(t)$ the connected component of $\Omega\setminus\gamma[0,(t+\tau(s))\wedge T(\epsilon)]$ which contains $d$ and denote by $b(0)$ the last hitting point of $\gamma[0,(t+\tau(s))\wedge T(\eps)]$ on $[bc]$ (if there is one). We define $T_n(\eps)$ and $\Omega_{\delta_n}(t)$ and $b_{\delta_n}(0)$  similarly in the discrete setting. By considering the continuous modification, we may also assume that $T_n(\eps)$ converges to $T(\eps)$ almost surely.
Now, we explain briefly how to deduce that given $\gamma[0,\tau(s)]$, the conditional law of $\gamma$ is the same as the law of $\SLE_2(-1,-1;-1,-1)$ in $\Omega(0)$, which starts from $\gamma(\tau(s))$ with marked points $(a,d;b(0),c)$.

We fix $\tilde c,\tilde d\in[cd]$. We denote by $\tilde c_{\delta_n} ,\tilde d_{\delta_n}$ some discrete approximations of $(c_{\delta_n} d_{\delta_n})$. Given $\Omega_{\delta_n}(t)$, for $v_{\delta_n}\in \Omega_{\delta_n}(t)$, we define 
\[h_{t,(\tilde c_{\delta_n}\tilde d_{\delta_n})}(v_{\delta_n}):=\PP[\LR_{\delta_n}\text{ exits }\Omega_{\delta_n}(t)\text{ at }(\tilde c_{\delta_n}\tilde d_{\delta_n})],\]
where $\LR_{\delta_n}$ has the same law as the random walk on $\Omega_{\delta_n}(t)$, which starts from $v_{\delta_n}$ and is killed on $(a_{\delta_n}b_{\delta_n}(t)) \cup (c_{\delta_n} d_{\delta_n})$, with the reflecting boundaries $(b_{\delta_n}(t)c_{\delta_n})\cup(d_{\delta_n} a_{\delta_{n}})$.

Note that for $r<t$, by the domain Markov property for loop-erased random walk,
\begin{align}\label{eq:hit}
\PP[\gamma_{\delta_n}\text{ hits }(\tilde c_{\delta_n}\tilde d_{\delta_n})\cond\gamma_{\delta_n}[0,(t+\tau_n(s))\wedge T_n(\epsilon)]]
=&\frac{\sum_{v_{\delta_n}\sim \gamma_{\delta_n}((t+\tau_n(s))\wedge T_n(\eps))}h_{t,(\tilde c_{\delta_n}\tilde d_{\delta_n})}(v_{\delta_n})}{\sum_{v_{\delta_n}\sim \gamma_{\delta_n}((t+\tau_n(s))\wedge T_n(\eps))}h_{t,(c_{\delta_n} d_{\delta_n})}(v_{\delta_n})}.
\end{align}
Recall that $\varphi$ is the conformal map from $(\Omega;a,b,c,d)$ onto the rectangle, which maps $(a,b,c,d)$ to $(0,1,1+\ii K,\ii K)$. We fix the conformal map $f$ from the rectangle onto $\HH$, which maps $(0,1,\ii K)$ to $(0,1,\infty)$. We denote $f\circ\varphi(c)$ by $x$ and denote $f\circ\varphi(\tilde c)$ by $\tilde x$ and denote $f\circ\varphi(\tilde d)$ by $\tilde y$. We fix the conformal map $h_{\tau(s)}$ from $f\circ\varphi(\Omega(0))$ onto $\HH$ such that $h_{\tau(s)}(z)=z+o(1)$ when $z$ tends to $\infty$. We denote by $W_{\tau(s),\cdot}$ the driving function of $h_{\tau(s)}\circ f\circ\varphi(\gamma[\tau(s),T(\eps)])$ and denote by $(h_{\tau(s),t};0\le t\le T(\eps)-\tau(s))$ the corresponding Loewner flow. Then, by using results from \cite{HanLiuWuUST} (see in particular the start of proof of Theorem 1.6), it follows that the conditional law of $h_{\tau(s)}\circ f\circ\varphi(\gamma[\tau(s),T(\eps)])$ given $\gamma[0, \tau(s)]$ must be an SLE$_2(-1, -1, -1, -1)$. (Indeed, given $\gamma[0, \tau(s)]$ the rest of the curve is the loop-erasure of a chordal SLE in which the boundary conditions alternate between Dirichlet and Neumann, and the change of boundary condition does not occur at the tip of the curve, so this is exactly the setup of \cite{HanLiuWuUST}). Briefly, we recall the argument: first, using \eqref{eq:hit} and the same argument as in~\cite[Corollary 5.5]{HanLiuWuUST}, we have 
\begin{align*}
& \lim_{n\to\infty}\PP[\gamma_{\delta_n}\text{ hits }(\tilde c_{\delta_n}\tilde d_{\delta_n})\cond\gamma_{\delta_n}[0,(t+\tau_n(s))\wedge T_n(\epsilon)]]\\
=&\int_{h_{\tau(s),t}(\tilde x)}^{h_{\tau(s),t}(\tilde y)}\partial_n P(z;h_{\tau(s),t}(0),W_{\tau(s),t},h_{\tau(s),t}\circ h_{\tau(s)}\circ f\circ\varphi(b(0)),h_{\tau(s),t}\circ h_{\tau(s)}(x))|_{z=w}dw,
\end{align*}
where $P$ is the Poisson kernel corresponding to our boundary conditions, which is given in~\cite[Equation 5.10]{HanLiuWuUST}.
By the same argument in~\cite[Lemma 5.10]{HanLiuWuUST} (similar to \cite{LawlerSchrammWernerLERWUST}), we know that for $w\in (x,\infty)$,
\[h_{\tau(s),t}'(w)\partial_n P(z;h_{\tau(s),t}(0),W_{\tau(s),t},h_{\tau(s),t}\circ h_{\tau(s)}\circ f\circ\varphi(b(0)),h_{\tau(s),t}\circ h_{\tau(s)}(x))|_{z=w}\]
is a martingale for $h_{\tau(s)}\circ f\circ\varphi(\gamma[\tau(s),T(\eps)])$. This can be used to identify uniquely the law of the driving function. 

In particular $\gamma$ will not hit $[bc]\cup[da]$ after time $\tau(s)$. By letting $\eps\to 0$, we know that given $\gamma[0,\tau(s)]$, the conditional law of $\gamma$ is the same as the law of $\SLE_2(-1,-1;-1,-1)$ in $\Omega(0)$, which starts from $\gamma(\tau(s))$ with marked points $(a,d;b(0),c)$. Finally, by letting $s\to 0$, this implies that $\gamma$ only hits $\partial\Omega$ at its two ends. This implies that we always have $b(0)=b$. We denote by $W$ the driving function of $f\circ\varphi(\gamma)$. We denote by $(g_t;t\ge 0)$ the corresponding Loewner flow. Then, we have
\[W_{\tau(s)+t}=W_{\tau(s),t}\quad\text{and}\quad g_{\tau(s)+t}=h_{\tau(s),t}\circ h_{\tau(s)}.\]
By definition of $\SLE_2(-1,-1;-1,-1)$, after parameterizing $f\circ\varphi(\gamma)$ by its half-plane capacity, we have
\[W_{t+\tau(s)}-W_{\tau(s)}=\sqrt{2}(B_{t+\tau(s)}-B_{\tau(s)})+\int_{\tau(s)}^{\tau(s)+t}\left(\frac{-1}{W_r-g_r(0)}+\frac{-1}{W_r-g_r(1)}+\frac{-1}{W_r-g_r(x)}\right) dr,\]
for $t\le T(\eps)-\tau(s)$. By letting $\eps\to 0$ and then letting $s\to 0$, we have that 
\[W_{t}-W_{0}=\sqrt{2}B_{t}+\int_{0}^{t}\left(\frac{-1}{W_r-g_r(0)}+\frac{-1}{W_r-g_r(1)}+\frac{-1}{W_r-g_r(x)}\right) dr,\]
where $B$ is the standard $1$-dimensional Brownian motion. Thus, we have that $\gamma$ has the same law as the law of $\SLE_2(-1,-1;-1,-1)$ in $\Omega$, which starts from $b$ with marked points $(a,d;b,c)$. This completes the proof.  
\end{proof}

Now, we discuss the convergence of $\{\eta_L^\delta\}_{\delta>0}$ and $\{\eta_R^\delta\}_{\delta>0}$. In~\cite{HanLiuWuUST}, the authors considered the $\ust$ on $(\Omega_\delta;a_\delta,b_\delta,c_\delta,d_\delta)$ where $(a_\delta b_
\delta)$ and $(c_\delta d_\delta)$ are wired. As mentioned before, this tree has a unique branch connecting the two wired vertices to one another, which we denote by $\tilde \gamma_\delta$ (note that this curve is not conditioned on its starting point).
 These authors proved that the discrete Peano curves on either side of $\tilde \gamma_{\delta}$, which we denote by $\tilde\eta_L^\delta$ and $\tilde\eta_R^\delta$, converge to a pair of $\hSLE_8$. In our setting, the tightness of discrete Peano curves may not been proved as in~\cite{HanLiuWuUST}, since we condition on an event with very small probability. But assuming the convergence of $\tilde\eta_L^\delta$ and $\tilde\eta_R^\delta$, we can prove the convergence of $\eta_L^\delta$ and $\eta_R^\delta$ before hitting $(a_\delta b_\delta)$. In the remaining part, we will only consider the convergence of $\{\eta_R^\delta\}_{\delta>0}$, since the convergence for $\{\eta_L^\delta\}_{\delta>0}$ is similar.

We denote by $\tau_{\eps}^\delta$ to be the hitting time by $\eta_R^\delta$ of the $\eps$-neighbourhood of $(d_\delta b_\delta)$ and denote by $\tilde\tau^\delta_{\eps}$ for $\tilde\eta_R^\delta$ similarly.  We denote by $\tilde\eta_R$ the $\hSLE_8$ curve from $c$ to $b$ with marked points $d,a$. We first recall the driving function of this curve. Let us fix a conformal map $\phi$ from $\Omega$ onto $\HH$, which maps $(c,d,b)$ to $(0,1,\infty)$. Suppose $\tilde W$ is the driven function of $\phi(\tilde\eta_R)$ and $(\tilde g_t:t\ge 0)$ the corresponding conformal maps. Then, before hitting $(\phi(a),+\infty)$, by definition, the law of $\tilde W$ can be described as that of the solution to the following hypergeometric stochastic differential equation:
\begin{equation}\label{E:hSLE}
d\tilde W_t=\sqrt 8 dB_t+\left(\frac{2}{\tilde W_t-\tilde g_t(\phi(d))}-\frac{2}{\tilde W_t-\tilde g_t(\phi(a))}-8\frac{\tilde g_t(\phi(a))-\tilde g_t(\phi(d))}{\left(\tilde g_t(\phi(a))-\tilde W_t\right)^2}\frac{F'\left(\frac{\tilde g_t(\phi(d))-\tilde W_t}{\tilde g_t(\phi(a))-\tilde W_t}\right)}{F\left(\frac{\tilde g_t(\phi(d))-\tilde W_t}{\tilde g_t(\phi(a))-\tilde W_t}\right)}\right)dt,
\end{equation}
where $F(\cdot)={}_2F_{1}(1/2,1/2,1;\cdot)$.
See~\cite{HanLiuWuUST} for more details.
For every $v_\delta\in\Omega_\delta\setminus\tilde\eta_R^\delta[0,t]$, we define 
\[h_{t}(v_\delta):=\PP[\LR_\delta\text{ hits }(\tilde\eta_R^\delta(t)d_\delta)\text{ before }(a_\delta b_\delta)],\]
where $\LR_\delta$ has the same law as as before, namely a random walk starting from $v_\delta$ and ending at $(a_\delta b_\delta)\cup(\tilde\eta_R^\delta(t)d_\delta)$, with reflecting boundaries $(b_\delta \tilde\eta_R^\delta(t))$ and $(d_\delta a_\delta)$.

We still need two notations. We denote by $SF_2^\delta(\Omega_\delta)$ the set of spanning forests on $\Omega_\delta$ where $(a_\delta b_\delta)$ and $(c_\delta d_\delta)$ are wired but belong to different components. We denote by $ST^\delta(\Omega_\delta)$ the set of spanning trees on $\Omega_\delta$ (where, once again, $(a_\delta b_\delta)$ and $(c_\delta d_\delta)$ are wired).
\begin{lemma}\label{lem::disRN}
Suppose $r$ is a continuous function on the curves space $(\LP,d)$. Then, for every $t>0$, we have
\[\E[r(\eta_R^\delta[0,t\wedge\tau^\delta_{\eps}])]=\E\left[r(\tilde\eta_R^\delta[0,t\wedge\tilde\tau^\delta_{\eps}])\frac{h_{t\wedge\tilde\tau^\delta_{\eps}}(b^+_\delta)\times \frac{|SF_2^\delta(\Omega_\delta\setminus\tilde\eta_R^\delta[0,t\wedge\tilde\tau^\delta_{\eps}])|}{|ST^\delta(\Omega_\delta\setminus\tilde\eta_R^\delta[0,t\wedge\tilde\tau^\delta_{\eps}])|}}{h_{(c_\delta d_\delta)}(b_\delta^+)\times\frac{|SF_2^\delta(\Omega_\delta)|}{|ST^\delta(\Omega_\delta)|}}\right],\]
where $SF_2^\delta(\Omega_\delta\setminus\tilde\eta_R^\delta[0,t\wedge\tilde\tau^\delta_{\eps}])$ and $ST^\delta(\Omega_\delta\setminus\tilde\eta_R^\delta[0,t\wedge\tilde\tau^\delta_{\eps}])$ are defined for $\Omega_\delta\setminus\tilde\eta_R^\delta[0,t\wedge\tilde\tau^\delta_{\eps}]$ similarly to $SF_2^\delta(\Omega_\delta)$ and $ST^\delta(\Omega_\delta)$, and $h_{(c_\delta d_\delta)}$ is defined in \eqref{E:hcd}.
\end{lemma}
\begin{proof}
We denote by $|S|$ the cardinality of a set $S$. Note that 
\begin{align*}
&\PP[\text{ the first step of }\tilde \gamma_{\delta} \text{ is }b_\delta b_\delta^+]\\
=&\frac{|\{t^\delta\in ST^\delta(\Omega_\delta): \text{ the first step of the branch connecting }a_\delta b_\delta \text{ to }c_\delta d_\delta\text{ is }b_\delta b_\delta^+\}|}{|ST^\delta(\Omega_\delta)|}\\
=&\frac{|SF_2^\delta(\Omega_\delta)|}{|ST^\delta(\Omega_\delta)|}\times\frac{|\{f^\delta\in SF_2^\delta(\Omega_\delta):b_\delta^+\text{ connects to }(c_\delta d_\delta)\}|}{|SF_2^\delta(\Omega_\delta)|}\\
=&\frac{|SF_2^\delta(\Omega_\delta)|}{|ST^\delta(\Omega_\delta)|}\times h_{(c_\delta d_\delta)}(b_\delta^+),
\end{align*}
where the last equality is from Wilson's algorithm. The conclusion is obtained from the domain Markov property of the $\ust$.
\end{proof}
Now, we prove the convergence of $\{\eta_R^\delta[0,\tau^\delta_{\eps}]\}_{\delta>0}$ for every $\eps>0$. We define $\tilde\tau_{\eps}$ to be the hitting time by $\tilde\eta_R$ of the $\eps$-neighbourhood of $(db)$. 
\begin{proposition}\label{prop::hsle}
The discrete Peano curves $\{\eta_R^\delta[0,\tau^\delta_{\eps}]\}_{\delta>0}$ converge in law to $\SLE_8(2,2)$ with marked points $d,a$ stopped at $ \tau_\eps$. The precise formula of the driven function is given in~\eqref{eqn::counterSLE}.
\end{proposition}
\begin{proof}
By~\cite[Theorem~1.5]{HanLiuWuUST}, we have $\tilde\eta_R^\delta$ converges to $\tilde\eta_R$ in law. We couple $\{\tilde\eta_R^\delta\}_{\delta>0}$ and $\tilde\eta_R$ together, such that $\tilde\eta_R^\delta$ converges to $\tilde\eta_R$ almost surely. If we define $\tilde\tau_{\eps,+}:=\inf\{\tilde\tau_{\eps'}:\eps'<\eps\}$, from the absolute continuity of $\tilde\eta_R$ with respect to $\SLE_8$, we know that almost surely, $\tilde\tau_{\eps}=\tilde\tau_{\eps,+}$. Thus, under this coupling, we have $\tilde\tau^\delta_{\eps}$ converges to $\tilde\tau_{\eps}$ almost surely. 

By the same argument as in Lemma~\ref{lem::conv_subarc}, we can prove that $h_{t\wedge\tilde\tau^\delta_{\eps}}$ converges a.s. locally uniformly (under this coupling distribution). We denote the limit by $h_{t\wedge\tilde\tau_{\eps}}$, and we note that it is a bounded harmonic function which satisfies the following boundary conditions: 
$$h_{t\wedge\tilde\tau_{\eps}}
\circ\phi^{-1}\circ\tilde g_{t\wedge\tilde\tau_{\eps}}^{-1}
\begin{cases}
 =1 & \text{ on } (\tilde W_{t\wedge\tilde\tau_{\eps}},\tilde g_{t\wedge\tilde\tau_{\eps}}(\phi(d)))\\
  = 0 & \text{ on }(\tilde g_{t\wedge\tilde\tau_{\eps}}(\phi(a)),\infty)\\
  \end{cases}
  $$
  and 
  $$
 \partial_n \left(h_{t\wedge\tilde\tau_{\eps}}
\circ\phi^{-1}\circ\tilde g_{t\wedge\tilde\tau_{\eps}}^{-1}\right) 
= 0 \text{ on } (-\infty,\tilde W_{t\wedge\tilde\tau_{\eps}})\cup(\tilde g_{t\wedge\tilde\tau_{\eps}}(\phi(d)),\tilde g_{t\wedge\tilde\tau_{\eps}}(\phi(a))).
$$
 By~\cite[Lemma 4.6]{HanLiuWuUST}, we have that $\frac{|SF_2^\delta(\Omega_\delta\setminus\tilde\eta_R^\delta[0,t\wedge\tilde\tau^\delta_{\eps}])|}{|ST^\delta(\Omega_\delta\setminus\tilde\eta_R^\delta[0,t\wedge\tilde\tau^\delta_{\eps}])|}$ converges to the conformal modulus of the domain $(\Omega\setminus\tilde\eta_R[0,t\wedge\tilde\tau_\eps])$ with marked points $(a, b, \tilde \eta(t \wedge \tilde \tau_\eps), d)$, as $\delta\to 0$ (as can be seen from that proof, the only requirement for this convergence is the Carath\'eodory convergence of the domains). Thus, if we denote by $f_{t\wedge\tilde\tau_{\eps}}$ the conformal map from $(\HH;\tilde W_{t\wedge\tilde\tau_{\eps}},\tilde g_{t\wedge\tilde\tau_{\eps}}(\phi(d)),\tilde g_{t\wedge\tilde\tau_{\eps}}(\phi(a)),\infty)$ onto $(R;0,1,1+\ii K,K)$, where $R$ is a rectangle with unit width, then we have 
\[h_{t\wedge\tilde\tau^\delta_{\eps}}(\cdot)\times \frac{|SF_2^\delta(\Omega_\delta\setminus\tilde\eta_R^\delta[0,t\wedge\tilde\tau^\delta_{\eps}])|}{|ST^\delta(\Omega_\delta\setminus\tilde\eta_R^\delta[0,t\wedge\tilde\tau^\delta_{\eps}])|}\]
converges to $K-\Im f_{t\wedge\tilde\tau_{\eps}}\circ \tilde g_{t\wedge\tilde\tau_{\eps}}\circ\phi$ locally uniformly on $\Omega\setminus\tilde\eta_{R}[0,t\wedge\tilde\tau_{\eps}]$ (as this is the only bounded harmonic function satisfying the required boundary conditions on $R$).
From the Schwarz--Christorffel formula, we have
\[f_{t\wedge\tilde\tau_{\eps}}(z)=\frac{\int_{0}^{\frac{z-\tilde W_t}{\tilde g_{t\wedge\tilde\tau_{\eps}}(\phi(d))-\tilde W_{t\wedge\tilde\tau_{\eps}}}}\left(s(s-1)\left(s-\frac{\tilde g_{t\wedge\tilde\tau_{\eps}}(\phi(a))-\tilde W_{t\wedge\tilde\tau_{\eps}}}{\tilde g_{t\wedge\tilde\tau_{\eps}}(\phi(d))-\tilde W_{t\wedge\tilde\tau_{\eps}}}\right)\right)^{-1/2}ds}{\int_{0}^1\left(s(s-1)\left(s-\frac{\tilde g_{t\wedge\tilde\tau_{\eps}}(\phi(a))-\tilde W_{t\wedge\tilde\tau_{\eps}}}{\tilde g_{t\wedge\tilde\tau_{\eps}}(\phi(d))-\tilde W_{t\wedge\tilde\tau_{\eps}}}\right)\right)^{-1/2}ds}.\]
For the denominator, if we change the variable by setting $s=\sin^2\theta$, we have
\begin{align*}
&\int_{0}^1\left(s(s-1)\left(s-\frac{\tilde g_{t\wedge\tilde\tau_{\eps}}(\phi(a))-\tilde W_{t\wedge\tilde\tau_{\eps}}}{\tilde g_{t\wedge\tilde\tau_{\eps}}(\phi(d))-\tilde W_{t\wedge\tilde\tau_{\eps}}}\right)\right)^{-1/2}ds\\ =&2\sqrt{\frac{\tilde g_{t\wedge\tilde\tau_{\eps}}(\phi(d))-\tilde W_{t\wedge\tilde\tau_{\eps}}}{\tilde g_{t\wedge\tilde\tau_{\eps}}(\phi(a))-\tilde W_{t\wedge\tilde\tau_{\eps}}}}\times\int_{0}^{\pi/2}\frac{d\theta}{\sqrt{1-\frac{\tilde g_{t\wedge\tilde\tau_{\eps}}(\phi(d))-\tilde W_{t\wedge\tilde\tau_{\eps}}}{\tilde g_{t\wedge\tilde\tau_{\eps}}(\phi(a))-\tilde W_{t\wedge\tilde\tau_{\eps}}}\sin^2\theta}}\\
=&2\sqrt{\frac{\tilde g_{t\wedge\tilde\tau_{\eps}}(\phi(d))-\tilde W_{t\wedge\tilde\tau_{\eps}}}{\tilde g_{t\wedge\tilde\tau_{\eps}}(\phi(a))-\tilde W_{t\wedge\tilde\tau_{\eps}}}}\times{}_2F_1\left(1/2,1/2,1;\frac{\tilde g_{t\wedge\tilde\tau_{\eps}}(\phi(d))-\tilde W_{t\wedge\tilde\tau_{\eps}}}{\tilde g_{t\wedge\tilde\tau_{\eps}}(\phi(a))-\tilde W_{t\wedge\tilde\tau_{\eps}}}\right),
\end{align*}
where the last equality is obtained from the standard relation between elliptic function and hypergeometric function. In particular, as $z \to \infty$, 
$$
f'_{t \wedge \tilde \tau_\eps} (z) \sim \sqrt{{\tilde g_{t\wedge\tilde\tau_{\eps}}(\phi(a))-\tilde W_{t\wedge\tilde\tau_{\eps}}}}
\times z^{-3/2} \times
\frac1{{}_2F_1\left(1/2,1/2,1;\frac{\tilde g_{t\wedge\tilde\tau_{\eps}}(\phi(d))-\tilde W_{t\wedge\tilde\tau_{\eps}}}{\tilde g_{t\wedge\tilde\tau_{\eps}}(\phi(a))-\tilde W_{t\wedge\tilde\tau_{\eps}}}\right)}.
$$
Furthermore, since
\[g_{t\wedge\tilde\tau_{\eps}}(z)=z+o(1) \quad\text{as } z\to\infty,\]
we obtain by Lemma~\ref{lem::ratio},
\begin{align}
&\lim_{\delta\to 0} \frac{h_{t\wedge\tilde\tau^\delta_{\eps}}(b^+_\delta)\times \frac{|SF_2^\delta(\Omega_\delta\setminus\tilde\eta_R^\delta[0,t\wedge\tilde\tau^\delta_{\eps}])|}{|ST^\delta(\Omega_\delta\setminus\tilde\eta_R^\delta[0,t\wedge\tilde\tau^\delta_{\eps}])|}}{h_{(c_\delta d_\delta)}(b_\delta^+)\times\frac{|SF_2^\delta(\Omega_\delta)|}{|ST^\delta(\Omega_\delta)|}}\nonumber \\
=&\sqrt{\frac{\tilde g_{t\wedge\tilde\tau_{\eps}}(\phi(a))-\tilde W_{t\wedge\tilde\tau_{\eps}}}{\phi(a)-\tilde W_0}}\times\frac{{}_2F_1\left(1/2,1/2,1;\frac{\phi(d)-\tilde W_0}{\phi(a)-\tilde W_0}\right)}{{}_2F_1\left(1/2,1/2,1;\frac{\tilde g_{t\wedge\tilde\tau_{\eps}}(\phi(d))-\tilde W_{t\wedge\tilde\tau_{\eps}}}{\tilde g_{t\wedge\tilde\tau_{\eps}}(\phi(a))-\tilde W_{t\wedge\tilde\tau_{\eps}}}\right)}.\label{Eq:DF}
\end{align}
We denote the limit by $M_{t\wedge\tilde\tau_{\eps}}$. Note that 
\begin{equation}\label{Eq:boundratio}
\frac{h_{t\wedge\tilde\tau_{b_\delta,\eps},(c_\delta d_\delta)}(b^+_\delta)}{h_{(c_\delta d_\delta)}(b_\delta^+)}\le C
\end{equation}
where $C = C(\eps)$ depends only on $\eps>0$. Indeed, if the walk $\cR$ reaches $\tilde \eta^\delta ([0, \tilde \tau_\eps^\delta]) $ before touching $(a_\delta b_\delta)$, then it can hit $(d_\delta, c_\delta)$ with uniformly positive probability (depending only on $\eps$). 

Moreover, we claim that the ratio
 $$\frac{|SF_2^\delta(\Omega_\delta\setminus\tilde\eta_R^\delta[0,t\wedge\tilde\tau^\delta_{\eps}])|}{|ST^\delta(\Omega_\delta\setminus\tilde\eta_R^\delta[0,t\wedge\tilde\tau^\delta_{\eps}])|}$$
 is uniformly bounded by a constant which may depend on $\eps$ but not on $\tilde \eta$. Indeed, let us argue by contradiction. Suppose this was not the case. Then we would find a sequence $\delta_n \to 0$ and a a sequences of paths $\eta^{n}$ such that the ratio above tends to $\infty$ along the sequence $\delta_n$. But arguing by compactness, we can always find a furthersubsequence (call it $\delta_n$ again with an abuse of notations) such that $\Omega\setminus \eta^n[0, \tau^{\delta_n}_\eps]$ converges in the Carath\'eodory sense. But, as already mentioned (and see Lemma 4.6 in \cite{HanLiuWuUST} for a proof), such a convergence is sufficient to guarantee the convergence of the above ratio to conformal modulus of the limiting domain, which gives the desired contradiction.

This implies that there exists a constant $C'$ depending only on $\eps$ and $\Omega$, such that for every $\delta>0$,
\[
 \frac{|SF_2^\delta(\Omega_\delta\setminus\tilde\eta_R^\delta[0,t\wedge\tilde\tau^\delta_{\eps}])|}{|ST^\delta(\Omega_\delta\setminus\tilde\eta_R^\delta[0,t\wedge\tilde\tau^\delta_{\eps}])|}\le C'.
 \]
Therefore, the martingale $M_{t\wedge \tilde \tau_\eps^\delta}$ is uniformly bounded by a constant (which may depend on $\eps$ but not on anything else).  
Thus, if $r$ is a continuous function on the curves space $(\LP,d)$, by dominated convergence theorem, we have
\[\lim_{\delta\to 0}\E[r(\eta_R^\delta)[0,t\wedge\tau_{b_\delta,\eps}]]=\E[r(\tilde\eta_R[0,t\wedge\tilde\tau_\eps])M_{t\wedge\tilde\tau_\eps}].\]
This complete the proof of the convergence of $\{\eta_R^\delta[0,\tau_{\eps}^\delta]\}_{\delta>0}$ in law. We denote by $\eta_R[0, \tau_\eps]$ the limit. 

Now, we derive the explicit formula of the law of $\phi(\eta_R[0,\tau_{\eps}])$. Since $(\LP,d)$ is a metric space, for any open set $O$, the indicator function $\one_O$ can be approximated by bounded continuous functions $\{r_n\}_{n\ge 0}$ such that $r_{n}\nearrow \one_O$. Thus, we have
\[
\E[\one_O(\eta_R[0,t\wedge\tau_\eps])]=\E\left[\one_O(\tilde\eta_R[0,t\wedge\tilde\tau_
\eps])M_{t\wedge\tilde\tau_\eps}\right].
\]
This implies that the law of $(\eta_R[0,t\wedge\tau_\eps]:t\ge 0)$ equals the law of $(\tilde\eta_R[0,t\wedge\tilde\tau_
\eps]:t\ge 0)$ weighted by the martingale $(M_{t\wedge\tilde\tau_\eps}:t\ge 0)$.
If we denote by $W$ the driving function of $\phi(\eta_R[0,\tau_{\eps}])$ and denote by $(g_{t\wedge\tau_{\eps}}:t\ge 0)$ the corresponding conformal maps, by Girsanov's theorem and \eqref{Eq:DF}, we have
\begin{equation}\label{eqn::counterSLE}
dW_{t\wedge\tilde\tau_\eps}=d\tilde W_{t\wedge\tilde\tau_\eps}+\frac{d \langle M, \tilde W\rangle_{t\wedge\tilde\tau_\eps}}{M_{t\wedge\tilde\tau_\eps}}=\sqrt{8}dB_{t\wedge\tilde\tau_\eps}+\left(\frac{2}{W_{t\wedge\tilde\tau_\eps}-g_{t\wedge\tilde\tau_\eps}(\phi(d))}+\frac{2}{W_{t\wedge\tilde\tau_\eps}-g_{t\wedge\tilde\tau_\eps}(\phi(a))}\right)d{t\wedge\tilde\tau_\eps}.
\end{equation}
This completes the proof.
\end{proof}
\begin{remark}\label{rem::hyper}
\begin{itemize}
\item
By a similar proof, we have that for every $\eps>0$, the discrete Peano curve $\eta_L^\delta$ stopped at the hitting time of the $\eps$-neighbourhood of $(b_\delta d_\delta)$ converges to chordal $\SLE_8(2,-2)$ from $d$ to $a$ with force points $b,c$ stopped at the hitting time of the $\eps$-neighbourhood of $(ab)$.
\item
In the discrete setting, $\gamma_{\delta}$ is the left boundary of $\eta_L^\delta$ and is also the right boundary of $\eta_R^\delta$. The convergence of $\{\eta_L^\delta\}_{\delta>0}$ and $\{\eta_R^\delta\}_{\delta>0}$ as well as that of $\gamma_\delta$ in Proposition \ref{prop::middle} imply that the left boundary of $\SLE_8(2,-2)$ and the right boundary of $\SLE_8(2,2)$ are both given by a chordal $\SLE_2(-1,-1;-1,-1)$ starting from the appropriate boundary point. This identity is equivalent to the coupling between the flow line and counterflow lines given in~\cite{MillerSheffieldIG1}. In Section~\ref{sec::Convergence of winding}, we will further show the convergence of the discrete winding field (i.e., dimer height function) to the corresponding Gaussian free field under which flow and counterflow lines are coupled. 
\item
Recall that the law of the $\ust$ we consider in this paper is equivalent to the law of $\ust$ with alternating boundary conditions (as 
described e.g. below Theorem \ref{T:IG}), conditional on the discrete Peano curve $\eta_R^\delta$ hitting $(a_\delta b_\delta)$ at  $b_\delta$. By almost the same argument as above, it can be shown that if we consider the $\ust$ with alternating boundary conditions, if we now condition on the event that the discrete Peano curve $\eta_R^\delta$ hits $(a_\delta b_\delta)$ at $z_\delta$ and assume $z_\delta$ converges to $z\in(ab)$, then the conditional law of $\eta_R^\delta$ converges to $\SLE_8(2,2,-4)$ with marked points $d$, $a$ and $z$.  This implies that given the hitting point of $\hSLE_8$ at $(ab)$, which we denote by $z$, the conditional law will be $\SLE_8(2,2;-4)$ with marked points $d$, $a$ and $z$. From this Theorem \ref{T:condSLE8} follows easily.
\item
It is natural to ask whether the identities described above hold more generally for chordal $\hSLE_\kappa(\nu)$ process for $\kappa\in (4,8)$ and $\nu>\frac{\kappa}{2}-4$. In fact, in the forthcoming paper \cite{Liu_reverse}, the following result will be shown. Suppose $\eta$ has the same law as a chordal $\hSLE_\kappa(\nu)$ curve from $b$ to $c$ with marked points $d$ and $a$, given the hitting point of $\eta$ at $(ab)$ (which we denote by $z$), the conditional law of $\eta$ equals to the law of $\SLE_\kappa(2,\kappa-6,-4)$ with marked points $d,a$ and $z$. As a consequence, using tools from imaginary geometry, this decomposition implies the time-reversibility of $\hSLE_\kappa(\nu)$ for $\kappa\in(4,8)$ and $\nu>-2$. 
\end{itemize}
\end{remark}

\subsection{Convergence of the discrete trees}
\label{subsec::coupling}

In this section, we will prove Theorem~\ref{thm::treeconv}. Recall that at this stage we have proved convergence of the interface $\gamma_\delta$ towards an SLE$_2$ type chordal curve, and we have proved convergence of the Peano curves on either side of it, but only up until they hit $(ab)$. This is not completely sufficient to get convergence of the discrete trees in the Schramm sense, because in order to describe all the branches of the tree we would need to know the convergence of the full Peano curves not just stopped when they hit $(ab)$. On the other hand, it is clear that, at the discrete level, given $\gamma_\delta$, the two components describing the rest of the tree can be viewed as a uniform spanning trees in their respective domain with Dobrushin boundary conditions, with the arcs $(da)$ and $(bc)$ reflecting, and every other part of the boundary being wired. This is very close to the setup of the original paper of Lawler, Schramm and Werner \cite{LawlerSchrammWernerLERWUST}, but there is a technical subtlety: namely, in order to get a strong enough form of convergence (say uniform) the authors of \cite{LawlerSchrammWernerLERWUST} require the boundary to be smooth, which obviously is not the case here (convergence in the sense of driving function and hence Carath\'eodory convergence does not require smoothness in their paper, unfortunately this is not sufficient here). 

However, in order to upgrade the form of convergence from driving function convergence to strong (or uniform up to reparametrisation) convergence, it suffices to prove tightness with respect to this strong topology, as then the limit law is uniquely identified by the convergence of the driving function (which, we recall, holds without assumption on the domain). But tightness in this sense is in fact not hard to show: it follows along the same line as already argued in Lemma \ref{lem::tightness}, and in fact is considerably simpler since there is no need to condition on an event of asymptotically degenerate probability. 
We leave the details to the reader; and obtain the convergence of the branches of the tree $\cT^\delta$ to the brances of $\cT$ in the sense of finite-dimensional distribution of the branches. To conclude the proof of Theorem~\ref{thm::treeconv}, it remains to apply the following well known compactness argument, due to Schramm \cite[Theorem 10.2]{SchrammFirstSLE} (this is sometimes known as Schramm's lemma). Its proof is unchanged, and so we can simply quote the result here:

\begin{lemma}\label{lem::subtree}
For every $\epsilon>0$, there exist $\delta_0>0$ and $\epsilon_0>0$, such that the following holds. For any set $\{z_1,\ldots,z_{n}\}$ with the property that every $z\in \Omega$, there exists $z_i$ such that there exists a curve connecting $z$ and $z_i$ whose length is less than $\eps_0$. We view $(ab)$ and $(cd)$ as two points and define $V:=\{z_1,\ldots,z_n,(ab),(cd)\}$. Denote by $V_\delta$ the discrete approximation on $\Omega_\delta$. Denote by $\LT_\delta(V)$ the minimal subtree containing $V_\delta$. Then,we have
\[\PP[d_{\LH}(\LT_\delta,\LT_\delta(V))>\delta_0]\le\eps.\] 
\end{lemma}

This concludes the proof of Theorem \ref{thm::treeconv}. 

\section{General case}
\label{sec::Convergence of branches}

In this section, we will prove of the convergence of the discrete trees considered in Lemma~\ref{lem::bij} and give the coupling of the limiting tree with a GFF for $n\ge 3$ in the sense of Theorem \ref{thm::flowcoupling}. Fix a simply connected domain 
\[(\Omega;x_1,\ldots,x_{2n};z_1,\ldots,z_{n-1}),\]
 where $x_1,\ldots,x_{2n}$ are marked boundary points on $\partial \Omega$ (understood as prime ends) in counterclockwise order and $z_i$ equals $x_{2i}$ or $x_{2i+1}$ for $1\le i\le n-1$. (It may be useful to keep in mind the combinatorial setup of Section \ref{subsec::Temperleyan}.) We assume that $\partial \Omega$ is locally connected and there exists a simply connected domain $\tilde \Omega$ whose boundary is $C^1$ and simple such that $\Omega\subset\tilde \Omega$ and $\partial \Omega\cap\partial\tilde \Omega$ equals $\cup_{i=1}^{n}(x_{2i-1}x_{2i})$. Suppose a sequence of discrete simply connected domains $(\Omega_\delta;x_1^\delta,\ldots,x^\delta_{2n};z^\delta_1,\ldots,z^\delta_{n-1})$ converges to $(\Omega;x_1,\ldots,x_{2n};z_1,\ldots,z_{n-1})$ in the following sense: there exists $C>0$ such that for $1\le i\le n$, we have
\begin{equation}\label{eqn::converge1}
d((x^\delta_{2i-1}x^\delta_{2i}),(x_{2i-1}x_{2i}))\le C\delta,\quad d((x^\delta_{2i}x^\delta_{2i+1}),(x_{2i}x_{2i+1}))\to 0;
\end{equation}
and for $1\le i\le n-1$, we have 
\begin{equation}\label{eqn::converge2}
d(z_i^\delta,z_i)\to 0.
\end{equation}
Recall the spanning trees appearing in Temperley's bijection are described by a set $\LU(\overline\Omega^*_0)$, see Lemma \ref{lem::bij}.
 

It turns out dealing with $\LU(\overline\Omega^*_0)$ is hard because, although it can still be seen as an asymptotically degenerate conditioning of the uniform spanning tree with alternating boundary conditions (which is well understood by \cite{HanLiuWuUST}), this conditioning is still to complicated to be described directly. Instead, we will describe the complement of the event which serves to condition the tree with alternating boundary conditions, using an inclusion-exclusion description.  First, let $\tilde \cT_\delta$ be the UST on $\Omega_\delta$ with alternating boundary conditions (where  $c_i = (x_{2i}x_{2i+1})$ is wired for each $1\le i \le n$). Let 
$$\tilde E = \tilde E(z_{1}^\delta z_1^{\delta,+},\ldots,z_{n-1}^{\delta}z_{n-1}^{\delta,+})$$ be the event that the first step of the branch in $\tilde \cT_\delta$ connecting $(x_{2i}^{\delta}x_{2i+1}^{\delta})$ to $(x_{2n}^{\delta}x_{1}^{\delta})$ is through the edge $z_{i}^{\delta}z_{i}^{\delta,+}$. Thus $\LU(\overline\Omega^*_0)$ consists of the restriction of alternating spanning trees $\tilde \cT_\delta$ to this event  $\tilde E(z_{1}^\delta z_1^{\delta,+},\ldots,z_{n-1}^{\delta}z_{n-1}^{\delta,+})$.

Now, we consider a graph $G^*$ with the boundary condition that $\cup_{i=1}^{n}(x^\delta_{2i-1}x^\delta_{2i})$ is wired to be a single vertex (thus not only $c_i = (x_{2i} x_{2i+1})$ is wired, but these arcs are also wired together). Fix $1\le s\le n-1$, 
and fix a subset of indices $1\le i_1<  \ldots <  i_s\le n-1$. Define the set of spanning trees 
\begin{equation}
\label{eq:A}
\LA_{i_1,\ldots,i_s}:=\{\LT^*_\delta \text{ on } G^*, z_{i_k}^{\delta,+}\text{ connects to }(x_{2i_{k+1}}^\delta x_{2i_{k+1}+1}^\delta)\text{ in }\Omega_\delta\cap\LT^*_\delta,\text{ for }1\le k\le s\},
\end{equation}
where here we used the convention that  $i_{s+1}=i_1$ in this description.
 Thus an equivalent description of the event $\LA_{i_1,\ldots,i_s}$ is that the edges $(z^\delta_{i_1} z_{i_1}^{\delta,+}), \ldots, (z^\delta_{i_s} z_{i_s}^{\delta,+})$ are all closed (in the sense that they are not part of the tree $\cT^*_\delta$), but adding them to the tree would create a loop connecting $(c_{i_1}, \ldots, c_{i_s})$ in $\Omega_\delta$.

Finally, define the event 
\begin{equation}\label{eqn::eventdef}
E^* = E^*(\Omega_\delta; z_1^{\delta,+},\ldots,z_{n-1}^{\delta,+}):=\left(\cup_{s=1}^{n-1}\cup_{i_1,\ldots,i_s}\LA_{i_1,\ldots,i_s}\right)^c.
\end{equation}
Note that there is a bijection from $\tilde E$ to $E^*$: 
for every $\tilde \LT_\delta  \in \tilde E(z_{1}^\delta z_1^{\delta,+},\ldots,z_{n-1}^{\delta}z_{n-1}^{\delta,+})$, we delete the edges $\{z_1^\delta z_1^{\delta,+},\ldots,z_{n-1}^\delta z_{n-1}^{\delta,+}\}$ and connect $\cup_{i=1}^n(x_{2i}^\delta x_{2i+1}^\delta)$ outside of $\Omega_\delta$ through the wired vertex. The inclusion-exclusion principle shows this is a bijection.

In this section, we will therefore focus on consider the convergence of the uniform spanning tree $\cT^*_\delta$ on $G^*$ after conditioning it on the event $E^*$ above. 
%
We denote by $\gamma_i^\delta$ the path in $\cT^*_\delta$ starting from $z_i^{\delta,+}$ and ending at the first time it hits $\cup_{j\neq i}(x^\delta_{2j}x^\delta_{2j+1})$. 
The goal of this section is to show the following theorem.
\begin{theorem}\label{thm::gentreeconv}
The discrete tree $\LT^*_\delta$ converges as $\delta\to 0$ under the metric $d_{\LH}$. Denote by $\LT^*$ the limit. 
The continuous tree $\LT^*$ can be constructed as follows. Fix any dense set $\{z_i\}_{i\ge 1}$ in $\Omega$. We first sample the coupling of $\{\gamma_{1},\ldots,\gamma_{n-1}\}$ and the GFF $h_{\Omega}+u_{\Omega}(\cdot;x_1,\ldots,x_{2n};z_1,\ldots,z_{n-1})$ as in Theorem~\ref{thm::flowcoupling}. Recall that we denote by $\Omega_1,\ldots,\Omega_{n}$ the connected components of $D\setminus\cup_{i=1}^{n-1}\gamma_i$. For each $i \ge 1$, we denote by $j(i)$ the unique index $1\le j(i) \le n$ such that $z_i\in \Omega_{j(i)}$. Then, we sample the flow line starting from interior point $\gamma_{z_i}$ in $\Omega_{j(i)}$ similarly as in  Theorem~\ref{thm::treeconv}. Then $\cT^*$ is the closure of $\cup_{i\ge 1} \gamma_{z_i}$.
\end{theorem}

By the same arguments as in Section~\ref{subsec::coupling}, it suffices to show the convergence of the boundary branches:
\begin{theorem}\label{thm::genflow}
The discrete curves $\{\gamma_1^\delta,\ldots,\gamma_{n-1}^\delta\}$ converge as $\delta \to 0$ to the flow lines $\{\gamma_1^I,\ldots,\gamma_{n-1}^I\}$ defined in Theorem~\ref{thm::flowcoupling}. 
\end{theorem}

\subsection{Tightness of the boundary branches}

We start with the tightness of the boundary branches. First of all, we will give an estimate about the probability of $E^*(\Omega_\delta;x_1^\delta,\ldots,x_{2n}^\delta;z_1^{\delta,+},\ldots,z_{n-1}^{\delta,+})$. Fix a small $r>0$ such that $\partial B(x_i,2r)\cap\partial B(x_j,2r)=\emptyset$ for $i\neq j$.
We denote by $\partial B(x_i^\delta,r)$ the discrete approximation of $\partial B(x_i,r)$ on $\Omega_\delta$. For every $v_\delta\in \cup_{i=1}^{n-1} B(z_i^\delta,r)$, we define the harmonic function 
\[h^\delta_r(v_\delta):=\PP_{v_\delta}[\LR_\delta\text{ hits }\cup_{i=1}^{n-1} \partial B(z_i^\delta,r)\text{ before }\cup_{i=1}^{n}(x_{2i}^\delta x_{2i+1}^\delta)],\]
where $\LR_\delta$ has the same law as the random walk killed at $\cup_{i=1}^{n}(x_{2i}^\delta x_{2i+1}^\delta)$, and with reflecting boundaries $\cup_{i=1}^{n}(x_{2i-1}^\delta x_{2i}^\delta)$.
\begin{lemma}\label{lem::estimate}
There exists $C>0$ and $c>0$ such that
\[c\le \frac{\PP[ \cT^*_\delta \in E^*(\Omega_\delta;x_1^\delta,\ldots,x_{2n}^\delta;z_1^{\delta,+},\ldots,z_{n-1}^{\delta,+})]}{\Pi_{i=1}^{n-1}h^\delta_r(z_i^{\delta,+})}\le C.\]
\end{lemma}
\begin{proof}
For the lower bound, we consider $n-1$ disjoint open sets $\{U_1,\ldots,U_n\}$ such that $B(z_i,r)\cap \Omega\subset U_i$ and $\partial U_i\cap\partial D\subset (x_{2n}x_{1})\cup (B(z_i,r)\cap\partial \Omega)$. Denote $U_i^\delta$ a discrete approximation of $U_i$ on $\Omega_\delta$ (see Figure~\ref{fig::lowerbound}).
\begin{figure}[ht!]
\begin{center}
\includegraphics[width=0.5\textwidth]{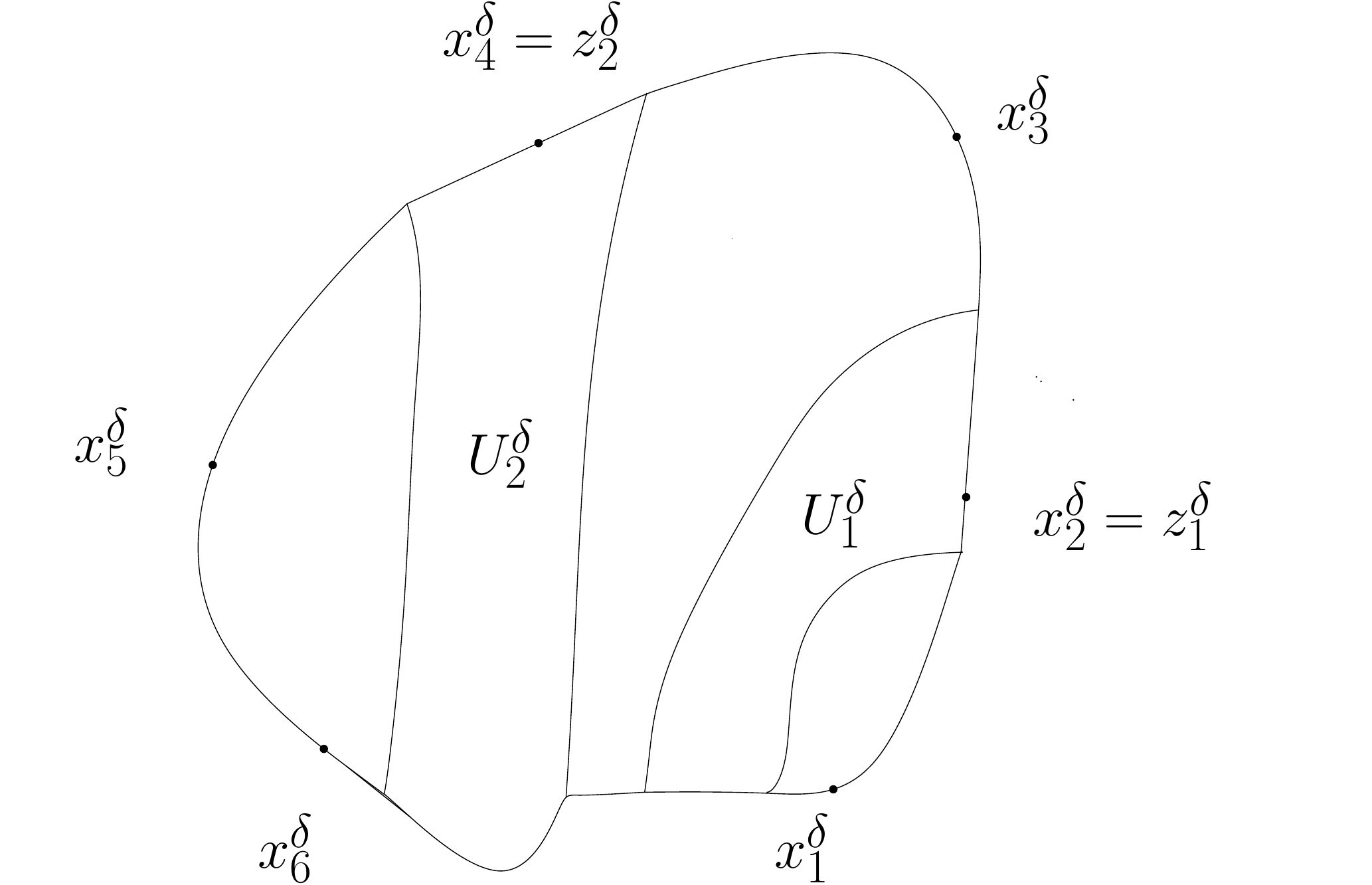}
\end{center}
\caption{\label{fig::lowerbound}This is an illustration of the discrete neighbour $U_1^\delta$ and $U_2^\delta$ when $n=3$.}
\end{figure}

Note that
\[\PP[\cT^*_\delta \in E^*(\Omega_\delta;x_1^\delta,\ldots,x_{2n}^\delta;z_1^{\delta,+},\ldots,z_{n-1}^{\delta,+})]\ge\PP[\cap_{i=1}^{n-1}\{\gamma_i^\delta\subset U_i^\delta\}].\]
For $v_\delta\in \Omega_\delta$, we define
\[h_i^\delta(v_\delta):=\PP_{v_\delta}[\LR_\delta \text{ leaves }U_i^\delta\text{ in }(x_{2n}^\delta x_{1}^\delta)],\]
where $\LR_\delta$, as usual, is a random walk killed at $\cup_{i=1}^{n}(x_{2i}^\delta x_{2i+1}^\delta)$, and with reflecting boundaries on $\cup_{i=1}^{n}(x_{2i-1}^\delta x_{2i}^\delta)$. Thus, by Wilson's algorithm, we have
\[\PP[\cT^*_\delta \in E^*(\Omega_\delta;x_1^\delta,\ldots,x_{2n}^\delta;z_1^{\delta,+},\ldots,z_{n-1}^{\delta,+})]\ge\Pi_{i=1}^{n-1}h_i^\delta(z_i^{\delta,+})\]
By Lemma~\ref{lem::ratio}, we have there exists a constant $c>0$, such that for every $1\le i\le n$, 
\[\frac{h_i^\delta(z_i^{\delta,+})}{h^\delta_r(z_i^{\delta,+})}\ge c.\]
The lower bound follows.

The upper bound is much more delicate and uses a careful combinatorial analysis together with choice of ordering in Wilson's algorithm. Roughly, we want to say that if we were to generate the spanning tree $\cT^*_\delta$ with Wilson's algorithm, we would roughly have $n-1$ independent walks which must at least verify the event defining $h_r^\delta (z_i^{\delta, +})$. The trouble is that it is \emph{a priori} possible for a walk starting from $z_i^{\delta, +}$ to come very close to $z_j^{\delta, +}$ (closer than distance $r$) for some $j \neq i$. This would prevent us from comparing the probability of $E^*$ with that of $n-1$ independent events. 

To deal with this, we will show up to a finite number of combinatorial possibilities, it is always possible to reveal the curves one at a time, in in a certain order such that, each time,  the walk we choose to reveal must escape the corresponding ball before touching the wired boundary (which includes the paths already revealed so far, as dictated by Wilson's algorithm). 

More precisely, we must consider events of the following type. Call an index $1\le i \le m-1$ \emph{good} if $ \gamma^\delta_i \cap (\cup_j B_j) = \emptyset$, with $B_j = B(z_j^{\delta, +}, r)$, and call it bad otherwise. Let $I$ denote the set of good indices. 
We then consider the events 
\begin{equation}\label{Eq:badindex}
\PP( \cT^*_\delta \in E^*, I = \{i_1, \ldots, i_m\} )
\end{equation}
where $1\le m \le n-1$, $1\le i_1 < \ldots < i_m \le n-1$. For each such event that we need to consider, we want to bound its probability by $C \prod_{i=1}^{n-1} h_r^\delta (z_i^{\delta, +})$.  To do so we use Wilson's algorithm with an order that depends specifically on the event we are considering. Namely, in a first stage, we reveal the \emph{good} branches $\gamma^\delta_j$ for $j\in I $, using independent random walks $\cR_j^\delta$ from the point $z_j^{\delta, +}$ for $j \in I$. Clearly, for $E^*$ to be satisfied, it must be the case that during this first stage, for each $j \in I$, $\cR_j^\delta$ escapes $B_j$. By independence, this has a probability $\prod_{j \in I} h_r^\delta (z^{\delta, +}_j)$. Now in a second stage, given the good branches, let us wire together the target arc $(x_{2n}, x_1)$ and the good branches; this forms a new wired boundary in Wilson's algorithm; see Figure~\ref{fig::upperbound}.
\begin{figure}
\begin{center}
\includegraphics[width=0.5\textwidth]{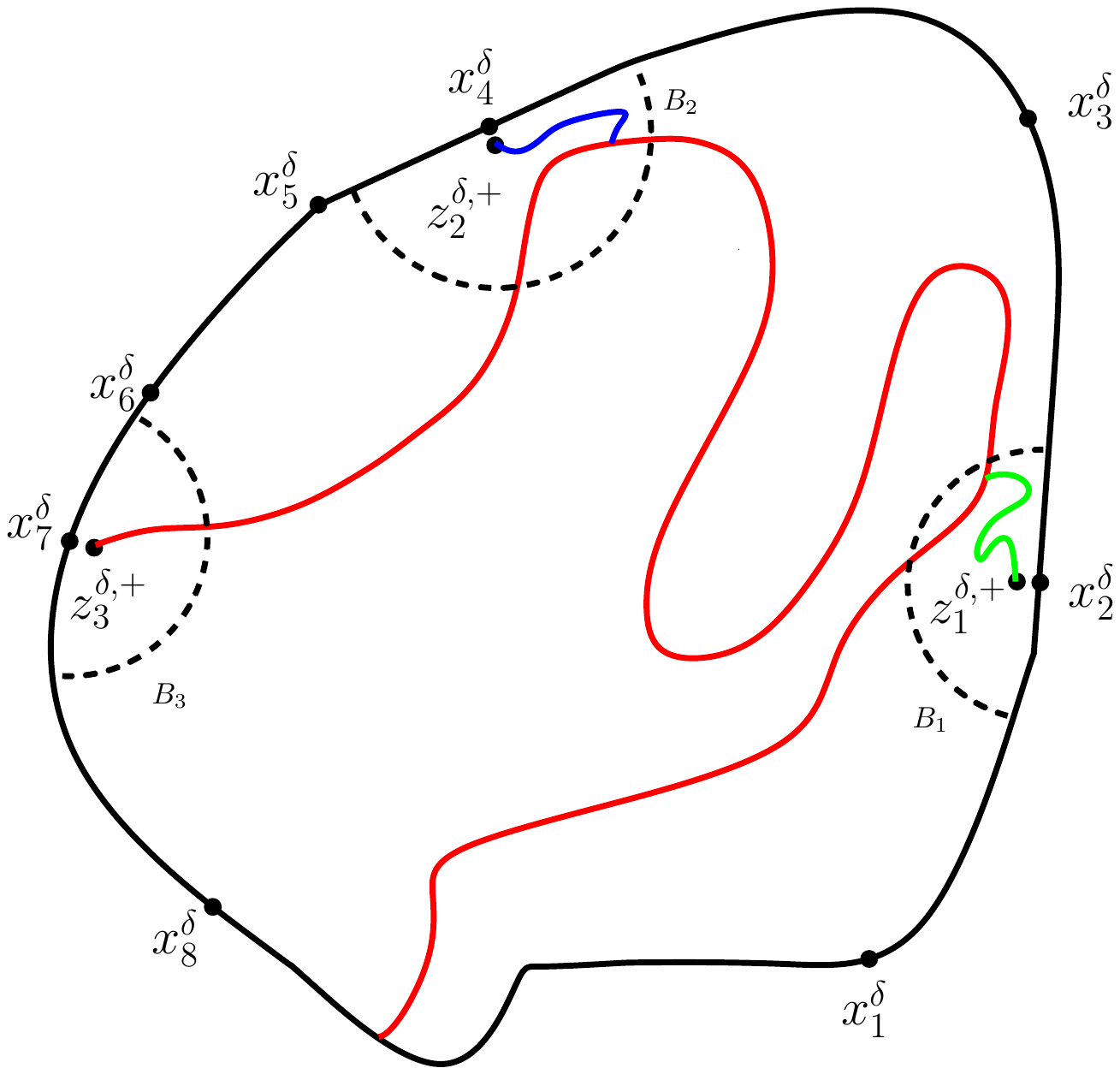}
\end{center}
\caption{\label{fig::upperbound}
An illustration of $\gamma_1^\delta, \gamma_2^\delta, \gamma_3^\delta$ for $n = 4$. In this case, only the green curve $\gamma_1^\delta$ is good initially, and so we start by revealing this in Wilson's algorithm. Once it is revealed and added to the wired boundary, we observe that the blue curve $\gamma_2^\delta$ becomes good, and so we reveal it at this point. Only once this is done, does the curve $\gamma_3^\delta$ become good and can be revealed. 
}
\end{figure}

We can now consider, among the remaining bad branches, those which are \emph{good} with respect to this new target. By iterating this procedure (i.e., induction) can reduce the number $n-1$ of branches that we need to consider;  at each stage of this induction we bound the conditional probability of $E^*$ by the product over all the good indices of $h_r^\delta (z^{\delta, +}_j)$. The only exception is when all the branches are bad.  

Thus let us assume that all the branches are bad, i.e., let us explain how to bound the probability in \eqref{Eq:badindex} when $I= \emptyset$. Consider $\gamma_1^\delta$, which is an oriented path from $z_1^{\delta , +}$ to the (current) wired boundary. We let $J$ be the index $j$ of the last (with respect to the chronological order of the path $\gamma_1^\delta$) ball $B_j$ such that $\gamma_1^\delta\cap B_{j}\neq\emptyset$, and $\gamma_{j}^\delta$ merges with $\gamma_1^\delta$ in $B_{j}$. If there is no such an index, we define $J:=1$. 

Note that by definition of $J$, on the event $\{ J = j\}$, for all $i\neq j$, $\gamma_i^\delta$ does not merge with $\gamma_{j}^\delta$ in $B_i$: this is because $\gamma_j^\delta$ coincides with the path $\gamma_1^\delta$ after the latter hits $B_j$, and $J$ was defined to be the last (chronologically) index where this property holds. However we mention that it is possible that $\gamma_{j}^\delta$ enters $B_i$, however in that case the path $\gamma^\delta_i$ does not merge with $\gamma_{j}^\delta$ there.
Either way, we add the event $\{ J = j\}$ to \eqref{Eq:badindex}, and try to bound $\PP( \cT^*_\delta \in E^*, I = \emptyset, J= j )$ by  
first choosing to reveal first $\gamma_{j}^\delta$ in Wilson's algorithm. Note that on this event, the walk $\cR_j^\delta$ must leave $B_j$, and so we pick up another term $h_r^\delta (z^{\delta, +}_j)$ in the conditional probability. We then iterate the argument one more step. (Having revealed $\gamma_j^\delta$, note the following delicate subtlety: although as explained above $\gamma_j^\delta$ \emph{may} have entered $B_i$ for some $i \neq j$, on the event $\{ J = j\}$ it is nevertheless the case that $\cR_i^\delta$ will need to escape $B_i$ before the touching the wired boundary, and that will bring a factor $h_r^\delta (z^{\delta, +}_j)$ to the subsequent conditional probability in the later steps of this induction).

Altogether, there is only a finite (combinatorial) number of cases that one needs to consider. In each case, at each step of the inductive argument the conditional probability is bounded by a product of term of the form $h_r^\delta (z^{\delta, +}_j)$ and the corresponding indices are then removed from the set of unexplored indices. The induction only stops when this set is empty, so in total, in each combinatorial case the overall probability is bounded by $\prod_{j=1}^{n-1} h_r^\delta (z^{\delta, +}_j)$. Summing over all the (finite number of) possible combinatorial cases, we obtain the desired upper-bound. 
\end{proof}

We have the following corollary, which is needed to prove the tightness of $\{\gamma_{1}^\delta,\ldots,\gamma_{n-1}^\delta\}_{\delta>0}$.
\begin{corollary}\label{coro::distance}
For $1\le i\le n-1$, let $(x_{2e_i}^\delta x_{2e_i+1}^\delta)$ be the boundary arc in which $\gamma_i^\delta$ terminates. Then, for every $\eps>0$, there exists $\eps_0>0$ such that
\[\PP[\text{there exists }i\text{ such that }d(\gamma_i^\delta,\cup_{j\neq i,e_i}(x_{2j}^\delta x_{2j+1}^\delta))< \eps_0|E^*(\Omega_\delta;x_1^\delta,\ldots,x_{2n}^\delta;z_1^{\delta,+},\ldots,z_{n-1}^{\delta,+})]<\epsilon.\]
\end{corollary}
\begin{proof}
We decompose over possible combinatorial cases considered in 
 in Lemma~\ref{lem::estimate} and suppose without loss of generality that  the corresponding order in which the curves are revealed is $\{\gamma_{1}^\delta,\ldots,\gamma_{n-1}^\delta\}_{\delta>0}$ (let $A$ be this event). 
 Note that this implies that, for $j<k$, $\gamma^\delta_{k}$ does not merge with $\gamma^\delta_{j}$ before leaving $B_k$. We will show the following statement by induction: for $1\le i\le n-1$, there exists $\eps_0>0$, such that
\[\PP[A; d(\gamma_i^\delta,\cup_{j\neq i,e_i}(x_{2j}^\delta x_{2j+1}^\delta))< \eps_0|E^*
]\le\eps_i.\]
for some choice of $\eps_i$ which can be made arbitrarily small. For $i=1$, note that by Lemma~\ref{lem::estimate} and Lemma~\ref{lem::Beur}, there exists $C>0$ and $c>0$ such that
\begin{align*}
\PP[A;d(\gamma_1^\delta,\cup_{j\neq 1,e_1}(x_{2j}^\delta x_{2j+1}^\delta))< \eps_0|E^*]
\le C\frac{\PP[d(\gamma_1^\delta,\cup_{j\neq 1,e_1}(x_{2j-1}^\delta x_{2j}^\delta))< \eps_0]}{h^\delta_r(z_1^{\delta,+})}\le C\eps_0^c.
\end{align*}
By choosing $\eps_1 = \eps_0^c$ small enough, the statement is proved for $i=1$.
Suppose the statement holds for $i\le l-1$.
We choose $\eps_l$ which is much smaller than $\eps_{l-1}$. Then, for $i=l$, by Lemma~\ref{lem::estimate} and Lemma~\ref{lem::Beur}, there exist $C>0$ and $c>0$, such that 
\begin{align*}
&\PP[A;d(\gamma_l^\delta,\cup_{j\neq l,e_l}(x_{2j}^\delta x_{2j+1}^\delta))< \eps_l | E^*]\\
\le&\sum_{i=1}^{l-1}\eps_{i}+C\frac{\PP[A;E^*\cap\{d(\gamma_l^\delta,\cup_{j\neq l,e_l}(x_{2j}^\delta x_{2j+1}^\delta))< \eps_l,d(\gamma_k^\delta,\cup_{j\neq k,e_k}(x_{2j}^\delta x_{2j+1}^\delta))>\eps_{l-1}\text{ for }k<l\}]}{\Pi_{k=1}^{n-1}h^\delta_r(z_l^{\delta,+})}\\
\le& \sum_{i=1}^{l-1}\eps_{i}+C\frac{\PP[d(\gamma_l^\delta,\cup_{j\neq l,e_l}(x_{2j}^\delta x_{2j+1}^\delta))< \eps_l\cond d(\gamma_k^\delta,\cup_{j\neq k,e_k}(x_{2j}^\delta x_{2j+1}^\delta))>\eps_{l-1}\text{ for }k<l]}{h^\delta_r(z_l^{\delta,+})}\\
\le&\sum_{i=1}^{l-1}\eps_{i}+ C\left(\frac{\eps_l}{\eps_{l-1}}\right)^c.
\end{align*}
By considering $\eps_l\ll \eps_{l-1}$, we complete the statement for $i=l$. Thus, by induction, we get the result (see Figure~\ref{fig::hitting estimate}).
\begin{figure}
\begin{center}
\includegraphics[width=0.5\textwidth]{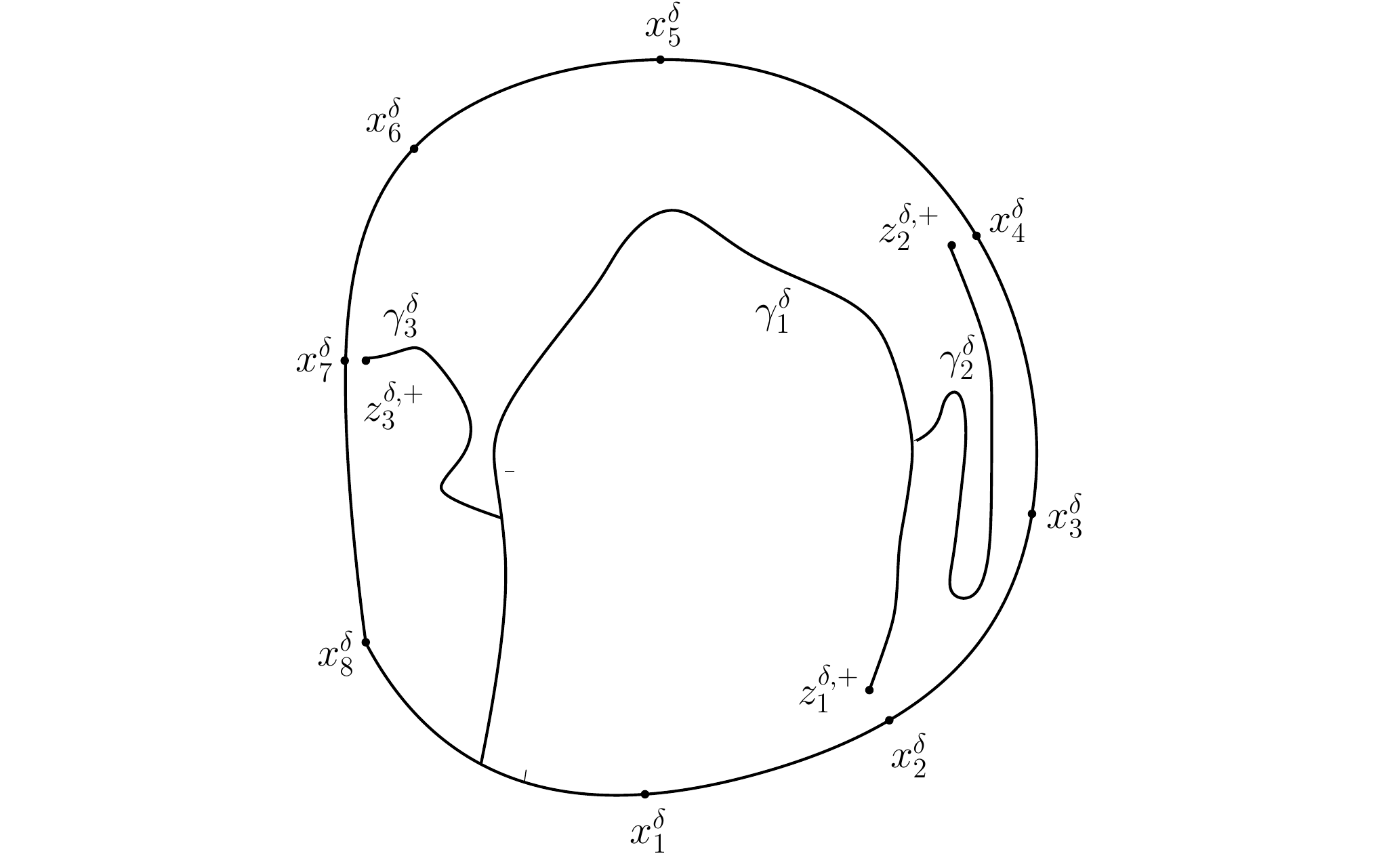}
\end{center}
\caption{\label{fig::hitting estimate}This is an illustration of the induction. We first generate $\gamma_1^\delta$ and then generate $\gamma_2^\delta$. Note that for $\gamma_2^\delta$, we restrict on the event that $d(\gamma_1^\delta,(x_3^\delta x_4^\delta)\cup(x_5^\delta x_6^\delta))>\eps_1$. }
\end{figure}
\end{proof}
Next, we consider the tightness of the  branches $\{\gamma_{1}^\delta,\ldots,\gamma_{n-1}^\delta\}_{\delta>0}$.
\begin{lemma}\label{lem::Gtightness}
The sequence of discrete curves $\{\gamma_{1}^\delta,\ldots,\gamma_{n-1}^\delta\}_{\delta>0}$ is tight. Moreover, for any subsequential limit $\{\gamma_{1},\ldots,\gamma_{n-1}\}$, almost surely, the limiting curve $\gamma_i$ is simple for each $1\le i \le n-1$, and only hits $\cup_{j=1}^{n}[x_{2j}x_{2j+1}]$ at its two endpoints.
\end{lemma}
\begin{proof}
The proof is very similar to the proofs of Lemma~\ref{lem::simple} and Lemma~\ref{lem::tightness}. We explain how to modify the proof of Lemma~\ref{lem::simple} to this setting in detail. Fix an $\eps$ and choose $\epsilon_0$ as in Corollary~\ref{coro::distance}. We fix a constant $\beta<\frac{\eps_0}{4}$. Let $\LA_{\delta}(z_0,\beta,\alpha;\cup_{i=1}^{n-1}\gamma_{i}^\delta)$ denote the event that there are two points $o_1,o_2\in\cup_{i=1}^{n-1}\gamma_{i}^\delta$ with
$o_1,o_2\in B(z_0,\alpha)$, such that the subarc of $\cup_{i=1}^{n-1}\gamma_{i}^\delta$ between $o_1$ and $o_2$ (which potentially goes via the wired boundary) is not contained in $B(z_0,\beta)$. We choose $\alpha$ such that $\alpha<\frac{1}{C^n_0}$, where $C_0$ is a large constant we will determine later. Once again we will generate $\cup_{i=1}^{n-1}\gamma_\delta^i$ according to the order given in Lemma~\ref{lem::estimate}; we may suppose that the order is given by $\gamma_1^\delta,\ldots,\gamma_{n-1}^\delta$ (let $A$ be this event). Recall that we denote by $(x^\delta_{2e_i}x^\delta_{2e_i+1})$ the arc which is hit by $\gamma^\delta_i$. Define 
\[E_i:=\{d(\gamma_l^\delta,\cup_{j\neq l,e_l}(x_{2j}^\delta x_{2j+1}^\delta))>\eps_0,\text{ for }1\le l\le i-1\}\]
and set $E = \cap_{i=1}^{n-1} E_i$.
 Let $\cN_i( r)$ denote the $r-$neighbourhood of $\cup_{j< i} \gamma_{j}^\delta$. Define  also the event 
\begin{align*}
\tilde E_i :=&\{\gamma^\delta_i \text{ hits }\cN_i (1/C_0^i), 
\text{then exits  } \cN_i(1/C_0^{i-1}) 
\text{ and then comes back to $\cN_i( 1/ C_0^i)$}\}.
\end{align*}
Note that 
\begin{align*}
&\PP[A \cap E\cap\LA_{\delta}(z_0,\beta,\alpha;\cup_{i=1}^{n-1}\gamma_{i}^\delta)|E^*(\Omega_\delta;x_1^\delta,\ldots,x_{2n}^\delta;z_1^{\delta,+},\ldots,z_{n-1}^{\delta,+})]\\
\le&\sum_{i=1}^{n-1}\PP[A \cap E\cap\LA_{\delta}(z_0,\beta,\alpha;\gamma_{i}^\delta)|E^*]+\sum_{i=1}^{n-1}\PP[A \cap E_i\cap\tilde E_i|E^*].
\end{align*}
By applying Lemma~\ref{lem::estimate}, we have that 
\begin{align}\label{eqn::1}
&\PP[A \cap E\cap \LA_{\delta}(z_0,\beta,\alpha;\gamma_{i}^\delta)|E^*(\Omega_\delta;x_1^\delta,\ldots,x_{2n}^\delta;z_1^{\delta,+},\ldots,z_{n-1}^{\delta,+})]\notag\\
\le &C\frac{\PP[E_{i}\bigcap (\cap_{j=1}^{i}\{\gamma_j^\delta\text{ does not connect to }(x_{2j}^\delta x_{2j+1}^\delta)\}) \bigcap\LA_{\delta}(z_0,\beta,\alpha;\gamma_{i}^\delta)]}{\Pi_{j=1}^{i}h^\delta_r(z_j^{\delta,+})}\notag\\
\le &C\frac{\PP[\{\gamma_i^\delta\text{ does not connect to }(x_{2i}^\delta x_{2i+1}^\delta)\}\cap\LA_{\delta}(z_0,\beta,\alpha;\gamma_{i}^\delta)\cond E_{i}\bigcap(\cap_{j=1}^{i-1}\{\gamma_j^\delta\text{ does not connect to }(x_{2j}^\delta x_{2j+1}^\delta)\}) ]}{h^\delta_r(z_i^{\delta,+})}\notag\\
\le& C\eps,
\end{align}
where the last inequality is obtained by applying Lemma~\ref{lem::simple} to $\Omega_\delta\setminus\cup_{j<i}\gamma_j^\delta$, and choosing $C_0$ large enough. 

Second, note that if $E_i\cap\tilde E_i\cap E^*\left(\Omega_\delta;x_1^\delta,\ldots,x_{2n}^\delta;z_1^{\delta,+},\ldots,z_{n-1}^{\delta,+}\right)\neq\emptyset$, we must have that the distance of $z_0$ and $(x_{2i}^\delta x_{2i+1}^\delta)$ is larger than $\eps_0-2\beta>\eps_0/2$. Then using once again Lemma~\ref{lem::estimate}, we have that 
\begin{align}\label{eqn::2}
&\PP [A; E_i\cap\tilde E_i|E^*(\Omega_\delta;x_1^\delta,\ldots,x_{2n}^\delta;z_1^{\delta,+},\ldots,z_{n-1}^{\delta,+})]\notag\\
\le &C\frac{\PP[\cap_{j=1}^{i}\{\gamma_j^\delta\text{ does not connect to }(x_{2j}^\delta x_{2j+1}^\delta)\}\cap E_i\cap\tilde E_i]}{\Pi_{j=1}^{i}h^\delta_r(z_j^{\delta,+})}\notag\\
\le& C\times C_0^{-c}\frac{\PP[\cap_{j=1}^{i-1}\{\gamma_j^\delta\text{ does not connect to }(x_{2j}^\delta x_{2j+1}^\delta)\}]}{\Pi_{j=1}^{i-1}h^\delta_r(z_j^{\delta,+})}\times\frac{h^\delta_{\eps_0/2}(z_i^{\delta,+})}{h^\delta_r(z_i^{\delta,+})}\notag\\
\le& C\eps,
\end{align}
where the second inequality is obtained from the discrete Beurling estimate given in Lemma~\ref{lem::Beur} and the last inequlity is obtained by choosing large constant $C_0$, Lemma~\ref{lem::ratio} and Lemma~\ref{lem::estimate}.
By Corollary~\ref{coro::distance}, we have
\begin{equation}\label{eqn::3}
\PP[E^c|E^*(\Omega_\delta;x_1^\delta,\ldots,x_{2n}^\delta;z_1^{\delta,+},\ldots,z_{n-1}^{\delta,+})]\le \epsilon.
\end{equation}
Combining~\eqref{eqn::1},~\eqref{eqn::2} and~\eqref{eqn::3}, we complete the proof. See Figure~\ref{fig::twocases}.
\begin{figure}[ht!]
\begin{center}
\includegraphics[width=0.8\textwidth]{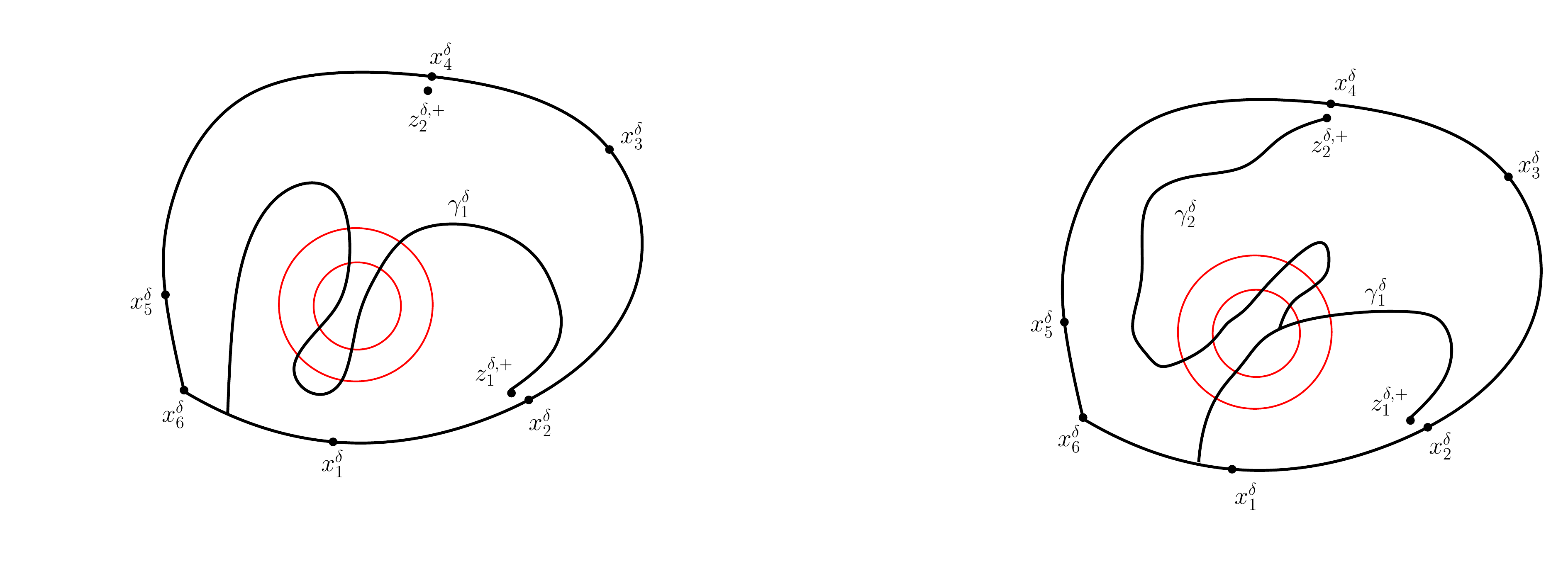}
\end{center}
\caption{\label{fig::twocases}This is an illustration of the two estimates~\eqref{eqn::1} and~\eqref{eqn::2}. The left one is the case that $\LA_{\delta}(z_0,\beta,\alpha_0;\gamma_{1}^\delta)$ occurs. The right one is the case that $\gamma^\delta_1$ and $\gamma_2^\delta$ together make an almost loop.}
\end{figure}
\end{proof}

\subsection{Identification of the limit}

In the remaining part of this section, we will identify the subsequential limit of the boundary branches. We will adopt the following strategy. First, we will consider the convergence of the branch in the $\ust$ with the boundary condition that $\cup_{i=1}^{n}(x_{2i}^\delta x_{2i+1}^\delta)$ are all wired to be a single boundary vertex, conditional on $z_1^{\delta,+}$ connects to $\cup_{j=2}^{n}(x^\delta_{2j}x^\delta_{2j+1})$. We denote this branch by $\tilde\gamma_1^\delta$; this has a unique scaling limit identified by the arguments for the case $n =2 $ (in Section \ref{sec::GFFtree}).
 Second, we will show that $\gamma_1^\delta$ is absolutely continuous with respect to $\tilde\gamma_1^\delta$ and identify the limiting Radon-Nikodym derivative, when the curve is restricted to a neighbourhood of $z_1^{\delta,+}$. In the third and final step, we will derive the law of $\gamma_1$ from the Radon-Nikodym derivative and combine with imaginary geometry (Theorem~\ref{thm::flowcoupling}) to complete the proof.

For the first step, we denote by $\tilde E(\Omega_\delta;x_1^\delta,\ldots,x_{2n}^\delta;z_1^{\delta,+})$ the set of spanning trees with the boundary condition that $\cup_{i=1}^{n}(x_{2i}^\delta x_{2i+1}^\delta)$ are all wired to be a single vertex and $z_1^{\delta,+}$ connects to $\cup_{l=2}^{n}(x_{2l}^\delta x_{2l+1}^\delta)$ within $\Omega_\delta$. We denote by $\tilde\gamma_1^\delta$ the branch connecting $z_1^{\delta,+}$ to $\cup_{l=2}^n(x_{2l}^\delta x_{2l+1}^\delta)$. We recall that we have fixed a conformal map $\varphi$ from $D$ onto $\HH$ such that $\varphi(x_1)<\cdots<\varphi(x_{2n})$. 

We also denote by $\phi_1(\cdot)=\phi_1(\cdot;\varphi(x_1),\ldots,\varphi(x_{2n}))$ the unique conformal map which maps the upper half plane $(\HH;\varphi(x_1), \ldots,\varphi(x_{2n}))$ onto a rectangle with vertical slits, such that $\phi_1(\varphi(x_{1}))=0$, $\phi_1(\varphi(x_{2}))=1$, $\phi_1(\varphi(x_{3}))=1+\ii K$, $\phi_1(\varphi(x_{2n}))=\ii K$, where $K$ is the height of the rectangle. See Figure~\ref{fig::conformal}.
\begin{figure}[ht!]
\begin{center}
\includegraphics[width=0.8\textwidth]{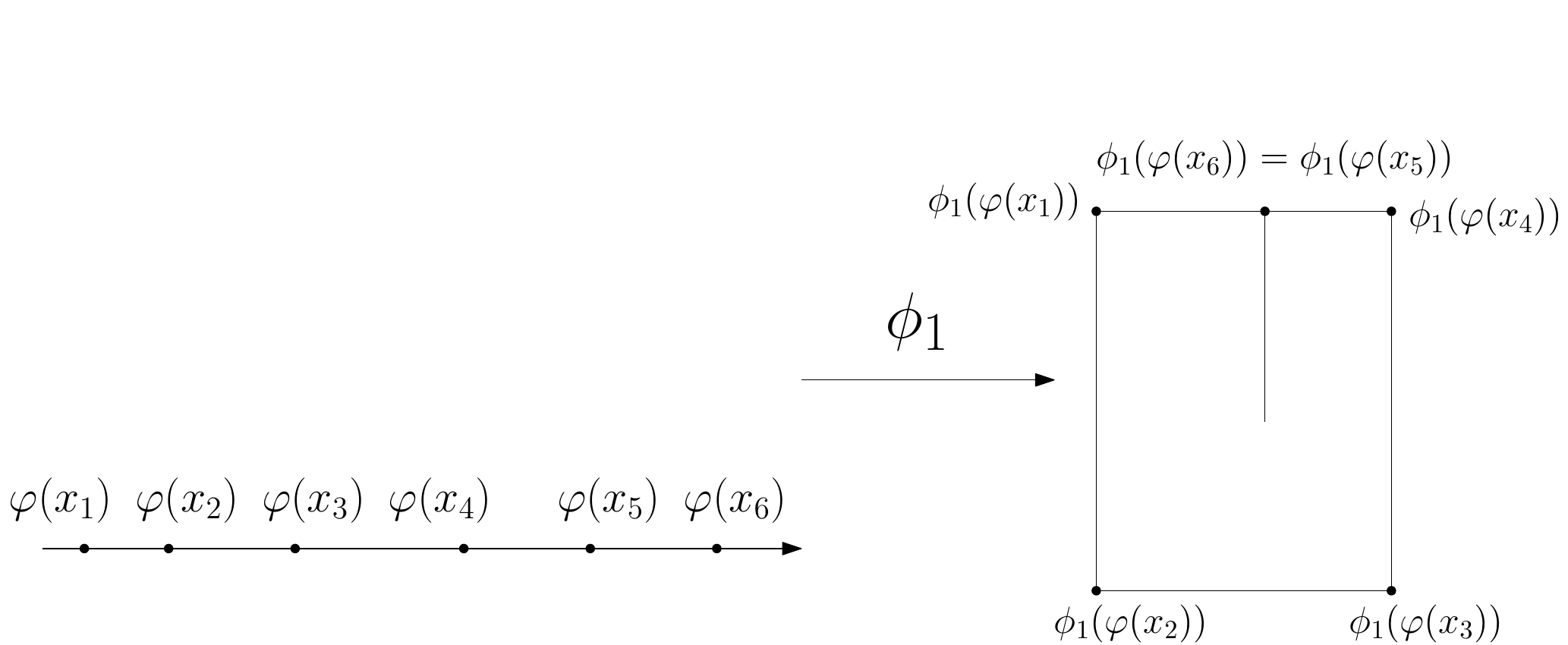}
\end{center}
	\caption{\label{fig::conformal}This is an illustration of $\phi_1$ when $n=3$. In this picture, $\phi_1(\varphi(x_{1}))=\ii$, $\phi_1(\varphi(x_{2}))=0$, $\phi_1(\varphi(x_{3}))=K$, $\phi_1(\varphi(x_{4}))=K+\ii$, where $K$ is the width of the rectangle. Note that $\phi_1$ maps $(\varphi(x_5)\varphi(x_6))$ to the vertical slit and $\phi_1(\varphi(x_5))=\phi(\varphi(x_6))$.}
\end{figure} 
(We will later write down an exact expression for $\phi_1$).
\begin{proposition}\label{P:RNinter}
The discrete curves $\{\tilde\gamma_1^\delta\}_{\delta>0}$ converge when $\delta\to 0$. We denote the limiting curve by $\tilde\gamma_1$. Then, almost surely $\tilde\gamma_1$ hits $\cup_{i=1}^n[x_{2i}x_{2i+1}]$ only at its two endpoints. The law of $\tilde\gamma_1$ can be described as follows: we denote by $\tilde W$ the driving function of $\varphi(\tilde\gamma_1)$ and denote by $(\tilde g_t:t\ge 0)$ the corresponding Loewner flow in the upper half plane, then we have
\[d \tilde W_t=\sqrt{2} B_t+2\partial_w\log(\partial_z \phi_1(z;\tilde g_t(\varphi(x_1)),\ldots,\tilde g_t(\varphi(x_{2n})))|_{z=w})|_{w=\tilde W_t}dt,\quad \tilde W_0=\varphi(z_1).\]
Moreover, $\tilde\gamma_1$ does not hit $(x_1x_2)\cup(x_{3}x_{4})$ almost surely.
\end{proposition}
\begin{proof}
This follows from the same arguments as in Proposition~\ref{prop::middle}. The fact that $\tilde\gamma_1$ does not hit $(x_1x_2)\cup(x_{3}x_{4})$ almost surely is from the absolutely continuity of $\tilde\gamma_1$ with respect to $\SLE_2(-1,-1;-1,-1)$. See the explicit expression of $\phi_1$ in~\eqref{eqn::conformalexpression}.
\end{proof}
Now, we come to the second step. By Lemma~\ref{lem::Gtightness}, we may suppose that the discrete curves $\{\gamma_1^\delta,\ldots,\gamma_{n-1}^\delta\}$ converge to $\{\gamma_1,\ldots,\gamma_{n-1}\}$ in law as $\delta\to 0$. We couple $\{\gamma_1^\delta,\ldots,\gamma_{n-1}^\delta\}$ and $\{\gamma_1,\ldots,\gamma_{n-1}\}$ together such that $\{\gamma_1^\delta,\ldots,\gamma_{n-1}^\delta\}$ converge to $\{\gamma_1,\ldots,\gamma_{n-1}\}$ almost surely as $\delta\to 0$ with respect to the sup norm topology up to reparametrisation defined in \eqref{eqn::curves_metric}. 

We fix a neighbourhood $U$ of $z_1$ such that $U$ does not contain $x_j$ for $x_j\neq z_1$. We denote by $U_\delta$ a discrete approximation. We denote by $\tau_U$ the hitting time by $\gamma_1$ of $\partial U$ and similarly denote by $\tau_{U}^\delta$ the corresponding time for $\gamma_1^\delta$. We couple $\{\tilde\gamma_1^\delta\}_{\delta>0}$ and $\tilde\gamma_1$ together such that $\tilde\gamma_1^\delta$ converges to $\tilde\gamma_1$ almost surely. Moreover, by changing $U$ slightly if necessary, we may assume without loss of generality that $\tau_{U}^\delta$ converges to $\tau_U$ almost surely. We define $\tilde\tau_U$ and $\tilde\tau_{U}^\delta$ for $\tilde\gamma_1$ and $\tilde\gamma_1^\delta$ similarly. Again, we may assume without loss of generality that $\tilde\tau_U^\delta$ converges to $\tilde\tau_U$ almost surely.

In the next lemma, we will derive the discrete Radon-Nikodym derivative of $\gamma_1[0,\tau_U]$ with respect to
 $\tilde\gamma_1[0,\tilde\tau_U]$. For every vertex $v_\delta\in \Omega_\delta$ and every edge $e_\delta$ which connects to $\cup_{i=1}^{n}(x_{2i}^\delta x_{2i+1}^\delta)$, we define
\[h_{\Omega_\delta}(v_\delta;(x^\delta_{2j}x^\delta_{2j+1})):=\PP_{v_\delta}[\LR_\delta\text{ hits }(x^\delta_{2j}x^\delta_{2j+1})\text{ before }\cup_{i\neq j}(x_{2i}^\delta x_{2i+1}^\delta)],\]
and define 
\[h_{\Omega_\delta}(v_\delta;e_\delta):=\PP_{v_\delta}[\LR_\delta\text{ hits }\cup_{j=1}^{n}(x^\delta_{2j}x^\delta_{2j+1})\text{ through }e_\delta],\]
where, as usual, $\LR_\delta$ is a random walk starting from $v_\delta$ and ending at $\cup_{i=1}^{n}(x^\delta_{2i}x^\delta_{2i+1})$, with reflecting boundaries along $\cup_{i=1}^{n}(x^\delta_{2i-1}x^\delta_{2i})$. 

We fix a simple discrete curve $(\eta_\delta(i):0\le i\le m)$ connecting $ z_1^{\delta,+}$ to $\partial U_\delta$ (i.e., the outer vertex boundary of $U_\delta$). We define $\Omega_\delta(\eta_\delta)$ to be $\Omega_\delta\setminus(z_1^\delta z_1^{\delta,+}\cup\eta_{\delta}[0,m-1])$. We define $z_1^{\delta,+}(\eta_\delta):=\eta_\delta(m)$ and define $z_i^{\delta,+}(\eta_\delta):=z_i^{\delta,+}$ for $2\le i\le n-1$.
\begin{lemma}\label{lem::disradon}
 We have the following explicit formula:
\begin{align*}
\frac{\PP[\gamma_1^\delta[0,\tau_U^\delta]=\eta_\delta]}{\PP[\tilde\gamma_1^\delta[0,\tilde\tau_U^\delta]=\eta_\delta]}=
&\frac{1-h_{\Omega_\delta}(z_1^{\delta,+};(x^\delta_{2}x^\delta_{3}))}{\det (F(\Omega_\delta))} 
\frac
{\det (F(\Omega_\delta (\eta_\delta)))}
{1-h_{\Omega_\delta(\eta_\delta)}(z_1^{\delta,+}(\eta_\delta);(x^\delta_{2}x^\delta_{3}))}.
\end{align*}
where the matrix $F(\Omega_\delta)$ is defined by 
$$
F(\Omega_\delta) = I-(h_{\Omega_\delta}(z_i^{\delta,+};(x_{2j}^\delta x_{2j+1}^\delta)))_{1\le i,j\le n-1}.
$$
\end{lemma}
\begin{proof}
This is a version of Fomin's determinantal formula (\cite{MR1837248} and see also Proposition 9.6.2 in \cite{LawlerLimic}). 
Note that if we write $E^* (\Omega_\delta; z_1^{\delta,+},\ldots,z_{n-1}^{\delta,+})$ for $E^*( \Omega_\delta; x_1,^\delta, \ldots, x_{2n}^\delta; z_1^{\delta,+},\ldots,z_{n-1}^{\delta,+})$ and likewise with $\tilde E$, 
\begin{align}\label{eqn::formula}
\frac{\PP[\gamma_1^\delta[0,\tau_U^\delta] =\eta_\delta\cond E^*(\Omega_\delta;z_1^{\delta,+},\ldots,z_{n-1}^{\delta,+})]}{\PP[\tilde\gamma_1^\delta[0,\tilde\tau_U^\delta]
=\eta_\delta\cond \tilde E(\Omega_\delta; z_1^{\delta,+})]}
&= \frac{|\tilde E(\Omega_\delta; z_1^{\delta,+})|}{| E^*(\Omega_\delta; z_1^{\delta,+},\ldots,z_{n-1}^{\delta,+})|}\frac{| E^*(\Omega_\delta(\eta_\delta); z_1^{\delta,+}(\eta_\delta),\ldots,z_{n-1}^{\delta,+}(\eta_\delta))|}{|\tilde E(\Omega_\delta(\eta_\delta); z_1^{\delta,+}(\eta_\delta))|}\notag\\
&= \frac{\PP[\tilde E(\Omega_\delta; z_1^{\delta,+})]}{\PP[E^*(\Omega_\delta; z_1^{\delta,+},\ldots,z_{n-1}^{\delta,+})]} \frac{\PP[E^*(\Omega_\delta(\eta_\delta); z_1^{\delta,+}(\eta_\delta),\ldots,z_{n-1}^{\delta,+}(\eta_\delta))]}{\PP[\tilde E(\Omega_\delta(\eta_\delta); z_1^{\delta,+}(\eta_\delta))]}.
\end{align}
Clearly the numerator of the first fraction and the denominator of the second are given by $1- h (z_1^{\delta, +} ;( x_2^\delta, x_3^\delta))$ in their respective domains $\Omega_\delta$ and $\Omega_\delta (\eta_\delta)$. The remaining two terms can be written as determinants, as follows.   
Recall the definition of $E^*(\Omega_\delta; z_1^{\delta,+},\ldots,z_{n-1}^{\delta,+})$ in~\eqref{eqn::eventdef}. By the inclusion-exclusion formula, we have
\begin{align*}
\PP[E^*(\Omega_\delta; z_1^{\delta,+},\ldots,z_{n-1}^{\delta,+})] 
=&1-\sum_{j=1}^{n-1}(-1)^{j-1}\sum_{(i_{1,1},\ldots,i_{1,s_1}),\ldots,(i_{j,1},\ldots,i_{j,s_j})}\PP[\LA_{i_{1,1},\ldots,i_{1,s_1}}\cap\cdots\cap \LA_{i_{j,1},\ldots,i_{j,s_j}}],
\end{align*}
where the sum is taken over all the non-intersecting set of numbers and we recall that the events $\cA_{j_1, \ldots, j_s}$ are defined in \eqref{eq:A}. Note that if we fix a subset of $\{1,\ldots,n-1\}$,  which we denote by $\{r_1,\ldots,r_k\}$, then every partition of $\{r_1,\ldots,r_k\}$ appears in the sum exactly one time. We note that each term in this sum corresponds uniquely to a permutation $\sigma$ of the indices $\{1, \ldots, n-1\}$: for instance the event $\LA_{i_{1,1},\ldots,i_{1,s_1}}$ specifies that the permutation contains a cycle mapping $i_{1,1}$ to $i_{1,2}, \ldots, i_{1, s_1}$ to $i_{1,1}$. Furthermore, the sign $(-1)^{j}$ then can be written as  $(-1)^k \text{sgn}(\sigma)$, where $k = |\text{Supp}(\sigma)| $ is the size of the support of $\sigma$. 
Thus, if we simply denote by $\sigma$ the corresponding partition and denote by $\LA_\sigma$ the intersection of the corresponding events, we have
\begin{align*}
\PP[E^*(\Omega_\delta; z_1^{\delta,+},\ldots,z_{n-1}^{\delta,+})]
=&1+\sum_{k=1}^{n-1}(-1)^{k}
\sum_{\sigma\in S_{n-1} : |\text{Supp}(\sigma)| = k}
\text{sgn}(\sigma)\PP[\LA_\sigma].
\end{align*}
Note that on the event $\LA_\sigma$, the branches from $z_{r_j}^{\delta,+}$ for $1\le j\le k$ are disjoint. This fact comes from our assumptions on $(\Omega;x_1,\ldots,x_{2n};z_1,\ldots,z_{n-1})$ that $z_j$ equals $x_{2j-1}$ or $x_{2j}$ for $1\le j\le n-1$ and that the branches are oriented. By~\cite[Theorem 6.2]{MR1837248} (see also see also Proposition 9.6.2 in \cite{LawlerLimic}), we know that for any edges $\{e^\delta_{1},\ldots,e^\delta_{k}\}$ such that $e^\delta_{j}$ connects to $(x_{2r_j}^\delta x_{2r_j+1}^\delta)$ for $1\le j\le k$, we have
\begin{align*}
&\sum_{\sigma:\sigma\text{ is a permutation on }\{r_1,\ldots,r_k\}}\text{sgn}(\sigma)\PP[\cap_{j=1}^{k}\{z_{r_j}^{\delta,+}\text{ connects to }(x^\delta_{2\sigma(r_j)}x^\delta_{2\sigma(r_j)+1})\text{ in }\Omega_\delta\text{ through }e_{\sigma(r_j)}^\delta\}]\\=&\det((h_{\Omega_\delta}(z_{r_j}^{\delta,+},e^\delta_{l}))_{1\le j,l\le k}).
\end{align*}
By summing over all the possible edges $\{e^\delta_{1},\ldots,e^\delta_{k}\}$, we deduce that 
\[\sum_{\sigma:\sigma\text{ is a permutation on }\{r_1,\ldots,r_k\}}\text{sgn}(\sigma)\PP[\LA_\sigma]=\det((h_{\Omega_\delta}(z_{r_j}^{\delta,+},(x_{2r_l}^\delta x_{2r_l+1}^\delta)))_{1\le j,l\le k}).\]
Thus, we have
\[\PP[E(\Omega_\delta; z_1^{\delta,+},\ldots,z_{n-1}^{\delta,+})]=1+\sum_{k=1}^{n-1}(-1)^{k}\sum_{1\le r_1<\ldots<r_{k}\le n-1}\det((h_{\Omega_\delta}(z_{r_j}^{\delta,+},(x_{2r_l}^\delta x_{2r_l+1}^\delta)))_{1\le j,l\le k}).\]
The conclusion comes from a simple observation: suppose $A=(A_{ij})_{1\le i,j\le n-1}$ is a $(n-1)\times (n-1)$ matrix, then we have
\[\det(I-rA)=1+\sum_{k=1}^{n-1}(-r)^{k}\sum_{1\le r_1<\ldots<r_{k}\le n-1}\det((A_{r_j,r_l})_{1\le j,l\le k}).\]
(This is well known and easily seen by expanding the left side and considering the coefficient of $r^j$ for $0\le j\le n-1$.) We conclude the proof by setting $r=1$ in the above.
\end{proof}
\begin{remark}
A similar Fomin determinantal expression for the Radon-Nikodym derivative also appears in~\cite{KarrilaKytolaPeltolaCorrelationsLERWUST} and~\cite{KarrilaUSTBranches}, although the boundary conditions considered there are somewhat different: in~\cite{KarrilaKytolaPeltolaCorrelationsLERWUST}, the authors considered the branches of the $\ust$ given their starting points and ending points. In~\cite{KarrilaUSTBranches}, the author gave an explicit formula of the law of multiple $\SLE_2$ depending on the construction in~\cite{KarrilaKytolaPeltolaCorrelationsLERWUST}. In other words the difference is that we do not specify the endpoints of the branches here. 
\end{remark}
The following corollary gives the explicit form of the Radon-Nikodym derivative of $\gamma_1^\delta[0,\tau_U^\delta]$ with respect to $\tilde\gamma_1^\delta$. For $\tilde\gamma_1^\delta$, we define 
\[M^\delta_{t}:=
\frac
{\det ( F(\Omega_\delta(\tilde\gamma_1^\delta[0,t\wedge\tilde\tau_U^\delta])))  }
{1-h_{\Omega_\delta(\tilde\gamma_1^\delta[0,t\wedge\tilde\tau_U^\delta])}(z_1^{\delta,+}(\tilde\gamma_1^\delta[0,t\wedge\tilde\tau_U^\delta]);(x^\delta_{2}x^\delta_{3}))}\]
\begin{corollary}\label{coro::disradon}
Suppose $r$ is a continuous function on the curves space $(\LP,d)$ defined in~\eqref{eqn::curves_metric}. Then, we have
\[\E[r(\gamma_1^\delta[0,t\wedge\tau_U^\delta])]=\E\left[\frac{M^\delta_{t}}{M^\delta_{0}}r(\tilde\gamma_1^\delta[0,t\wedge\tilde\tau_U^\delta])\right].\]
In particular, $(M^\delta_{t}:t\ge 0)$ is a martingale for $\tilde\gamma_1$ stopped at $\tilde\tau_U^\delta$. Furthermore, the martingale $M^\delta_t/M_0^\delta$ is uniformly bounded in time and in $\delta$. 
\end{corollary}
\begin{proof}
This comes from Lemma~\ref{lem::disradon} by summing over all possible $\eta_\delta$. The uniform bound comes from ~\eqref{eqn::formula} and Lemma~\ref{lem::estimate} (the uniformity in $t$ then follows in the same manner as in \eqref{Eq:boundratio}).
\end{proof}
Now, we come to the third step. We will first show the convergence of $M^\delta_{t}/M^\delta_{0}$. To do this, let us recall some basic facts about Schwarz--Christoffel maps. 
Recall that $\varphi$ is a fixed conformal map from $\Omega$ to $\HH$ such that $\varphi(x_1)\cdots<\varphi(x_{2n})$. Let $y_i = \varphi(x_i)$ and let $\mathbf{y}$ denote the point configuration $\mathbf{y} = (y_1, \ldots, y_{2n})$. For $1\le j\le n$, suppose $h_j$ is the bounded harmonic function satisfying the following boundary condition: $h_j$ equals $1$ on $(x_{2j}x_{2j+1})$ and $h_j$ equals $0$ on $\cup_{i\neq j}(x_{2i}x_{2i+1})$; furthermore $\partial_n (h_j\circ\varphi^{-1})$ equals $0$ on $\cup_{i=1}^n(y_{2i-1}, y_{2i})$. By the results in~\cite[Section 4]{LiuPeltolaWuUST}, we have that $1-h_j\circ \varphi^{-1}$ is the imaginary part of the conformal map $\phi_j$, which maps $(\HH;\varphi(x_1),\ldots,\varphi(x_{2n}))$ onto the rectangle with $n-2$ slits as in Figure \ref{fig::conformal}. The explicit form of the conformal map $\phi_j(\cdot)=\phi_j(\cdot;\mathbf{y})$ is given by 
\begin{equation}\label{eqn::conformalexpression}
\phi_j(z)= \phi_j(z ; \mathbf{y}) =i\int_{y_{2j}}^z\frac{\Pi_{l=1}^{n-2}(u-\mu_{l}^{(j)})}{\Pi_{j=1}^{2n}(u-y_j)^{1/2}}du\times \left(\int_{y_{2j+1}}^{y_{2j+2}}\frac{\Pi_{l=1}^{n-2}(u-\mu_{l}^{(j)})}{\Pi_{j=1}^{2n}(u-y_j)^{1/2}}du\right)^{-1},
\end{equation}
where the parameter $\mu_l^{(j)}$ (the pre-image of the tip of the $l$-th slit) is uniquely determined by the following equations:
\[
\phi_j(y_{2l+2}) - \phi_j(y_{2l+1}) =C_j 
\int_{y_{2l+1}}^{y_{2l+2}}\frac{\Pi_{l=1}^{n-2}(u-\mu_{l}^{(j)})}{\Pi_{j=1}^{2n}(u-y_j)^{1/2}}du=0,\quad\text{ for }l\neq j-1,j,\]
(with $C_j$ the factor on the right hand side of \eqref{eqn::conformalexpression}).
We will only need the fact that $\mu_l^{(j)}=\mu_l^{(j)}(\mathbf{y})$ is a smooth function of $\mathbf{y}$,  by the inverse function theorem.
Note that from the boundary condition, by the same argument as in Lemma~\ref{lem::conv_circ} (i.e., convergence of random walk to Brownian motion with orthogonal reflection along the corresponding arcs), we know that $h_{\Omega_\delta}(\cdot;(x^\delta_{2j}x^\delta_{2j+1}))$ converges to $1-\Im(\phi_j\circ\varphi)$ uniformly for $1\le j\le n$. For the same reason, if we let $\mathbf{y}_t = \tilde g_{t\wedge \tilde \tau_U} ( \mathbf{y})$, where $\tilde g_t$ is the Loewner flow associated with $\varphi (\tilde \gamma_1)$ in $\HH$, then 
  $h_{\Omega_\delta(\tilde\gamma_1^\delta[0,t\wedge\tilde\tau_U^\delta])}(\cdot;(x_{2j}^\delta x_{2j+1}^\delta))$ converges to 
\[1-\Im [ \phi_j \circ 
\tilde g_{t\wedge\tilde\tau_U}(\cdot ; \mathbf{y}_t)
\circ\varphi] 
\] 
uniformly for $1\le j\le n$.

In the following lemma, we will show the convergence of $M^\delta_{t\wedge \tilde\tau_U^\delta}/M^\delta_{0}$. 
Recall that we denote by $\tilde W$ the driving function of $\varphi(\tilde\gamma_1)$. To simplify the notation, we introduce the following definitions. Define $z_{1,t}:=\tilde\gamma_1(t\wedge\tilde\tau_U)$ and define $z_{i,t}:=z_i$ for $2\le i\le n-1$. For $1\le i\le n-1$ and for $1\le j\le n$, we define
\begin{align}
h_{ij}(t):=&\frac{\Pi_{l=1}^{n-2}(\tilde g_{t\wedge\tilde\tau_U}(\varphi(z_{i,t}))-\mu_{l}^{(j)}( \mathbf{y}_t ))}{\Pi_{1\le l\le 2n,x_l\neq z_{i,t}}(\tilde g_{t\wedge\tilde\tau_U}(\varphi(z_{i,t}))-y_{l,t})^{1/2}}\times \left(\int_{ y_{2j+1,t}}^{y_{2j+2,t}}\frac{\Pi_{l=1}^{n-2}(u-\mu_{l}^{(j)}( \mathbf{y}_t)}{\Pi_{j=1}^{j=2n}(u- y_{l, t})^{1/2}}du\right)^{-1},\label{hij}
\end{align}
where $\mathbf{y}_t = (y_{1,t}, \ldots, y_{n,t})$ (which is related to the expression in \eqref{eqn::conformalexpression} by taking the derivative with respect to $z$ and changing the domain). 
Define 
\[M_{t}:=\frac{\det\left(\left(h_{ij}(t)\right)_{1\le i,j\le n-1}\right)}{h_{11}(t)}\times \Pi_{i=2}^{n-1}\tilde g'_{t\wedge\tilde\tau_U}(\varphi(z_{i,t})).\] 
\begin{lemma}\label{lem::contradon}
The discrete Radon-Nikodym derivative $M^\delta_{t}/M^\delta_{0}$ below Lemma \ref{lem::disradon} converges to $M_{t}/M_{0}$ as $\delta\to 0$. Moreover, $M_t/M_0$ is the Radon-Nikodym derivate of $\gamma_1[0,t\wedge\tau_U]$ with respect to $\tilde\gamma_1[0,t\wedge\tilde\tau_U]$,
\end{lemma}
\begin{proof}
For the first statement, we have assumed that $\tilde\gamma_1^\delta$ converges to $\tilde\gamma_1$ and $\tilde\tau_U^\delta$ converges to $\tilde\tau_U$ almost surely as $\delta\to 0$. Thus, by the same proofs of~\cite[Corollary 3.8]{ChelkakWanMassiveLERW} and Lemma~\ref{lem::ratio}, we have
\begin{align*}
&\lim_{\delta\to 0}\frac{\det(I-(h_{\Omega_\delta(\tilde\gamma_1^\delta[0,t\wedge\tilde\tau_U^\delta])}(z_i^{\delta,+}(\tilde\gamma_1^\delta[0,t\wedge\tilde\tau_U^\delta]);(x_{2j}^\delta x_{2j+1}^\delta)))_{1\le i,j\le n-1})}{h_{\Omega_\delta(\tilde\gamma_1^\delta[0,t\wedge\tilde\tau_U^\delta])}(z_1^{\delta,+}(\tilde\gamma_1^\delta[0,t\wedge\tilde\tau_U^\delta]);(x^\delta_{2n}x^\delta_{1}))\times \Pi_{i=2}^{n-1}h_{\Omega_\delta}(z_i^{\delta,+};(x^\delta_{2n}x^\delta_{1}))}\\
=&\frac{\det\left(\left(h_{ij}(t)\right)_{1\le i,j\le n-1}\right)}{h_{1,n}(t)\times\Pi_{i=2}^{n-1}h_{i,n}(0)}\times \Pi_{i=2}^{n-1}(\tilde g'_{t\wedge\tilde\tau_U}(\varphi(z_{i,t})),
\end{align*}
and 
\begin{align*}
\lim_{\delta\to 0}\frac{1-h_{\Omega_\delta(\tilde\gamma_1^\delta[0,t\wedge\tilde\tau_U^\delta])}(z_1^{\delta,+}(\tilde\gamma_1^\delta[0,t\wedge\tilde\tau_U^\delta]);(x^\delta_{2}x^\delta_{3}))}{h_{\Omega_\delta(\tilde\gamma_1^\delta[0,t\wedge\tilde\tau_U^\delta])}(z_1^{\delta,+}(\tilde\gamma_1^\delta[0,t\wedge\tilde\tau_U^\delta]);(x^\delta_{2n}x^\delta_{1}))}=\frac{h_{11}(t)}{h_{1,n}(t)}.
\end{align*}
The similar estimates hold for $t=0$. This implies the first statement.

For the second statement, by Corollary \ref{coro::disradon}, we know that $\{M^\delta_{t\wedge \tilde\tau_U^\delta}/M^\delta_{0}\}_{\delta>0}$ is uniformly bounded. Thus, for any bounded continuous function $r:(\LP,d)\to \R$, we have
\[\E[r(\gamma_1[0,t\wedge\tau_U])]=\lim_{\delta\to 0}\E[r(\gamma^\delta_1[0,t\wedge\tau_U^\delta])]=\lim_{\delta\to 0}\E\left[\frac{M^\delta_{t}}{M^\delta_0}r(\tilde\gamma^\delta_1[0,t\wedge\tilde\tau_U^\delta])\right]=\E\left[\frac{M_{t}}{M_0}r(\tilde\gamma_1[0,t\wedge\tilde\tau_U])\right].\]
We conclude by the monotone class theorem. 
\end{proof}
With Lemma~\ref{lem::contradon} at hand, we can show Theorem~\ref{thm::genflow}.
\begin{proof}[Proof of Theorem~\ref{thm::genflow}]
Let $W$ be the driving function of $\gamma_1$ and denote by $(g_t:t\ge 0)$ the corresponding Loewner flow in $\HH$.
By Proposition \ref{P:RNinter},  Lemma~\ref{lem::contradon} and Girsanov's theorem, we know that before $\tau_U$, the driving function $W$ satisfies the following SDE:
\begin{equation}\label{eqn::aux_1}
dW_t=\sqrt 2 B_t+2\partial_z\log(\partial_z \phi_1(z; \mathbf{y}_t)))|_{z= W_t}dt+d\langle \log M_t, \sqrt 2 B_t\rangle,\quad W_0=\varphi(z_1).
\end{equation}
Using the definition of $\phi_1(z)$, 
\begin{equation}\label{eqn::aux12}
\partial_z\log(\partial_z \phi_1(z; \mathbf{y}_t))|_{z= W_t}= \sum_{j=1}^{n-2} \frac1{W_t - \mu_j^l (\mathbf{y}_t)} - \frac12 \sum_{j=1}^{2n} \frac1{W_t - y_{j,t}} .  
\end{equation}
Applying It\^o's formula to $\log h_{11} (t)$ (coming from the definition of $h_{11}(t)$ in \eqref{hij}),
we deduce that 
\begin{equation}\label{eqn::aux2}
2\partial_z\log(\partial_z \phi_1(z; \mathbf{y}_t))|_{z= W_t}dt= 
d\langle \log h_{11}(t), \sqrt 2 B_t\rangle.
\end{equation}
(This could also have been deduced without doing the explicit computation but using the fact that  $\partial_z \phi_1(z, \mathbf{y}_t)|_{z = W_t} = h_{11}(t)$ and the above covariation term depends only on the martingale part of $\log h_{11} (t)$, which is computed using ordinary calculus). It follows that the logarithmic term in \eqref{eqn::aux_1} cancels the contribution coming from $h_{11}(t)$ in the term $d\langle \log M_t, \sqrt{2} B_t\rangle$.  

Now, we treat the contribution coming from $\det\left(\left(h_{ij}(t)\right)_{1\le i,j\le n-1}\right)$ in $d\langle \log M_t, \sqrt{2} B_t\rangle$. For $1\le j\le n$, we define $P_j$ as the polynomial
\[P_j(z):=\Pi_{l=1}^{n-2}(z-\mu_{l}^{(j)}(  \mathbf{y}_t)).\]
By the explicit form of $h_{ij}(t)$ in \eqref{hij}, if we write $U_{i,t} = g_{t\wedge\tau_U}(\varphi(z_{i,t}))$ for $1\le i\le n-1$ (so $U_{1,t} = W_t$),
\begin{align}\label{eqn::aux_3}
\det\left(\left(h_{ij}(t)\right)_{1\le i,j\le n-1}\right)=&\det(P_j( U_{i,t})_{1\le i,j\le n-1})\times\frac{1}{\Pi_{l=1}^{2n}(W_{t\wedge\tau_U}- y_{l,t})^{1/2}} \times Q( \mathbf{y}_t) 
\end{align}
where $Q$ is a smooth function of $ \mathbf{y}_t$ (whose precise value depends in particular on whether $z_i$ is $x_{2i}$ or $x_{2i+1}$).

If we replace $U_{i,t}$ by a dummy variable $u_i$, then the determinant of $(P_j( U_{i,t})_{1\le i,j\le n-1}$ is a priori a polynomial in the variables $u_i$. Furthermore, when $u_i = u_1 = W_t$, the determinant equals $0$. This implies that there exists a polynomial $R  = R(U_{1,t}, \ldots, U_{n-1, t})$ (which further depends on $\mu_l^{(j)} (\mathbf{y}_t)$ implicitly, although we suppress this dependence in the notation), such that
\begin{equation}\label{eqn::aux_4}
\det(P_j( U_{i,t})_{1\le i,j\le n-1})=\Pi_{i=2}^{n-1}(W_{t\wedge\tau_U}- U_{i,t})\times R.
\end{equation}
However, considering the degree of $U_{1,t} = W_{t\wedge\tau_U}$ in the determinant on the left hand side, we know that $R$ is a polynomial of $(U_{2,t}, \ldots, U_{n-1, t})$ as well as $\mu_l^{(j)} (\mathbf{y}_t)$. 
Since $\mathbf{y}_t$ is a smooth function of time, and since $\mu_l^{(j)}$ is also smooth (as already mentioned) it follows that that $R$ is a smooth function of time. When taking the logarithm it will therefore only contributed a finite term, which is irrelevant for the computation of $d \langle \log M_t, \sqrt{2} B_t \rangle$. The same comment applies to $Q( \mathbf{y}_t)$. 

Altogether, combining~\eqref{eqn::aux_1},~\eqref{eqn::aux2},~\eqref{eqn::aux_3} and~\eqref{eqn::aux_4} together and applying It\^o's formula, we know that before time $\tau_U$, the driving function satisfies the SDE given in Theorem~\ref{thm::flowcoupling}. Thus, we can construct a coupling of $\gamma_1$ and the flow line starting from $\varphi(z_1)$, which we denote by $\gamma_1^I$, such that they are equal before exiting the neighbourhood $U$. We denote the probability measure for this coupling by $\PP_U$. Now, we choose an increasing sequence of neighbourhoods $\{U_i\}_{i\ge 1}$ such that $\cup_{i=1}^{\infty}U_i=\Omega$. Note that $\{\PP_{U_i}\}_{i\ge 1}$ are consistent. Therefore, we can define a coupling law $\PP$ that works on all of $\Omega$. Under the probability measure $\PP$, the curves $\gamma_1$ and $\gamma_1^I$ are equal before hitting $\cup_{j=2}^{n}(\varphi(x_{2j})\varphi(x_{2j+1}))$. Since both of these  curves are continuous at the hitting time of this boundary, we know that they are also equal at the hitting time. This finishes the proof of the convergence of $\{\gamma_1^\delta\}_{\delta>0}$.

To finish the proof, we couple $\{\gamma^\delta_1\}_{\delta>0}$ and $\gamma_1^I$ together such that $\gamma_1^\delta$ converges to $\gamma_1^I$ almost surely. Recall that we denote by $\Omega_L$ the connected component on the left hand side of $\HH\setminus\gamma_1^I$ and denote by $\Omega_R$ the component on the right hand side. We denote by $\Omega_L^\delta$ and $\Omega_R^\delta$ similarly. Note that $\Omega_L^\delta$ converges to $\Omega_L$ and $\Omega_R^\delta$ converges to $\Omega_R$ in the sense of~\eqref{eqn::converge1} and~\eqref{eqn::converge2}. Applying the domain Markov property and using the conditional description of flow lines Theorem~\ref{thm::flowcoupling}, we can iterate the above argument to show that conditionally on $\gamma_1^\delta$, $\gamma_2^\delta$ also converges to $\gamma_2^I$. Iterating this argument, 
this finishes the proof of Theorem~\ref{thm::genflow}. 
\end{proof}
\section{Convergence of the winding field}
\label{sec::Convergence of winding}

We fix a simply connected domain $(\Omega;x_1,\ldots,x_{2n};z_1,\ldots,z_{n-1})$ such that $\partial \Omega$ is $C^1$ and simple. As before, we suppose that a sequence of discrete simply connected domains $\{\Omega_\delta;x_1^\delta,\ldots,x_{2n}^\delta;z_1^\delta,\ldots,z_{n-1}^{\delta}\}$ converges to $(\Omega;x_1,\ldots,x_{2n};z_1,\ldots,z_{n-1})$ in the sense that \eqref{eqn::converge1} and \eqref{eqn::converge2} hold.  

In this section, we will show how the convergence of the tree proved in Theorem \ref{thm::gentreeconv} implies the convergence of its associated winding field and hence a new proof of Theorem \ref{thm::heightconv}.
Recall that Temperley's bijection in Lemma~\ref{lem::bij} identifies the dimer configuration to a uniform spanning tree  $\ust$ such that $\cup_{i=1}^{n}(x_{2i}^\delta x_{2i+1}^\delta)$ is wired to be one vertex, conditional on the event $E^*(\Omega_\delta;x_1^\delta,\ldots,x_{2n}^\delta;z_1^{\delta,+},\ldots,z_{n-1}^{\delta,+})$, 
as studied in Section~\ref{sec::Convergence of branches}.  Recall that we denote by $\gamma_i^\delta$ the branch connecting $z_i^{\delta,+}$ to $\cup_{j\neq i}(x_{2j}^\delta x_{2j+1}^\delta)$ stopped at the first time when it hits $\cup_{j\neq i}(x_{2j}^\delta x_{2j+1}^\delta)$. For $v_\delta\in\Omega_\delta$, we denote by $\tilde\gamma_{v_\delta}$ the branch connecting $v_\delta$ to $\cup_{j=1}^{n}(x_{2j}^\delta x_{2j+1}^\delta)\cup\cup_{i=1}^{n-1}\gamma_i^\delta$. We define the winding field $h_\delta$ on $\Omega_\delta$ as follows. We denote by $\Omega_{v_\delta}$ the connected component of $\Omega_\delta\setminus(\cup_{i=1}^{n-1}\gamma_i^\delta)$ containing $v_\delta$ and suppose the free boundary arc $(x_{2j+1}^\delta x_{2j+2}^\delta)\subset\partial \Omega_{v_\delta}$. We denote by $\gamma_{v_\delta}$ the branch connecting $x_{2j+2}^\delta$ to $v_\delta$. Define $h_\delta(v_\delta)$ to be the \textbf{intrinsic winding} of $\gamma_{v_\delta}$ (from $x_{2j+2}^\delta$ to $v_\delta$), i.e., the sum of the turning angles of $\gamma_{v_\delta}$ together with the winding of each edge if any, plus the winding of the boundary arc $(x_2^\delta x_{2j+2}^\delta)$. Recall also that in Temperley's bijection, this winding field corresponds to the dimer height function with respect to the ``winding reference flow''; see, e.g., 
Proposition 4.4 in \cite{BLR_Riemann2} where this is carefully written in the required generality (note that for this reference flow the total mass coming into or out of every vertex is $2\pi$ instead of the sometimes more conventional unit mass; this affects the scaling of the height function).  
 
In this section, we will show the convergence of the (non)-centered winding field $h_\delta$ on $\Omega_\delta$ and show that the limit has the same distribution as $(1/\chi)(h_\Omega + u_\Omega)$, where we recall that $\chi$ is the imaginary geometry parameter, which here is $\chi = 1/\sqrt{2}$.
\begin{theorem}\label{thm::windconv}
For every $f\in C_c^\infty(\Omega)$, and for every $k\ge 1$, we have
\[\E\left[ \left|(h_\delta,f)-\frac{1}{\chi}(h_\Omega + u_\Omega,f)\right|^k\right]\to 0,\quad\text{ as }\delta\to 0.\]
\end{theorem}

When we are not working on the square lattice but rather in the generalised setup of Section \ref{SS:gen}, this result requires centering: that is, we obtain
\[
\E\left[\left|(h_\delta,f)
-\E( (h_\delta, f)) - \frac{1}{\chi}(h_\Omega ,f)\right|^k\right]\to 0,\quad\text{ as }\delta\to 0.
\]

The proof is similar to the proof of the main result in~\cite[Theorem 1.2]{BLRdimers}. See also the main result of \cite{BLR_Riemann1, BLR_Riemann2}, where a similar strategy is adopted; as made clear in these papers the convergence of the spanning tree $\cT_\delta$ is enough to guarantee the existence of a scaling limit uniquely determined by the limit of $\cT_\delta$ provided a few simple estimates can be established; roughly speaking these are moment bounds on the winding and improbability for a given branch to visit a given small ball. Here we also wish to identify the limit of the winding field with $(1/\chi) (h_\Omega + u_\Omega)$, and this makes the problem very similar to \cite{BLRdimers}.

 We will adapt the same  strategy and only point out how to modify the technical lemmas there to our setting. We first establish~\cite[Lemma 4.3]{BLRdimers}, which is fundamental in the whole proof there. For $0<r'<R$, we define $A(v_\delta,r',R):=B(v_\delta,R)\setminus B(v_\delta,r')$ and for $r>0$, we define $\Omega_\delta(r):=\Omega_\delta\setminus\cup_{i=1}^{n-1} B(z_i^{\delta,+},2r)$.
\begin{lemma}\label{lem::withoutcondhit}
Fix $1<C_1<C_2$. For every $C_1r'\le R\le C_2r'$, there exists $\alpha=\alpha(C_1,C_2)\in (0,1)$ and $c$ such that for every $v_\delta\in \Omega_\delta(r)$, for $w_\delta\in A(v_\delta,r'+\frac{R-r'}{3},R-\frac{R-r'}{3})\cap \Omega_\delta$ such that $B(v_\delta,R)\cap\cup_{i=1}^{n}(x_{2i}^\delta x_{2i+1}^\delta)=\emptyset$, if we denote by $X$ the random walk starting from $w_\delta$ killed at $\cup_{i=1}^{n}(x_{2i}^\delta x_{2i+1}^\delta)$ and with reflecting boundaries along $\cup_{i=1}^{n}(x_{2i-1}^\delta x_{2i}^\delta)$, then, 
\begin{equation}\label{eqn::without}
\PP_{w_\delta}\left[X\text{ separates }\partial B(v_\delta,r') \text{ and }\partial B(v_\delta,R) \text{ before hitting }\partial A(v_\delta,r',R)\right]\ge \alpha.
\end{equation}
\end{lemma}

\begin{proof}
This follows directly from the Beurling estimate of Lemma \ref{lem::Beur}. 
\end{proof} 

The following two lemmas in the discrete setting are the main inputs in the proof of~\cite[Theorem 1.2]{BLRdimers}. Recall that $r$ is a small constant such that $\partial B(x_i,2r)\cap\partial B(x_j,2r)=\emptyset$ for $i\neq j$. Define $\Omega_\delta(r):=\{v_\delta: v_\delta\in \Omega_\delta\text{ and }d(v_\delta,z_i^\delta)>2r,\text{ for }1\le i\le n-1\}$.
\begin{lemma}\label{lem::hitball}
There exists $c>0$ and $C>0$, such that for every $v_\delta\in \Omega_\delta(r)$ and $z\in \Omega(r)$, for every $\eps>0$, we have
\[\PP[\left(\cup_{i=1}^{n-1}\gamma^\delta_i\cup\tilde\gamma_{v_\delta}\right)\cap B(z,\eps d)\neq\emptyset\cond E^*]\le C\eps^c,\]
where we define $d:=|v_\delta-z|\wedge d(v_\delta,\cup_{i=1}^{n}(x_{2i}^\delta x_{2i+1}^\delta))\wedge d(z,\cup_{i=1}^{n}(x_{2i}^\delta x_{2i+1}^\delta))$.
%
%
\end{lemma}

We will give the proof after the proof of Lemma~\ref{lem::condhit}.
\begin{lemma}\label{lem::moments}
For $s<t$ and $v_\delta\in \Omega_\delta(r)$, we define $W(\gamma_{v_\delta}[s,t],v_\delta)$ to be the topological winding of $\gamma_{v_\delta}$ during the time interval $[s,t]$ viewed from $v_\delta$. We denote by $\sigma_l$ the first time that $\gamma_{v_\delta}$ hits $B(v_\delta,l)$ and denote by $\tau_l$ the last time it is in $B(v_\delta,el)$. Then for every $k\ge 1$, there exists $M_k>0$ depending on $r$ and $k$, such that for every $l>0$ and $v_\delta\in \Omega_\delta(r)$, we have
\begin{equation}\label{eqn::moments}
\E\left[\sup_{Y\subset \gamma_{v_\delta}[\sigma_l,\tau_l]\cap B^c(v_\delta, l)}|W(Y,v_\delta)|^k\bigg| E^*
\right]\le M_k,
\end{equation}
where the supreme is over all continuous paths in $\gamma_{v_\delta}[\sigma_l,\tau_l]$ and $|W(Y,v_\delta)|$ is the  absolute value of the topological winding viewed from $v_\delta$.
\end{lemma}
\begin{proof}
The proof is very similar to the proof of~\cite[Lemma 4.13]{BLRdimers}. We summarize the proof briefly and show how to modify the basic inputs there to our setting. We will show the following stronger result: there exists $C>0$ such that for $n\ge 1$, we have
\begin{align}\label{eqn::winding1}
\PP\left[\sup_{Y\subset \gamma_{v_\delta}[\sigma_l,\tau_l]\cap B^c(v_\delta, l)}|W(Y,v_\delta)|>n\bigg| E^*
\right]\le Ce^{-cn}.
\end{align}
This can be divided into the following four estimates: we define $\tau_i(l)$ to be the last time that $\gamma_i^\delta$ hits $\partial B(v_\delta,el)$.
\begin{equation}\label{eqn::estimate1}
\PP\left[\sup_{Y\subset\cup_{i=1}^{n-1}\gamma_i^\delta[0,\tau_i(l)]\cap B^c(v_\delta, l)}|W(Y,v_\delta)|>n\bigg| E^*\right]\le Ce^{-cn}.
\end{equation}
Recall that we denote by $\tilde\gamma_{v_\delta}$ the segment of $\gamma_{v_\delta}$ stopped at its hitting time at $\cup_{i=1}^{n-1}\gamma_i^\delta\cup\cup_{i=1}^{n}(x_{2i}^\delta x_{2i+1}^\delta)$. 
\begin{itemize}
\item
On the event that $\cup_{i=1}^{n-1}\gamma_i^\delta\cap B(v_\delta,el)=\emptyset$, if we define $\tilde\sigma_l$ the first time that $\tilde\gamma_{v_\delta}$ hits $\partial B(v_\delta,l)$ and denote by $\tilde\tau_l$ the last time it hits $\partial B(v_\delta,el)$, then
\begin{equation}\label{eqn::estimate2}
\PP\left[\sup_{Y\subset \tilde\gamma_{v_\delta}[\tilde\sigma_l,\tilde\tau_l]\cap B^c(v_\delta, l)}|W(Y,v_\delta)|>n\bigg| \gamma_1^\delta,\ldots,\gamma_{n-1}^\delta\right]\le Ce^{-cn}.
\end{equation}
\item
On the event that $\cup_{i=1}^{n-1}\gamma_i^\delta\cap B(v_\delta,el)\neq\emptyset$ and $\cup_{i=1}^{n-1}\gamma_i^\delta\cap B(v_\delta,l)=\emptyset$, in this case if we define $\tilde\tau_l'$
to be the minimum of $\tilde\tau_l$ and the hitting time of $\tilde\gamma_{v_\delta}$ at $\cup_{i=1}^{n-1}\tilde\gamma_i^\delta[0,\tau_i(l)]$, then
\begin{equation}\label{eqn::estimate3}
\PP\left[\sup_{Y\subset \tilde\gamma_{v_\delta}[\tilde\sigma_l,\tilde\tau'_l]\cap B^c(v_\delta, l)}|W(Y,v_\delta)|>n\bigg| \gamma_1^\delta,\ldots,\gamma_{n-1}^\delta\right]\le Ce^{-cn}.
\end{equation}
\item
On the event that $\cup_{i=1}^{n-1}\gamma_i^\delta\cap B(v_\delta,l)\neq\emptyset$, 
\begin{equation}\label{eqn::estimate5}
\PP\left[\sup_{Y\subset \tilde\gamma_{v_\delta}[\tilde\sigma_l,\tilde\tau'_l]\cap B^c(v_\delta, l)}|W(Y,v_\delta)|>n\bigg| \gamma_1^\delta,\ldots,\gamma_{n-1}^\delta\right]\le Ce^{-cn}.
\end{equation}
\end{itemize}
Note that~\eqref{eqn::estimate5} can be obtained from~\eqref{eqn::estimate1} since the winding of segment of $\tilde\gamma_{v_\delta}$ is bounded by the winding of the segment of $\cup_{i=1}^{n-1}\gamma_i^\delta$. Thus, we only need to check~\eqref{eqn::estimate1},~\eqref{eqn::estimate2} and~\eqref{eqn::estimate3}.

For~\eqref{eqn::estimate1}, we will generate $\{\gamma_{1}^\delta,\ldots,\gamma_{n-1}^\delta\}_{\delta>0}$ in the order given in Lemma~\ref{lem::estimate}. We denote by $A$ the event that the order is given by $\gamma_{1}^\delta,\ldots,\gamma_{n-1}^\delta$. Note that this implies that, for $1\le j<k \le n-1$, we have $\gamma^\delta_{k}$ does not merge with $\gamma^\delta_{j}$ before hitting $\partial B(z^\delta_{k},r)$.
Thus, by Lemma~\ref{lem::estimate}, to show \eqref{eqn::estimate1} it suffices to establish the following estimate: for $1\le i\le n-1$, on the event $A$,
\begin{align}\label{eqn::estimate4}
&\PP\left[\sup_{Y\subset\gamma_i^\delta[0,\tau_i(l)]\cap B^c(v_\delta, l)}|W(Y,v_\delta)|>n,\,\gamma^\delta_{i}\text{ does not merge with }\cup_{j=1}^{i-1}\gamma^\delta_{j}\text{ before hitting }\partial B(z^\delta_{i},r)\bigg| \gamma_1^\delta,\ldots,\gamma_{i-1}^\delta\right]\notag\\ \le &Ch_r^\delta(z_\delta^+)e^{-cn}.
\end{align}

The main input of the proof is the following statement, see~\cite[Lemma 4.7, Lemma 4.8]{BLRdimers}. We state it in our setting. 
\begin{lemma}\label{lem::piecewind}
Fix $1<C_1<C_2$. For every $C_1r'\le R\le C_2r'$, there exists $\alpha=\alpha(C_1,C_2)\in (0,1)$ and $C$ such that for every $v_\delta\in \Omega_\delta(r)$ and $B(v_\delta,r')\cap \cup_{i=1}^{n}(x_{2i}^\delta x_{2i+1}^\delta)=\emptyset$ for $w_\delta\in A(v_\delta,r'+\frac{R-r'}{3},R-\frac{R-r'}{3})\cap \Omega_\delta$, if we denote by $\cR$ the random walk starting from $w_\delta$ and killed at $\cup_{i=1}^{n}(x_{2i}^\delta x_{2i+1}^\delta)$ with reflecting boundaries $\cup_{i=1}^{n}(x_{2i-1}^\delta x_{2i}^\delta)$, then we have the following two estimates.
\begin{itemize}
\item
Define $\tau$ to be the exit time of $A(v_\delta,r',R)\cap \Omega_\delta$, if we assume that $B(v_\delta, R)\cap \cup_{i=1}^{n}(x_{2i}^\delta x_{2i+1}^\delta)=\emptyset$, for all $u$ such that $\PP_{w_\delta}(\cR_\tau=u)>0$, for $n\ge 1$,
\begin{equation}\label{eqn::wind1}
\PP_{w_\delta}\left[\sup_{Y\subset \cR[0,\tau]}|W(Y,v_\delta)|>n|X_\tau=u\right]\le C(1-\alpha)^n,
\end{equation}
where the supreme is over all continuous paths $Y$ obtained by erasing portions from $\cR[0,\tau]$.
\item
Define $\tilde \tau$ to be the exit time by $\cR$ of $\Omega_\delta\setminus B(v_\delta,r')$. If we assume $r'<d(v_\delta,\cup_{i=1}^{n}(x_{2i}^\delta x_{2i+1}^\delta))<er'$, for $n\ge 1$,
\begin{equation}\label{eqn::wind2}
\PP_{w_\delta}\left[\sup_{Y\subset \cR[0,\tilde\tau]}|W(Y,v_\delta)|>n|\cR_{\tilde\tau}\in\partial \Omega_\delta\right]\le C(1-\alpha)^n;
\end{equation}
for all $u\in\partial B(v_\delta,r')$ such that $\PP_{w_\delta}(\cR_{\tilde\tau}=u)>0$, for $n\ge 1$,
\begin{equation}\label{eqn::wind3}
\PP_{w_\delta}\left[\sup_{Y\subset \cR[0,\tilde\tau]}|W(Y,v_\delta)|>n|\cR_{\tilde\tau}=u\right]\le C(1-\alpha)^n;
\end{equation}
where the supreme is over all continuous paths $Y$ obtained by erasing portions from $\cR[0,\tilde\tau]$.
\end{itemize}
\end{lemma}
The similar estimate holds for the random walk generating $\tilde\gamma_{v_\delta}$ if we replace $\Omega_\delta$ by $\Omega_\delta\setminus\cup_{i=1}^{n-1}\gamma_{i}^\delta$ and replace $\cup_{i=1}^{n}(x_{2i}^\delta x_{2i+1}^\delta)$ by $\cup_{i=1}^{n}(x_{2i}^\delta x_{2i+1}^\delta)\cup\cup_{i=1}^{n-1}\gamma_i^\delta$. Note that in our setting, the annulus $A(v_\delta,r'+\frac{R-r'}{3},R-\frac{R-r'}{3})$ may intersect the reflecting boundaries. 
Lemma~\ref{lem::piecewind} can be proved by the same argument as~\cite[Lemma 4.7, Lemma 4.8]{BLRdimers} and Lemma~\ref{lem::withoutcondhit}.
 
We continue the proof of Lemma~\ref{lem::moments}. A technical lemma shows that $\tilde\gamma_{v_\delta}[\tilde\sigma_l,\tilde\tau_l]$ and $\gamma_{i}^\delta[0,\tau_i(l)]$ for $1\le i\le n-1$ can be decomposed into a number of pieces of random walk trajectories, whose numbers and contribution to the winding can be controled. See more details in~\cite[Lemma 4.14]{BLRdimers}. We describe the decomposition for $\tilde\gamma_{v_\delta}[\tilde\sigma_l,\tilde\tau_l]$ since it is same for $\gamma_{i}^\delta[0,\tau_i(l)]$. First, we recall some fundamental notations in~\cite[Section 4]{BLRdimers}. We define $r_i:=e^il$ for $i\ge -1$ and define $C_i:=\partial B(v_\delta,r_i)\cap \Omega_\delta$. If $C_i$ intersects $\cup_{i=1}^{n-1}\gamma_i^\delta\cup\cup_{i=1}^{n}(x_{2i}^\delta x_{2i+1}^\delta)$, we define $C_i$ to be $\cup_{i=1}^{n-1}\gamma_i^\delta\cup\cup_{i=1}^{n}(x_{2i}^\delta x_{2i+1}^\delta)$ and we define the corresponding index to be $i_{max}$. We define a sequence of stopping times and the corresponding indices as follows. We define $\tau_1$ to be the first hitting time at $C_{-1}$ and define $i(1)=-1$. Having defined $\tau_k$, we define $\tau_{k+1}$ to be the first hitting time at $C_{i(k)+1}$ or $C_{i(k)-1}$ and define $i(k+1)$ to be the index of the circle it hits. If $i(k)=-1$, we define $\tau_{k+1}$ to be the hitting time of $\LR$ at $C_0$ and define $i(k+1):=0$. We define $k_{exit}$ to be the index such that $i(k_{exit})=i_{max}$ and define $k_{max}$ to be the index of the last hitting time at $C_1$. As in~\cite[Lemma 4.14]{BLRdimers}, we can define a index set $K$ such that
\begin{equation}\label{eqn::decom}
\tilde\gamma_{v_\delta}[\tilde\sigma_l,\tilde\tau_l]\cap B(v_\delta,l)^c\subset \cup_{k\in K}\LR[\tau_k,\tau_{k+1}].
\end{equation}
and there exists $C>0$ and $c>0$ such that
\begin{equation}\label{eqn::expdecay}
\PP[|K|\ge n]\le Ce^{-cn}.
\end{equation}
The proof of~\cite[Lemma 4.14]{BLRdimers} works exactly in our setting. 
The main input is the following Lemma, see~\cite[Corollary 4.5, Corollary 4.6]{BLRdimers}.
\begin{lemma}\label{lem::condhit}
Fix $1<C_1<C_2$. For every $C_1r'\le R\le C_2r'$, there exists $\alpha=\alpha(C_1,C_2)\in (0,1)$ and $c>0$ such that the following holds. Let $v_\delta\in D_\delta(r)$, and let $w_\delta\in  A(v_\delta,r'+\frac{R-r'}{3}, R-\frac{R-r'}{3})\cap \Omega_\delta$ such that $B(v_\delta,R)\cap\cup_{i=1}^{n}(x_{2i}^\delta x_{2i+1}^\delta)=\emptyset$. 
Define $\tau$ to be the exit time of $A(v_\delta,r',R)\cap \Omega_\delta$. Let $u\in\partial A(v_\delta,r',R)\cap \Omega_\delta$ such that $\PP_{w_\delta}(\cR_\tau=u)>0$. Then
\begin{equation}\label{eqn::crossing}
\PP_{w_\delta}\left[\cR\text{ separates }\partial B(v_\delta,r') \text{ and }\partial B(v_\delta,R) \text{ before hitting }\partial A(v_\delta,r',R)|\cR_\tau=u\right]\ge \alpha
\end{equation}
and
\begin{equation}\label{eqn::hit1}
\PP_{w_\delta}\left[\cR\text{ hits }\partial B(v_\delta,r')\text{ before }\partial B(v_\delta, R)|\cR_\tau=u\right]\ge \alpha,
\end{equation}
\end{lemma}
Lemma~\ref{lem::condhit} can be proved by the same argument as~\cite[Corollary 4.5, Corollary 4.6]{BLRdimers} and Lemma~\ref{lem::withoutcondhit}.

We continue the proof of Lemma~\ref{lem::moments}. Note that combining with~\eqref{eqn::wind1},~\eqref{eqn::wind2} and~\eqref{eqn::wind3},~\eqref{eqn::decom} and~\eqref{eqn::expdecay}, we get~\eqref{eqn::estimate2} and~\eqref{eqn::estimate3} directly. For~\eqref{eqn::estimate4}, note that the contribution of the part of the random walk generating $\gamma_i^\delta$ before hitting $B(z_i^\delta,r)$ to the winding is less $2\pi$, since by our assumption $z_i^\delta\notin B(v_\delta,2r)$. Thus, we can apply the same argument for pieces of the random walk starting from $\partial B(z_i^\delta,r)$ without conditioning. This leads to~\eqref{eqn::estimate4} and completes the proof.
\end{proof}
 We still need to show Lemma~\ref{lem::hitball}.
\begin{proof}[Proof of Lemma~\ref{lem::hitball}]
It suffices to show that for every $1\le i\le n$,
\begin{equation}\label{eqn::hit1}
\PP[\gamma_{i}^\delta\cap B(z,\eps d)\neq\emptyset,\,\cup_{j=1}^{i-1}\gamma_{j}^\delta\cap B(z,\sqrt\eps d)=\emptyset\cond E^*]\le C\eps^c,
\end{equation}
and
\begin{equation}\label{eqn::hit2}
\PP[\tilde\gamma_{v_\delta}\cap B(z,\eps d)\neq\emptyset,\,\cup_{i=1}^{n-1}\gamma_{i}^\delta\cap B(z,\sqrt\eps d)=\emptyset\cond E^*]\le C\eps^c
\end{equation}
For~\eqref{eqn::hit1}, we will once again generate $\{\gamma_{1}^\delta,\ldots,\gamma_{n-1}^\delta\}_{\delta>0}$ in the order given in Lemma~\ref{lem::estimate}. Define the event 
\[A_i:=\{\gamma^\delta_{k}\text{ does not merge with }\gamma^\delta_{j}\text{ before hitting }\partial B(z^\delta_{k},r)\text{ for }1\le j<k\le i\}.\] By Lemma~\ref{lem::estimate}, there exists $C>0$, such that
\begin{align}\label{eqn::hit}
&\PP[\gamma_{i}^\delta\cap B(z,\eps d)\neq\emptyset,\,\cup_{j=1}^{i-1}\gamma_{j}^\delta\cap B(z,\sqrt\eps d)=\emptyset, A\cond E^*]
\le 
C\frac{\PP[\gamma_{i}^\delta\cap B(z,\eps d)\neq\emptyset,\,\cup_{j=1}^{i-1}\gamma_{j}^\delta\cap B(z,\sqrt\eps d)=\emptyset,\,A_i]}{\Pi_{j=1}^{i}h^\delta_r(z_i^{\delta,+})}.
\end{align}
We will prove~\eqref{eqn::hit} by induction. 
We define $r_j:=e^{j}\eps d$ for $1\le j\le [-\log\sqrt\eps]$. Note that on the event $\{\gamma_{1}^\delta\cap B(z,\eps d)\neq\emptyset\}$, the random walk $\LR$ generating $\gamma_1^\delta$ does not make a full turn in $A(z_\delta,r_{i-1},r_{i})$ for every $1\le i\le [-\log\sqrt\eps]$ after the last hitting of $\partial B(z_\delta, \eps d)$. We denote by $\tau_i$ the last hitting time of $B(z_\delta,\frac{r_{i-1}+r_{i}}{2})$ after the last hitting time of $\partial B(z_\delta,\eps d)$ and denote by $\tau_i'$ the first hitting time of $\partial B(z_\delta, r_i)$ after $\tau_i$. We will consider the portion of $\LR$ after the first hitting of $\partial B(z_\delta, \sqrt\eps d)$. Then, by~\eqref{eqn::crossing}, for every $w_\delta\in\partial B(z_\delta,\sqrt\eps d)$, we have
\begin{align*}
&\PP_{w_\delta}[\LR[\tau_i,\tau'_i]\text{ does not separate }\partial B(z_\delta,r_{i-1})\text{ and }\partial B(z_\delta,r_i))\text{ for }1\le i\le [-\log\sqrt\eps]]\\
=&\E_{w_\delta}[\PP_{w_\delta}[\LR[\tau_i,\tau'_i]\text{ does not separate }\partial B(z_\delta,r_{i-1})\text{ and }\partial B(z_\delta,r_i))\text{ for }1\le i\le [-\log\sqrt\eps]\cond \LR_{\tau'_1},\ldots,\LR_{\tau'_{[-\log\sqrt\eps]}}]\\
=&\E_{w_\delta}[\Pi_{i=1}^{[-\log\sqrt\eps]}\PP_{\LR_{\tau_i}}[\LR[\tau_i,\tau'_i]\text{ does not separate }\partial B(z_\delta,r_{i-1})\text{ and }\partial B(z_\delta,r_i))\cond \LR_{\tau'_i}]\\
\le& (1-\alpha)^{[s]}\\
\le& C\eps^c,\quad\text{ for some }C>0\text{ and }c>0.
\end{align*}
Then, by the Markov property of $\LR$ at the hitting time of $B(z_\delta,\sqrt\eps d)$, there exists $C>0$ and $c>0$,
\begin{align}\label{eqn::case21}
&\frac{\PP[\gamma_{1}^\delta\text{ hits }B(z_\delta,\eps d)]}{h^\delta_r(z_1^{\delta,+})}\notag\\
\le&\max_{w_\delta\in\partial B(z_1^\delta,r)}\PP_{w_\delta}[\LR[\tau_i,\tau'_i]\text{ does not separate }\partial B(z_\delta,r_{i-1})\text{ and }\partial B(z_\delta,r_i))\text{ for }1\le i\le [-\log\sqrt\eps]]\notag\\
\le &C\eps^c
\end{align}

Suppose~\eqref{eqn::hit1} holds for $1\le i\le m-1$. For $i=m$, on the event 
$\{\cup_{j=1}^{m-1}\gamma_{j}^\delta\cap B(z,\sqrt\eps d)=\emptyset,\,A_{m-1}\}$,
we need to show that there exists $C>0$ and $c>0$ such that
\begin{align*}
\PP[\gamma_{m}^\delta\cap B(z,\eps d)\neq\emptyset,\,A_m\cond\gamma_1^\delta\ldots,\gamma_{m-1}^\delta]\le C\eps^ch_r^\delta(z_m^{\delta,+}).
\end{align*}
This can be shown by the same argument for $i=1$.
Thus, by induction, we get~\eqref{eqn::hit1}.

For~\eqref{eqn::hit2}, note that on the event $\{\cup_{i=1}^{n-1}\gamma_{i}^\delta\cap B(z,\sqrt\eps d)=\emptyset\}$, $\tilde\gamma_{v_\delta}$ hits $\partial B(z,\eps d)$ before merging with $\cup_{i=1}^{n-1}\gamma_i^\delta$. In this case, ~\eqref{eqn::hit2} can be shown by the same argument as~\eqref{eqn::hit1}.
This completes the proof.
\end{proof}
To show Theorem~\ref{thm::windconv}, we first need to show the convergence of the truncated winding field in the continuous setting. Then, by applying the same argument as in \cite{BLRdimers} (see also~\cite[Theorem 6.1]{BLR_Riemann1} where the same idea is used, as well as \cite{OffcriticalDimers}), we will get Theorem~\ref{thm::windconv} directly. Recall the continuous tree $\LT$ defined in~\ref{thm::gentreeconv}, almost surely, for almost everywhere $z\in \Omega$, there is a unique path connecting $z$ to $\cup_{i=1}^n(x_{2i}x_{2i+1})\cup\cup_{i=1}^{n-1}\gamma_i$, which we denote by $\tilde\gamma_z$. From the construction, we know that almost surely,  $\tilde\gamma_z$ hits $\cup_{i=1}^n(x_{2i}x_{2i+1})\cup\cup_{i=1}^{n-1}\gamma_i$ only at its endpoint. We extend $\tilde\gamma_z$ as follows: we denote by $\Omega_z$ the connected component of $\Omega\setminus (\cup_{i=1}^{n-1}\gamma_i)$ which contains $z$ and suppose $(x_{2j+1}x_{2j+2})\subset\partial \Omega_z$. We add to $\tilde \gamma_z$ the ``boundary segment'' (including possibly part of $\cup_{i=1}^{n-1}\gamma_i$) connecting the endpoint of $\tilde \gamma_z$ to $x_{2j+2}$; we call the resulting path $\gamma_z$.
We then parametrise $\gamma_z$ by $[-1,\infty)$, such that the segment $\gamma[-1,0]$ is from $x_{2j+2}$ to the endpoint of $\tilde  \gamma_z$. For $t\ge 0$, we parametrise $\gamma_z$ by its capacity in $\Omega$, that is, for $t\ge 0$, we have 
\[-\log(R(z,\Omega\setminus\gamma_z[-1,t]))=t-\log(R(z,\Omega)),\] where $R(\cdot,\cdot)$ is the conformal radius. Note that there are continuous function $r(s)$ and $\theta(s)$, such that
\[z-\gamma_z(s)=r(s)e^{i\theta(s)}.\]
We define the truncated winding field as 
\[h_t(z):=\theta(t)-\theta(-1)+\text{the winding of $(x_2 x_{2j+2})$}.\] 
We need to prove the following Proposition. See~\cite[Theorem 3.1]{BLRdimers} and we state it in our setting (and in a weaker version).
\begin{proposition}~\cite[Theorem 3.1]{BLRdimers}
For every $f\in C_c^\infty(\Omega)$, and for every $k\ge 1$, we have
\[\E\left[\left|(h_t,f)-\frac{1}{\chi}(h_\Omega+u_\Omega,f)\right|^k\right]\to 0,\quad\text{ as }\delta\to 0.\]
\end{proposition}
We will not give a proof in this paper since the proof of~\cite[Theorem 3.1]{BLRdimers} works exactly in our setting. We only point out why the basic inputs  also hold in our setting and derive the convergence of the expectation:
\begin{equation}\label{eqn::convexp}
\lim_{t\to\infty}\E[h_t(z)]=\frac{1}{\chi}u_\Omega(z;x_1,\ldots,x_{2n};z_1,\ldots,z_{n-1}).
\end{equation}
Note that~\cite[Lemma 2.9]{BLRdimers} is obtained from Lemma~\ref{lem::hitball} and~\cite[Lemma 2.10]{BLRdimers} is obtained from~\eqref{eqn::winding1}. It remains to prove the analogue of~\cite[Theorem 2.11]{BLRdimers}, which is the following:
\begin{lemma}\label{lem::windanddistance}
Suppose $\Omega$ is the unit disk and $v=0$. 
\begin{itemize}
\item
For every $k\ge 1$, there exists $C_k>0$ such that
\[\E[|h_t(0)|^k]\le C_k (1+t)^k.\]
\item
There exists $C>0$ and $c>0$ such that for all $0<s<t$, we have
\[\PP[|\gamma_0(t)|>e^{-t+s}]\le Ce^{-cs}.\]
\end{itemize}
\end{lemma}
\begin{proof}
For the first term, by Koebe's $1/4$ theorem and Schwartz's lemma, we have
\[\frac{e^{-t}}{4}\le d(0,\gamma_0(t))\le e^{-t}.\]
We define $r_i:=e^{-t+i}$ for $0\le i\le [t]$ and define $r_{[t]+1}:=\partial \Omega$. Define $\tau_i$ to be the first time of $\tilde\gamma_0$ at $\partial B(0,r_i)$. Then, by Lemma~\ref{lem::moments} and the convergence of disrete branches, we have
\[\E[|h_t(0)|^k]^{1/k}\le \sum_{i=0}^{[t]}\E[|h_{\tau_{i}}(0)-h_{\tau_{i+1}}(0)|^k]^{1/k}+2\pi\le M_kt+2\pi.\]
This completes the proof of the first item.

For the second term, recall that we denote by $\{\gamma_1^\delta,\ldots,\gamma^\delta_{n-1}\}$ the discrete branches starting from $z^{\delta,+}_1,\ldots,z^{\delta,+}_{n-1}$ to $(x^\delta_{2n}x^\delta_{1})$. Recall that we denote by $\tilde\gamma_0^\delta$ the branch from $0$ to $\cup_{i=1}^{n}(x^\delta_{2i} x^\delta_{2i+1})\cup\cup_{i=1}^{n-1}\gamma_{i}^\delta$. By~\eqref{eqn::hit1}, we have
\[\PP[d(\cup_{i=1}^{n-1}\gamma^\delta_i,0)<e^{-t}]\le Ce^{-ct}.\] 
On the event $\{d(\cup_{i=1}^{n-1}\gamma^\delta_i,0)>e^{-t}\}\cap\{|\gamma^\delta_0(t)|>e^{-t+s}\}$, we have $\tilde\gamma^\delta_0$ hits $\partial B(0,e^{-t+s})$ before the last hitting of $\partial B(0,e^{-t})$. We denote by $X$ the random walk generating $\tilde\gamma^\delta_0$ and denote by $\tilde \tau_i$ the last hitting time of $B(0,\frac{r_{i}+r_{i-1}}{2})$ for $1\le i\le [s]$ before the last hitting time of $\partial B(0,e^{-t})$. We denote by $\tilde\tau_i'$ the first hitting time of $\partial B(0,r_{i-1})$ after $\tilde\tau_i$. Then, by~\eqref{eqn::crossing}, we have
\begin{align*}
&\PP[\{d(\cup_{i=1}^{n-1}\gamma^\delta_i,0)>e^{-t}\}\cap\{|\gamma^\delta_0(t)|>e^{-t+s}\}]\\
\le&\PP[X[\tilde\tau_i,\tilde\tau'_i]\text{ does not separate }\partial B(0,r_{i})\text{ and }\partial B(0,r_{i-1})\text{ for }1\le i\le [s]]\\
=&\E[\PP[\LR[\tilde\tau_i,\tilde\tau'_i]\text{ does not separate }\partial B(0,r_{i})\text{ and }\partial B(0,r_{i-1})\text{ for }1\le i\le [s]\cond \LR_{\tau'_1},\ldots,\LR_{\tau'_{[s]}}]\\
=&\E[\Pi_{i=1}^{[s]}\PP_{\LR_{\tau_i}}[\LR[\tau_i,\tau'_i]\text{ does not separate }\partial B(0,r_{i})\text{ and }\partial B(0,r_{i-1})]\\
\le& (1-\alpha)^{[s]}.
\end{align*}
This implies that
\[\PP[|\gamma^\delta_0(t)|>e^{-t+s}]\le \PP[d(\cup_{i=1}^{n-1}\gamma^\delta_i,0)<e^{-t}]+\PP[\{d(\cup_{i=1}^{n-1}\gamma^\delta_i,0)>e^{-t}\}\cap\{|\gamma^\delta_0(t)|>e^{-t+s}\}]\le Ce^{-cs}.\]
From the convergence of $\gamma_0^\delta$ to $\gamma_0$, after parameterizing by capacity, $\gamma_0^\delta$ converges to $\gamma_0$ on every compact interval of $[-1,\infty)$. This completes the proof.
\end{proof}
Now, we will prove~\eqref{eqn::convexp}, which is the finally ingredient we need to show Theorem~\ref{thm::windconv}.
\begin{proof}
Conditional on $\{\gamma_1,\ldots,\gamma_{n-1}\}$, we denote by $\Omega_z$ the connected component of $\Omega\setminus (\cup_{i=1}^{n-1}\gamma_i)$ which contains $z$ and suppose $(x_{2j+1}x_{2j+2})\subset\partial \Omega_z$. Choose the conformal map $\phi$ from $\Omega_z$ onto the unit disk $\U$ such that $\phi(z)=0$ and $\Re\phi(x_{2j+1})=\Re\phi(x_{2j+2})$. We define $h_t^{\U}$ similarly as $h_t$ with respect to the basis point $\phi(x_{2j+2})$ for $\phi(\tilde\gamma_z)$. By~\cite[Lemma 2.4]{BLRdimers}, on the event 
\[\LA_t:=\{d(\cup_{i=1}^{n-1}\gamma_i,0)>e^{-t/4}\}\cap\{d(\gamma_0(t),0)<e^{-t/2}\}\] 
there exists a constant $C>0$ such that
\[\left|\E[h_{t}^{\U}(0)]-\left(\E[h_{t}(z)-\text{the winding of $(x_2 x_{2j+2})$}|\gamma_1,\ldots,\gamma_{n-1}]+\arg\phi'(0)\right)\right|\le Ce^{-t/4}.\]
Recall that $\phi(\tilde\gamma_z)$ is the flow line of the GFF with the boundary condition $u_{(\U;\phi(x_{2j+2}),\phi(x_{2j+1}))}$. The law of $\phi(\tilde\gamma_z)$ is invariant after reflecting with respect to $\R$. Thus, we have
\begin{align*}
\E[h_{t}^{\U}(0)]=\pi +(2\pi-\text{the winding of }\phi(x_{2j+1}x_{2j+2}))=\frac{1}{\chi}u_{(\U;\phi(x_{2j+2}),\phi(x_{2j+1}))}(0).
\end{align*}
By definition of $u_{(\U;\phi(x_{2j+2}),\phi(x_{2j+1}))}$, we have
\[u_{(\U;\phi(x_{2j+2}),\phi(x_{2j+1}))}\circ\phi=u_{(\Omega_z;x_{2j+2},x_{2j+1})}+\chi(\arg \phi'-\text{the winding of $(x_2 x_{2j+2})$}).\]
This implies that on the event $\LA_t$,
\begin{equation}\label{eq::au1}
\left|\E[h_{t}(z)|\gamma_1,\ldots,\gamma_{n-1}]-\frac{1}{\chi}u_{(\Omega_z;x_{2j+2},x_{2j+1})}(z)\right|\le Ce^{-t/4}.
\end{equation}
By Theorem~\ref{thm::flowcoupling}, we deduce
\begin{equation}\label{eq::au2}
\E[u_{(\Omega_z;x_{2j+2},x_{2j+1})}(z)]=u_{\Omega}(z;x_1,\ldots,x_{2n};z_1,\ldots,z_{n-1}).
\end{equation}
Combining with Lemma~\ref{lem::hitball} and Lemma~\ref{lem::windanddistance}, by Cauchy's inequality, there exists $C>0$ and $c>0$ such that 
\begin{equation}\label{eq::au3}
\left|\E\left[\left(\E[h_{t}(z)|\gamma_1,\ldots,\gamma_{n-1}]-\frac{1}{\chi}u_{(\Omega_z;x_{2j+2},x_{2j+1})}(z)\right)\one_{\LA_t^c}\right]\right|\le Ce^{-ct}t.
\end{equation}
Combining with~\eqref{eq::au1},~\eqref{eq::au2} and~\eqref{eq::au3}, we therefore complete the proof of \eqref{eqn::convexp}, and, in turn of Theorem \ref{thm::windconv}.
\end{proof}

{\small \bibliographystyle{alpha}
\bibliography{bibliography}}

\begin{thebibliography}{KPW00}

\bibitem[BHS22]{OffcriticalDimers}
Nathana{\"e}l Berestycki and Levi Haunschmid-Sibitz.
\newblock Near-critical dimers and massive {SLE}.
\newblock {\em arXiv preprint arXiv:2203.15717}, 2022.

\bibitem[BLR16]{BLRtgraph}
Nathana{\"e}l Berestycki, Beno\^it Laslier, and Gourab Ray.
\newblock A note on dimers and {T}-graphs.
\newblock {\em arXiv:1610.07994}, 2016.

\bibitem[BLR19]{BLR_Riemann1}
Nathanaël Berestycki, Benoit Laslier, and Gourab Ray.
\newblock Dimers on {R}iemann surfaces, {I}: {T}emperleyan forests.
\newblock 2019.
\newblock arXiv:1908.00832.

\bibitem[BLR20]{BLRdimers}
Nathana{\"e}l Berestycki, Beno\^it Laslier, and Gourab Ray.
\newblock Dimers and imaginary geometry.
\newblock {\em The Annals of Probability}, 48(1):1--52, 2020.

\bibitem[BLR22]{BLR_Riemann2}
Nathana{\"e}l Berestycki, Benoit Laslier, and Gourab Ray.
\newblock Dimers on {R}iemann surfaces {II}: conformal invariance and scaling
  limits.
\newblock arXiv:2207.09875, 2022.

\bibitem[BPW21]{BeffaraPeltolaWuUniqueness}
Vincent Beffara, Eveliina Peltola, and Hao Wu.
\newblock On the uniqueness of global multiple {SLE}s.
\newblock {\em Ann. Probab.}, 49(1):400--434, 2021.

\bibitem[Che16]{ChelkakRobustComplexAnalysis}
Dmitry Chelkak.
\newblock Robust discrete complex analysis: A toolbox.
\newblock {\em The Annals of Probability}, 44(1):628--683, 2016.

\bibitem[CW21]{ChelkakWanMassiveLERW}
Dmitry Chelkak and Yijun Wan.
\newblock {On the convergence of massive loop-erased random walks to massive
  SLE(2) curves}.
\newblock {\em Electron. J. Probab.}, 26:Paper No. 54, 2021.

\bibitem[Fom01]{MR1837248}
Sergey Fomin.
\newblock Loop-erased walks and total positivity.
\newblock {\em Trans. Amer. Math. Soc.}, 353(9):3563--3583, 2001.

\bibitem[Gor21]{Gorin_book}
Vadim Gorin.
\newblock {\em Lectures on random lozenge tilings}, volume 193.
\newblock Cambridge University Press, 2021.

\bibitem[HLW20]{HanLiuWuUST}
Yong Han, Mingchang Liu, and Hao Wu.
\newblock Hypergeometric {SLE} with $\kappa=8$: Convergence of {UST} and {LERW}
  in topological rectangles, 2020.

\bibitem[Kar19]{KarrilaMultipleSLELocalGlobal}
Alex Karrila.
\newblock Multiple {SLE} type scaling limits: from local to global, 2019.
\newblock arXiv:1903.10354.

\bibitem[Kar20]{KarrilaUSTBranches}
Alex Karrila.
\newblock U{ST} branches, martingales, and multiple {$\rm SLE(2)$}.
\newblock {\em Electron. J. Probab.}, 25:Paper No. 83, 2020.

\bibitem[Kas61]{Kasteleyn}
P.~W. Kasteleyn.
\newblock The statistics of dimers on a lattice. {I.} {T}he number of dimer
  arrangements on a quadratic lattice.
\newblock {\em Physica}, 27:1209--1225, 1961.

\bibitem[Ken00]{Kenyon_ci}
Richard Kenyon.
\newblock Conformal invariance of domino tiling.
\newblock {\em Ann. Probab.}, 28(2):759--795, 2000.

\bibitem[Ken01]{Kenyon_GFF}
Richard Kenyon.
\newblock Dominos and the {G}aussian free field.
\newblock {\em Ann. Probab.}, 29(3):1128--1137, 2001.

\bibitem[Ken09]{KeyonDimerLecture}
Richard Kenyon.
\newblock Lectures on dimers.
\newblock {\em arXiv:0910.3129}, 2009.

\bibitem[KKP20]{KarrilaKytolaPeltolaCorrelationsLERWUST}
Alex Karrila, Kalle Kyt\"{o}l\"{a}, and Eveliina Peltola.
\newblock Boundary correlations in planar {LERW} and {UST}.
\newblock {\em Comm. Math. Phys.}, 376(3):2065--2145, 2020.

\bibitem[KPW00]{KPW}
Richard~W Kenyon, James~G Propp, and David~B Wilson.
\newblock Trees and matchings.
\newblock {\em Electron. J. Combin.}, 2000.

\bibitem[KS04]{dimer_tree}
Richard~W. Kenyon and Scott Sheffield.
\newblock Dimers, tilings and trees.
\newblock {\em J. Combin. Theory Ser. B}, 92(2):295--317, 2004.

\bibitem[Liu23]{Liu_reverse}
Mingchang Liu.
\newblock In prepartion.
\newblock 2023.

\bibitem[LL10]{LawlerLimic}
Gregory~F Lawler and Vlada Limic.
\newblock {\em Random walk: a modern introduction}, volume 123.
\newblock Cambridge University Press, 2010.

\bibitem[LPW21]{LiuPeltolaWuUST}
Mingchang Liu, Eveliina Peltola, and Hao Wu.
\newblock Uniform spanning tree in topological polygon, partition functions for
  {SLE}(8), and correlations in $c=-2$ logarithm {CFT}, 2021.

\bibitem[LSW04]{LawlerSchrammWernerLERWUST}
Gregory~F. Lawler, Oded Schramm, and Wendelin Werner.
\newblock Conformal invariance of planar loop-erased random walks and uniform
  spanning trees.
\newblock {\em Ann. Probab.}, 32(1B):939--995, 2004.

\bibitem[MS16]{MillerSheffieldIG1}
Jason Miller and Scott Sheffield.
\newblock Imaginary geometry {I}: {I}nteracting {SLE}s.
\newblock {\em Probab. Theory Related Fields}, 164(3-4):553--705, 2016.

\bibitem[MS17]{MillerSheffieldIG4}
Jason Miller and Scott Sheffield.
\newblock Imaginary geometry {IV}: interior rays, whole-plane reversibility,
  and space-filling trees.
\newblock {\em Probab. Theory Related Fields}, 169(3-4):729--869, 2017.

\bibitem[Pet15]{Petrov}
Leonid Petrov.
\newblock Asymptotics of uniformly random lozenge tilings of polygons. gaussian
  free field.
\newblock {\em The Annals of Probability}, 43(1):1--43, 2015.

\bibitem[Qia18]{WeiTrichordal}
Wei Qian.
\newblock Conformal restriction: the trichordal case.
\newblock {\em Probability Theory and Related Fields}, 171(3):709--774, 2018.

\bibitem[Rus18]{RusskikhDimers}
Marianna Russkikh.
\newblock Dimers in piecewise temperley domains.
\newblock {\em Communications in Mathematical Physics}, 359, 04 2018.

\bibitem[Sch00]{SchrammFirstSLE}
Oded Schramm.
\newblock Scaling limits of loop-erased random walks and uniform spanning
  trees.
\newblock {\em Israel J. Math.}, 118:221--288, 2000.

\bibitem[TF61]{TemperleyFIsher}
Harold Neville~Vazeille Temperley and Michael~E Fisher.
\newblock Dimer problem in statistical mechanics--an exact result.
\newblock {\em Philosophical Magazine}, 6(68):1061--1063, 1961.

\bibitem[Wu20]{WuHyperSLE}
Hao Wu.
\newblock Hypergeometric {SLE}: conformal {M}arkov characterization and
  applications.
\newblock {\em Comm. Math. Phys.}, 374(2):433--484, 2020.

\bibitem[Zha10]{Zhan}
Dapeng Zhan.
\newblock Reversibility of some chordal {SLE} ($\kappa$; $\rho$) traces.
\newblock {\em Journal of Statistical Physics}, 139(6):1013--1032, 2010.

\end{thebibliography}
\end{document}